\documentclass[12pt]{amsart}   % LaTeX with AMS style; 12 point for old eyes
\usepackage{listings}
\usepackage{amsmath}
\usepackage{amssymb,amsfonts}   % For better support of math
\usepackage{graphicx}	        % Enable for eps figures, if they occur
\newtheorem*{theorem*}{Theorem}
\newtheorem{theorem}{Theorem}[section]
\newtheorem{proposition}[theorem]{Proposition}
\newtheorem{lemma}[theorem]{Lemma}
\newtheorem{definition}[theorem]{Definition}

\newtheorem{corollary}[theorem]{Corollary}

\numberwithin{equation}{section}
\usepackage{enumitem}
\usepackage[top=1.00in, bottom=.8in, left=.9in, right=.9in]{geometry} 
\usepackage{amsmath,graphics,amsfonts,amsthm,amssymb,color}
\usepackage{graphics}
\usepackage{latexsym}
\usepackage{multicol}

\title{Invariant Measure Construction at a Fixed Mass}
\author{
        Justin T. Brereton \\}
\date{\today}
\begin{document}

\begin{abstract}
In this paper we analyze the derivative nonlinear Schr\"odinger equation on $\mathbb{T}$ with randomized initial data in $\cap_{s < \frac{1}{2}} H^{s}(\mathbb{T})$ chosen according to a Wiener measure. We construct an invariant measure at each sufficiently small, fixed mass $m$ through an argument that emulates the divergence theorem in infinitely many dimensions. We also prove that the density function needed to construct the Wiener measure is in $L^p$, even after scaling of the Fourier coefficients of the initial data. 
\end{abstract}

\maketitle

\section{Introduction}  \label{mathrefs}

Consider the initial value problem for the derivative non-linear Schr\"odinger equation (DNLS) on the torus $\mathbb{T} = \mathbb{R}/2\pi \mathbb{Z}:$

\begin{equation} \label{1}
\begin{cases}
\begin{split}
u(t,x) &: \mathbb{R} \times \mathbb{T} \rightarrow \mathbb{C}\\ 
iu_t + \Delta u &= i \partial_x(|u|^2u)\\
u(0)&= u_0 . \\
\end{split} \end{cases}
\end{equation}

\let\thefootnote\relax\footnote{This research was supported in part by the Simons Foundation}

The equation has been used as a model for Alfv\~en waves in a magnetized plasma \cite{Nahmod}. It is a completely integrable equation and the main conserved quantities are the mass $m(u)$, momentum/Hamiltonian $P(u)$, and energy $E(u)$ given by
\begin{equation}\begin{cases}
\begin{split} 
m(u(t)) = m(u_0) &= \int_{\mathbb{T}} |u(x)|^2 dx\\
P(u(t)) =P(u_0) &= \int_{\mathbb{T}} \frac{1}{2}|u|^4 + i\overline{u}\partial_{x}u dx\\
E(u(t)) = E(u_0) &=\int_{\mathbb{T}} |\partial_x  u|^2 +  \frac{3}{4}i\overline{u}^2\partial_{x}(u^2)+\frac{1}{2}|u|^6  dx.\\ 
\end{split} \end{cases}
\end{equation} 
Solutions to the real line equivalent of \eqref{1} satisfy the dilation symmetry  $u(t,x) \rightarrow u_{\alpha}(t,x) =  \alpha^{1/2}u(\alpha^2t, \alpha)$. This symmetry preserves the $L^2(\mathbb{R})$ norm of a solution, making this equation mass-critical. One might hope for $L^2(\mathbb{T})$ well-posedness, however, the derivative loss on the RHS of equation \eqref{1} means that it is very difficult to bound the Duhamel term when $u$ has low regularity. 

Global well-posedness is known for $H^{s}(\mathbb{T}), s > \frac{1}{2}$ \cite{Win}.  Gr\"unrock and Herr proved only local well-posedness in the Fourier-Lebesgue spaces $\widehat{H}^s_r(\mathbb{T})$ for $s \ge \frac{1}{2}, r \in (\frac{4}{3}, 2)$ \cite{Grunrock}. These spaces scale like $H^{s}(\mathbb{T})$  for $s > \frac{1}{4}$ \cite{Grunrock}.  A proof of well-posedness anywhere near $L^2(\mathbb{T})$ still seems unlikely, so research has turned to investigating almost sure well-posedness for randomized data. Failure of well-posedness could be due to certain exceptional initial data, and it still may be true that 'almost all' initial data lead to a solution. We first construct a measure $\rho$ that is invariant with respect to our PDE, in this case the DNLS equation. We then aim to prove that the set of data that fails to produce a solution on $\mathbb{R}$ has measure zero with respect to $\rho$. This is referred to as 'almost sure' well-posedness. 

The study of almost sure well-posedness began with the formative paper of  Lebowitz, Rose and Speer on the NLS equation \cite{Statistical}, and continued with Bourgain for the NLS and other equations in the 1990s \cite{B1} \cite{B2}, \cite{B3}. 
Several authors have studied the DNLS equation with random initial data taken from a Wiener measure based on the energy $E(u)$ of \eqref{1} \cite{Burq}, \cite{Nahmod}, \cite{TT}. Since the energy and mass are conserved by the flow of \eqref{1}, Liouville's Theorem implies that a measure defined based on these quantities is invariant by the flow of \eqref{1}. Such a Wiener measure yields a randomized initial data $u(0)$ that is almost surely in $H^{s}(\mathbb{T})$ for all $s < \frac{1}{2}$. 

It is natural to investigate the ergodicity of the measure, and whether the set of initial data can be decomposed into $\rho$-invariant subsets. This question was posed by Lebowitz-Rose-Speer in \cite{Statistical}, and in \cite{Oh} Oh and Quastel constructed an invariant measure conditioned on mass $\int_{\mathbb{T}} |u|^2 dx$  and momentum $\int_{\mathbb{T}} i u\overline{u_x} dx$ for the NLS equation, with conserved energy $\int_{\mathbb{T}} \frac{1}{2}|\partial_x u|^2 \pm \frac{1}{p}|u|^p dx$.

The $|u|^4$ term in the momentum of the DNLS equation is too high an exponent for the method of construction in \cite{Oh}, so we instead construct a measure conditioned only on mass $m$.  

\subsection{The initial measure $\mu$}
We start by defining the proper space of functions where our measure will be supported. We will adopt the convention that for each $f \in L^1(\mathbb{T})$, \newline  $\int_{\mathbb{T}} f(x)  dx = \int_0^{2\pi} \frac{1}{2\pi} f(x) dx$ and therefore $\|f\|_{L^p(\mathbb{T})}^p  = \int_{\mathbb{T}} |f(x)|^p dx$. 
\begin{definition}
Define \begin{equation}
H^{1/2-}(\mathbb{T}) = \cap_{s < \frac{1}{2}} H^{s}(\mathbb{T}).
\end{equation}
\end{definition}

Consider a probability space $(\Omega, \mathcal{F}, \mathbb{P})$ and a sequence $\{g_n\}$ of i.i.d. complex Gaussian random variables on $\Omega$ with mean $0$ and variance $1$. This means that $g_n$ can be written as a sum of its real and imaginary components $g_n = a_n + ib_n$, and both $a_n$ and $b_n$ are $N(0,\frac{1}{2})$ distributed, so $\mathbb{E}(|g_n|^2) = 1$. We select our initial data $u(x)$ by the randomization 
\begin{equation} \label{8}
u(x) = \phi(\omega,x) = \sum_{n \in \mathbb{Z}} \frac{g_n(\omega)}{\langle n \rangle}e^{inx}.
\end{equation} 
Observe that 
\begin{equation}
\|u\|^2_{H^s(\mathbb{T})} = \sum_{n \in \mathbb{Z}} \frac{\langle n \rangle ^{2s}|g_n|^2 }{\langle n \rangle ^2},
\end{equation} which converges almost surely for $s < \frac{1}{2}$. Thus the distribution of $u(x)$ induces a measure $\mu = \mathbb{P} \circ \phi^{-1}$ on $H^{1/2-}(\mathbb{T})$. The measure $\mu$ also satisfies the equation 
\begin{equation}
d\mu  = \lim_{N \rightarrow \infty} d_Ne^{-\int_{\mathbb{T}} |P_{\le N}u|^2 + |P_{\le N}\partial_x u|^2 dx} \Pi_{|n| \le N} d\widehat{u(n)},
\end{equation} for a sequence of normalizing constants $d_N$, the limit of which contains the mass and part of the energy in the exponential. 

We use this base measure $\mu$ to construct a sequence $\rho_N$ of measures that converge to $\rho$ as in section 1.4 of \cite{Burq}. These $\rho_N$ are defined to contain the exponential of the rest of the energy at the $N$ frequency, so that the limit measure $\rho$ effectively contains $e^{-E(u) - m(u)}$, and should be invariant with respect to \eqref{1}.

\begin{definition}

For any non-negative integer $N$, define the function  

\begin{equation}
\begin{split}
f_N(u) &= \frac{-3i}{4} \cdot \int_{\mathbb{T}} \overline{P_{\le N}u}^2 \cdot  \partial_x(P_{\le N}u)^2  dx \\
&= \frac{-3i}{2} \cdot \int_{\mathbb{T}} \overline{P_{\le N}u}^2 \cdot  P_{\le N}u \cdot \partial_x(P_{\le N}u)  dx, \\
\end{split}
\end{equation}
which equals the negative of the middle term of the DNLS energy. Indeed, we have  
\begin{equation*}
f_N(u) - \frac{1}{2}\int_{\mathbb{T}} |P_{\le N}u(x)|^6dx  - \int_{\mathbb{T}} |\partial_x  P_{\le N} u|^2 dx = -E(P_{\le N} u).
\end{equation*}
\end{definition}

While $\|e^{f_N(u) - \frac{1}{2}\int_{\mathbb{T}} |P_{\le N}u(x)|^6dx}\|_{L^1(H^{1/2-}(\mathbb{T}), d\mu)}$ is likely unbounded, if we multiply the exponential by a cutoff function that restricts to $u$ with small mass, we can obtain a density  in $L^1(\mu)$, which gives us a well-defined measure. We define the density of our measure to be
\begin{equation} \label{7}
\Psi_N(u) = 1_{\|P_{\le N} u \|_{L^2(\mathbb{T})}^2 < m}  e^{f_N(u) - \frac{1}{2}\int |P_{\le N}u(x)|^6dx},
\end{equation}
where $1_{S}$ denotes the indicator function of a set $S$, and $m$ is a constant maximum value of the mass that is chosen to be sufficiently small. Finally, the measure $\rho_N$ is defined as 
\begin{equation}
d\rho_{N}  = \beta_N\Psi_N(u) d\mu,
\end{equation}
 where $\beta_N = \|\Psi_N\|^{-1}_{L^1( \mu)}$ is chosen so that this is a probability measure with total measure $1$. Of course, this definition only makes sense if $\Psi_N$ is indeed in $L^1(\mu)$. It is sufficient to prove that the functions $\Psi_N(u)$ are contained in $L^p(\mu)$ and that they converge in measure to a function $\Psi(u)$ in $L^p(\mu)$ for all $p \ge 1$. One must prove that the functions $\Psi_N(u)$ satisfy 
\begin{equation} \label{9}
\|\Psi_N(u)\|_{L^p(\mu)} \le C(p)
\end{equation}
for all $p,N$. A special case of Proposition \ref{24} proves that $\{f_N, N \ge 0\}$ is a Cauchy sequence in $L^2(\mu)$, and this combined with \eqref{9} implies the convergence in $\mu$-measure of $\Psi_N(u)$ to $\Psi(u)$. This is identical to the argument in \cite{TT}, which presents \eqref{9} as Proposition 4.2. Though their general argument is sound the proof contains some error, specifically, the exponent in their Proposition 3.1 is incorrect. This error was brought to the attention of the authors, and Tzvetkov was able to quickly produce a correct proof. In Proposition \ref{24} of this paper we will prove a more general result that allows for the coefficients of $u$ to be scaled.

In \cite{TT} they use \eqref{9} to complete the construction of the invariant measure $\rho$ defined by 
 \begin{equation}
d\rho  = \beta\Psi(u) d\mu,
\end{equation} and then prove that $\rho_N$ is invariant with respect to a transformed version of equation \eqref{1} that is then truncated to finitely many dimensions. In addition, Burq, Thomann and Tzvetkov applied a compactness argument to these measures, implying the existence of a subsequence $\rho_{{N_k}}$ such that the solutions taken from $\rho_{N_k}$ converge to global solutions of \eqref{1} \cite{Burq}. This proves almost sure global well-posedness of \eqref{1} for initial data taken from $\rho$.   
So all that remains to complete their argument  is a correct proof of \eqref{9}.

Nahmod, Oh, Rey-Bellet and Staffilani give an alternative proof of almost sure global well-posedness in the Fourier-Lebesgue spaces by performing a gauge transformation on \eqref{1} in \cite{Nahmod}. They constructed a measure that is invariant with respect to the gauged equation.  Their proof references the incorrect Propositions 3.1  and 4.2 of \cite{TT}, however, their proof is still valid when relying on the weaker bounds presented in this paper. Then in \cite{Nahmod2},  Nahmod, Rey-Bellet, Sheffield and Staffilani showed that this gauged measure is absolutely continuous with respect to the measure studied in \cite{TT}, \cite{Burq}, and this paper.

More recently, Genovese, Luca and Valeri constructed invariant measures for the DNLS equation using conservation laws at the $H^k(\mathbb{T})$ level for $k \ge 2$ \cite{GG}. Both their proof and this proof of boundedness of the measure revolve around demonstrating that when the mass $m$ is small, we have the bound 
\begin{equation}
\mathbb{P}(e^{f_N(u)}> e^{\lambda}) \lesssim e^{-C\lambda/m^{\theta}}
\end{equation} for large $N, \lambda$ and some positive $\theta >0$. This implies that for sufficiently small $m$, the expected value of $1_{\|P_{\le N} u\|_{L^2}^2 <m_{p}} e^{f_N(u)}$ is bounded. The proof in \cite{GG} relies on the bound on \newline $\|\partial_x^k (P_{\le N} u) \cdot \overline{P_{\le N} u}\|_{L^{\infty}}$ for $k \ge 2$ in their Lemma 5.6. A bound for $k=1$ would be applicable to the $H^1(\mathbb{T})$ level energy studied in this paper, however, their proof of Lemma 5.6 fails when $k=1$ due to a logarithmic divergence. We present a different proof that utilizes dyadic decompositions. 

\begin{theorem} \label{5}
For each $p \ge 1$ there exists a constant $m_{p} > 0$ such that for all dyadic $N$ we have
\begin{equation} \label{6}
\begin{split} 
\| 1_{\|P_{\le N} u\|_{L^2(\mathbb{T})}^2 <m_{p}} e^{f_N(u)} \|^p_{L^p(H^{1/2-}(\mathbb{T}), d\mu)}   &\le C(p),\\ 
\end{split}
\end{equation}
which implies that 
\begin{equation} \label{7}
\begin{split} 
\| \Psi_N(u) \|^p_{L^p(H^{1/2-}(\mathbb{T}),d\mu)}  &\le C(p).\\ 
\end{split}
\end{equation}
\end{theorem}

We prove Theorem \ref{5} by proving the more general Theorem \ref{401} that allows for finite scaling of the Fourier coefficients of $u$. This completes the construction of the invariant measure $\rho$. We then use this base measure to construct a set of measures $\rho_m$, each supported on the set of functions with mass $m$, as in \cite{Oh}. We also present a new method of constructing such measures. This method can also been applied to the Benjamin-Ono equation, which can be found in the author's dissertation. The main result of this paper is the following theorem:

\begin{theorem} \label{17}
There exists a small constant $m'$ such that for each $m < m'$ there exists a measure $\rho_m$ supported on the set $H^{1/2-}({\mathbb{T}}) \cap \{ \|u\|_{L^2}^2 = m\}$.  There also exists a subset $\Sigma \subset H^{1/2-}(\mathbb{T}) \cap \{ \|u\|_{L^2}^2 = m\}$ of full $\rho_m$ measure such that each $u_0 \in \Sigma$ produces a global solution $u$ to \eqref{1}. The measure $\rho_m$ is invariant, meaning the random variable $u(t)$ has distribution $\rho_m$ for all $t \in \mathbb{R}$.
\end{theorem}

Theorem \ref{17} applies to the space $H^{1/2-}(\mathbb{T}) = \cap_{\sigma < \frac{1}{2}} H^{\sigma}(\mathbb{T})$, however, it is not always practical to do analysis in this space. Though $H^{1/2-}(\mathbb{T})$ inherits the intersection topology, it is not a normed space. On occasion we will  prove results in $H^{\sigma}(\mathbb{T})$ for all $\sigma = \frac{1}{2}-$, and infer that each result holds on the intersection.

Section 2 gives background on the fixed mass problem and how we expand on the methods developed in \cite{Oh}. Then in Section 3 we prove the necessary bounds on $\mu(|f_N-f_M| > \lambda)$ to complete the construction of $\rho$ in Section 4. 

The rest of the paper is devoted to proving invariance of the fixed mass. In Section 5 we define the measure $\mu_m$ for each sufficiently small $m$, and decompose it as a sum of measures $\nu_m^k$, while proving various properties of these measures. Then we use these bounds to construct the measure $\rho_m$ in Section 6. We finish the paper by constructing global solutions with $\rho_m$-distributed initial data and proving Theorem \ref{17} in Section 7.  

\subsection{Notation}
Given a function $u$ on the torus, the Fourier series of $u$ is $u(x) = \sum_{n \in \mathbb{Z}} \widehat{u(n)}e^{inx}$, which defines each Fourier coefficient $\widehat{u(n)}$. We let $u_N = P_{\le N}u = \sum_{|n| \le N} \widehat{u(n)}e^{inx}$. Similarly let $ P_{< N}u = \sum_{|n| < N} \widehat{u(n)}e^{inx}$.

In this paper a dyadic integer denotes an integer power of $2$. Given a dyadic integer $N$, let $n \sim N$ refer to the set of integers $n$ that satisfy $|n| \in (\frac{N}{2}, N]$. For dyadic $N$, $P_N u$ denotes $\sum_{n \sim N} e^{inx}\widehat{u(n)}$. We could equivalently define $P_N u  = P_{\le N}u - P_{\le \frac{N}{2}}u$. 

Given two dyadic integers $N,M$ we write $N \sim M$ if $N,M$ are within a factor of $8$ of each other. We will use this notation when noting that for any 4 dyadic integers $N_1 \ge N_2 \ge N_3 \ge N_4$, and 4 functions $u_1, u_2, u_3, u_4$, we have 
$\int_{\mathbb{T}} P_{N_1}u_1P_{N_2}u_2P_{N_3}u_3P_{N_4}u_4 dx = 0$ unless $N_1 \sim N_2$. 

We will use several constants throught this proof. Constants with a lowercase $c$: $c_0, c_1, c_2$, etc. will refer to specific numbers, usually small enough to satisfy a certain bound, whereas a capital $C$ refers to a generic constant, and $C(m), C(m,k)$ etc.. refer to constants that depend upon the variables inside the parenthesis.

\section{An overview of the fixed mass argument}
We start with a brief overview of the argument in \cite{Oh} and give a description of our method of constructing a fixed mass measure. In \cite{Oh}, Oh and Quastel constructed an invariant measure of the NLS equation for each mass $a >0$ and momentum $b \in  \mathbb{R}$. Given $a,b$ and $\epsilon >0$, they defined a measure $\mu^{a,b}_{\epsilon}$ supported on the set of functions in $H^{1/2-}_{\mathbb{T}}$ with mass in $(a-\epsilon,a+\epsilon)$ and momentum in $(b-\epsilon, b+\epsilon)$. They then took the limit of this sequence of measures as $\epsilon \rightarrow 0$ to define a measure $\mu^{a,b}$. Multiplying in  a factor of $e^{\pm \frac{1}{p}\int_{\mathbb{T}} |u|^p dx}$, the remaining term of the NLS energy, they defined the measure $\rho^{a,b}_{\epsilon}$ as 
\begin{equation} \label{50}
\rho^{a,b}_{\epsilon}(E) = Z^{a,b}_{\epsilon} \int_{E} e^{\pm \frac{1}{p}\int_{\mathbb{T}} |u|^p dx} d\mu^{a,b}_{\epsilon}.
\end{equation}

Each measure $\rho^{a,b}_{\epsilon}$ is invariant under the NLS, and taking the limit as $\epsilon \rightarrow 0$ defines the invariant measure $\rho^{a,b}$ that is supported on the set of functions with mass $a$ and momentum $b$. In order to show that this limit exists and is finite, Oh and Quastel proved that for $a,b$ the integral $\int e^{\pm \frac{1}{p}\int_{\mathbb{T}} |u|^p dx} d\mu^{a,b}_{\epsilon}$ is bounded uniformly in $\epsilon$. In the defocusing case, the $e^{ \frac{-1}{p}\int_{\mathbb{T}} |u|^p dx}$ term is trivially bounded by $1$.  For the focusing case they achieve this bound through a direct analysis of $\int_{\mathbb{T}} |u|^p dx$, by using Sobolev embedding, dyadic decomposition, etc.. For more complicated integrals this approach will be insufficient.

In this paper we develop a more general way of bounding $\|F\|_{L^1(\mu_m)}$ for an arbitrary function $F$, where $\mu_m$ is the analogue of the measure $\mu^{a,b}$ that is supported on functions with mass $m$. The method is motivated by the divergence theorem. The set of functions with mass $m$ is the set of functions $u$ that satisfy $\sum |\widehat{u(n)}|^2=m$. This is essentially an ellipsoid in infinitely many dimensions, and ideally one would like to calculate the integral of $F$ over this ellipsoid by relating it to partial derivatives of $F$ on the interior of the ellipsoid, in a manner similar to the divergence theorem. 

We do this through a scaling argument that can be used as a substitute for the divergence theorem in situations, such as infinitely many dimensions, where divergence and Stokes' theorems do not apply. We will use $\mathbb{S}^{1}$, the unit circle, as an example. This is the set of $(x, y)$ satisfying $x^2 +  y^2  = 1$. If we take $\sigma$ to be $\frac{1}{2\pi}$ times the surface measure $S$ on $\mathbb{S}^1$ (so $\sigma$ is a probability measure), and apply the divergence theorem we have 
\begin{equation} \label{52}
\begin{split}
\int_{\mathbb{S}^{1}} f(x,y)d\sigma(x,y) &= \int_{\mathbb{S}^{1}} f(x,y)( x,y ) \cdot ( x,y ) d\sigma  \\
&= \int_{\mathbb{S}^{1}} (  xf(x,y),yf(x,y)  ) \cdot ( x,y ) \frac{1}{2\pi}dS  \\
\end{split}
\end{equation}

\begin{equation*}
\begin{split}
\int_{\mathbb{S}^{1}} f(x,y)d\sigma(x,y) &=\frac{1}{2\pi} \int_{x^2 + y^2 \le 1} \text{Div}( xf(x,y), yf(x,y) ) dA\\
&= \frac{1}{2\pi} \int_{x^2 + y^2 \le 1} 2f(x,y)+ xf_{x}(x,y) + yf_{y}(x,y) dA.\\
\end{split}
\end{equation*}

If we instead write $\sigma$ as the limit of area integrals over the region between the circle of radius $s^2$ and $\frac{1}{s^2}$ as $s \rightarrow 1+$ we can apply change of variables and obtain

\begin{equation} \label{94}
\begin{split}
\int_{\mathbb{S}^{1}} f(x,y)d\sigma(x,y) &= \lim_{s \rightarrow 1+} \frac{1}{\pi(s^2 - 1/s^2)}\int_{\frac{1}{s^2} \le x^2 + y^2 \le s^2} 
f(x,y) dA\\
&= \lim_{s \rightarrow 1+} \frac{1}{\pi(s^2 - 1/s^2)} \left[  \int_{x^2 + y^2 \le s^2} 
f(x,y) dA  - \int_{x^2 + y^2 \le 1/s^2} 
f(x,y) dA             \right] \\
&= \lim_{s \rightarrow 1+} \frac{1}{\pi(s^2 - 1/s^2)} \left[  \int_{x^2 + y^2 \le 1} s^2f(sx,sy) - \frac{1}{s^2}f\left( \frac{x}{s},\frac{y}{s} \right) dA        \right] \\
&= \frac{1}{\pi} \frac{d}{dr}|_{r=1}  \int_{x^2 + y^2 \le 1} rf(\sqrt{r}x,\sqrt{r}y) dA\\
&= \frac{1}{2\pi} \int_{x^2 + y^2 \le 1}2f(x,y) + xf_{x}(x,y) + yf_{y}(x,y) dA.\\
\end{split}
\end{equation}
This calculation confirms the divergence theorem and demonstrates that we can emulate the divergence theorem through scaling. In $n$ dimensions we get a factor of $s^n$ when we apply change of variables. If we attempted this scaling with infinitely many variables this factor would be infinite and the argument would not work. So to bound functions on the set of $g_n$ such that $\sum \frac{|g_n|^2}{\langle n \rangle^2} = m$, we want to scale the $g_n$ one at a time.

We can decompose the surface measure on the circle as a sum of probability measures generated by scaling by letting $\sigma = \frac{1}{2}(2x^2\sigma + 2y^2\sigma)$. Indeed, if we scale with respect to $x$ we have

\begin{equation} \label{54}
\begin{split}
\lim_{s \rightarrow 1+} \frac{1}{\pi(s - 1/s)}&\int_{\frac{x^2}{s^2} + y^2 \le 1 \le s^2x^2 + y^2} 
f(x,y) dA\\ 
&= \lim_{s \rightarrow 1+} \frac{1}{\pi(s - 1/s)} \left[  \int_{\frac{x^2}{s^2} + y^2 \le 1} 
f(x,y) dA  - \int_{s^2x^2 + y^2 \le 1} 
f(x,y) dA             \right] \\
&= \lim_{s \rightarrow 1+} \frac{1}{\pi(s - 1/s)} \left[  \int_{x^2 + y^2 \le 1} sf(sx,y) - \frac{1}{s}f\left( \frac{x}{s}, y \right) dA        \right] \\
&= \frac{1}{\pi} \frac{d}{dr}|_{r=1}  \int_{x^2 + y^2 \le 1} rf(rx,y) dA\\
&= \frac{1}{\pi} \int_{x^2 + y^2 \le 1} f(x,y) + xf_{x}(x,y) dA.\\
\end{split}
\end{equation}
This also equals
\begin{equation} \label{55}
\begin{split}
 \frac{1}{\pi} \int_{x^2 + y^2 \le 1} f(x,y) + xf_{x}(x,y) dA &=  \frac{1}{\pi} \int_{x^2 + y^2 \le 1} \text{Div} (  xf(x,y),0 )  dA \\
&= \frac{1}{2\pi} \int_{\mathbb{S}^1}  ( 2xf(x,y),0 ) \cdot ( x,y) dS \\
&= \int_{\mathbb{S}^1} f(x,y) d[2x^2\sigma(x,y)].\\
\end{split}
\end{equation} 

The same scaling argument with $y$ instead of $x$ constructs the measure $2y^2\sigma$, and taking the average of $2x^2\sigma$ and $2y^2\sigma$ is $\sigma$, since clearly $x^2+y^2 =1$ on the circle. 

This demonstrates that we can decompose the surface measure of a sphere in finite dimensions into components associated with scaling of each variable. We will similarly write the measure $\mu_m$  as the sum 
\begin{equation}
\mu_m = \frac{1}{m}\sum_{k}  |\widehat{u(k)}|^2\mu_m = \frac{1}{m}\sum_{k \ge 0} c_{k,m} \nu_{m}^k
\end{equation} for a sequence of probability measures $\nu_{m}^{k}$ derived from scaling.

We will construct these measures $\nu_m^k$ from $\mu$ in a manner analogous to the construction of $2x^2\sigma$ and $2y^2\sigma$ from the Lebesgue area measure. Since the base measure $\mu$ is Gaussian and not a Lebesgue measure there will be an extra Gaussian factor, but it is easily bounded above. 

\section{Probabilistic bounds on $f_N(u)$}
In this section we develop some tools necessary for the proof that the energy remainder term $F_N(u)$ is bounded in the $L^p(\mu)$ norm and for the more general Theorem \ref{401}, which is needed for the fixed mass argument.

\subsection{Bounds on sums of Gaussians} \label{mathrefs} 
We start with some probabilistic bounds on the size of sums involving $g_n$ and $|g_n|^2$. 

\begin{lemma} \label{12}
Suppose $\{g_n\}$ is a sequence of i.i.d. complex Gaussians with mean $0$ and variance $1$, and $N$ is a dyadic integer. For $\lambda > 4N\ln(2)$ we have
\begin{equation}
\mathbb{P}(\sum_{n \sim N} |g_n|^2 > \lambda) \le e^{-\lambda/4} .
\end{equation}
\end{lemma}

\begin{proof}

By Markov's inequality, we have
\begin{equation}
\mathbb{P}(\sum_{n \sim N} |g_n|^2 > \lambda) \le e^{-\lambda t} \mathbb{E}_{\mathbb{P}}(e^{t \sum_{n \sim N} |g_n|^2}).
\end{equation}

Now we set $t = \frac{1}{2}$ and evaluate the expected value:
\begin{equation*}
\begin{split}
\mathbb{P}(\sum_{n \sim N} |g_n|^2 > \lambda) &\le e^{-\lambda t}  \int_{\mathbb{C}^{N}}\frac{e^{(t-1)\sum_{n \sim N}|x_n|^2}}{\pi^{N}}  \Pi_{n \sim N} dx_n\\
&=  e^{-\lambda/2}  \int_{\mathbb{C}^{N}} \frac{e^{(-1/2)\sum_{n \sim N}|x_n|^2}}{\pi^{N}}  \Pi_{n \sim N} dx_n\\
&\le e^{-\lambda /2}\left( \int_{\mathbb{R}} \sqrt{2}e^{-x^2/2}dx/\sqrt{2\pi}  \right)^{2N}\\
&= e^{-\lambda /2}2^{N}\\
&= e^{-\lambda /2}e^{N\ln(2)}\\
&\le  e^{-\lambda /4}.\\
\end{split}
\end{equation*}

\end{proof}
This is the first step towards proving the next bound on the $L^2$ norm of the high frequencies of $u$.  The following lemma gives us a bound on the probability that $\|P_Nu\|_{L^2}^2 > \lambda$ . 
\begin{lemma} \label{13}
Suppose $u(x)$ is chosen according to \eqref{8} and the measure $\mu$, where each $g_n$ has variance $\le \mathcal{M}^2$.  For any dyadic $N$ and $\lambda > \frac{16\mathcal{M}^2\ln(2)}{N}$, we have $$\mu(\|P_Nu\|^2_{L^2} >\lambda) \le Ce^{-N^2\lambda/16\mathcal{M}^2}.$$
\end{lemma}
\begin{proof} By equation \eqref{8} we have the bound 
\begin{equation}
\|P_Nu\|_{L^2}^2 = \sum_{n \sim N} \frac{|g_{n}|^2}{n^2} \le \frac{4\mathcal{M}^2}{N^2}\sum_{n \sim N} \frac{|g_{n}|^2}{\mathbb{E}(|g_n|^2)}.
\end{equation}

Clearly $\frac{g_{n}}{\mathbb{E}(|g_n|^2)^{1/2}}$ is a complex Gaussian random variable with mean $0$ and variance $1$. The probability that this is $>\lambda =  4\lambda'\mathcal{M}^2/N^2$, for $ \lambda' > 4\ln(2)N$, is bounded by $e^{-\lambda'/4}$. Multiplying by $\frac{N^2}{4\mathcal{M}^2}$ and applying Lemma \ref{12} gives us the desired bound.
\end{proof}

\subsection{Wiener Chaos bound} \label{mathrefs} 
In this section we  prove a more general version of Proposition 3.2 of \cite{TT}. We will need this more general version of the result for the fixed mass problem, where we apply scaling to the coefficients $g_n$ and must allow each to be multiplied by a constant factor. Let $\overline{n} = (n_1,n_2,n_3,n_4) \in \mathbb{Z}^4$ denote a $4$-tuple of integers. We aim to prove the following:

\begin{theorem} \label{401}
Let $c(n) : \mathbb{Z} \rightarrow \mathbb{R}$ be a function that is bounded above by a constant $\mathcal{M}$ and satisfies $c(n) = c(-n)$, and let $p \ge 1$ be a positive power. For the fourth order chaos function \begin{equation} \label{451}
S_{4,N}(\omega) = 3/2 \sum_{n_1+n_2 = n_3+n_4, |n_i| \le N} c(n_1)c(n_2)c(n_3)c(n_4)
n_1\frac{g_{n_1}(\omega)g_{n_2}(\omega)\overline{g_{n_3}(\omega)}\overline{g_{n_4}(\omega)}}{\langle n_1 \rangle \langle n_2 \rangle \langle n_3 \rangle \langle n_4 \rangle}
\end{equation} 
and sufficiently small $m_{\mathcal{M},p}$ we have 
$$ \int_{H^{1/2-}(\mathbb{T})} 1_{\|u \|_{L^2}^2 < m_{\mathcal{M},p}} e^{pS_{4,N}(u)}d\mu \le C(\mathcal{M},p)$$
for some constant depending on $\mathcal{M},p$.
\end{theorem}
When $c(n)=1$ identically we have equation \eqref{6}, the necessary bound for the measures $\rho_N$ and $\rho$. So proving Theorem \ref{401} in turn proves Theorem \ref{5}.  Let $S_{4}(\omega)$ be the sum in equation \eqref{451} without the restriction  $|n_i| \le N$. We expect $S_{4,N}(\omega)$ to converge to $S_{4}(\omega)$.

To show $S_{4}(\omega) \in L^p(\Omega)$  we prove that $S_{4,N}(\omega)$ is a Cauchy sequence in $L^p(\Omega)$. We desire a bound on $\mathbb{P}(|S_{4,N}(\omega) - S_{4,M}(\omega)| > \lambda)$ for large $N,M$. Since $S_{4,N}(\omega) - S_{4,M}(\omega)$ is a sum of products of Gaussians we expect a lot of cancellation of positive and negative terms. It turns out a bound on  $\|S_{4,N}(\omega) - S_{4,M}(\omega) \|_{L^2(\Omega)}$ implies a bound in much higher $L^p(\Omega)$ norms, which in turn leads to a bound on $\mathbb{P}(|S_{4,N}(\omega) - S_{4,M}(\omega)| > \lambda)$.

The following proposition is a well known result on stochastic series, it can be found as Proposition 2.4 of \cite{TT} and in Section 3.2 of \cite{Banach}.
\begin{proposition}[Wiener Chaos] \label{20}

Let $d \ge 1$ and $c(n_1,n_2, \ldots, n_k)\in \mathbb{C}$. Let $g_n$ be a sequence of i.i.d. mean 0, variance 1 complex Gaussians on a probability space $\Omega_0$. For $k \ge 1$ let $\Lambda(k,d)$ denote the set of all k-tuples $(n_1,n_2,\ldots, n_k) \in \{1,2,\ldots, d\}^k$ satisfying $n_1 \le n_2 \le \ldots \le n_k$ and let
\begin{equation*}
S_k(\omega) = \sum_{\Lambda(k,d)} c(n_1,\ldots, n_k)g_{n_1}(\omega)g_{n_2}(\omega)\cdots g_{n_k}(\omega).
\end{equation*}
Then for $p \ge 2$, 
\begin{equation*}
\|S_k\|_{L^p(\Omega_0)} \le \sqrt{k+1}(p-1)^{k/2}\|S_k\|_{L^2(\Omega_0)}.
\end{equation*}   
\end{proposition}

We now state the key bound on the $L^2(\Omega)$ norm that guarantees the $S_{4,N}$ are a Cauchy sequence, which we prove at the end of the section. 
\begin{proposition} \label{24} 
Let  $c(\overline{n}) = c(n_1, n_2, n_3, n_4): \mathbb{Z}^4 \rightarrow \mathbb{R}$ be a function that satisfies the following properties:
\begin{itemize}
\item{} For all $\overline{n} \in \mathbb{Z}^4$, $|c(\overline{n})| \le \mathcal{M}^4 < \infty$  

\item{} The function $c(\overline{n})$ is even in each entry, meaning the value of $c(\overline{n})$ stays the same if you replace one of the entries $n_i$ with its negative.

\item{} For all $\overline{n}$ and $\overline{n'}$ whose entries are permutations of each other, $c(\overline{n}) = c(\overline{n'})$.
\end{itemize} For the fourth order chaos function $$S_{4,N}(\omega) = 3/4 \sum_{n_1+n_2 = n_3+n_4, |n_i| \le N} c(\overline{n})(n_1+n_2)\frac{g_{n_1}(\omega)g_{n_2}(\omega)\overline{g_{n_3}(\omega)}\overline{g_{n_4}(\omega)}}{\langle n_1 \rangle \langle n_2 \rangle \langle n_3 \rangle \langle n_4 \rangle}$$ and for any $M \ge N$ we have 
$$\int_{H^{1/2-}(\mathbb{T})} |S_{4,N} - S_{4,M}|^2 d\mu \lesssim \frac{\mathcal{M}^8}{N}.$$
\end{proposition}
During the proof we will find the condition that $c(\overline{n})$ is even is necessary to ensure that we can exploit symmetry. 

From Proposition \ref{24} we obtain the following corollary:
\begin{corollary} \label{25}
Given dyadic $N \le M$, we have $$\mathbb{P}( |S_{4,N}(\omega) - S_{4,M}(\omega)| > \lambda) \lesssim e^{-CN^{1/4}\lambda^{1/2}/\mathcal{M}^2}.$$
\end{corollary}
\begin{proof}
 Combing Propositions \ref{20} and \ref{24} we have $\| S_{4,N}(\omega) - S_{4,M}(\omega)\|_{L^p(\Omega)} \le \frac{\tau \mathcal{M}^4p^2}{N^{1/2}}$ for some constant $\tau$. By Markov's inequality, we have 
\begin{equation*}
\begin{split}
\mathbb{P}(| S_{4,N}(\omega) - S_{4,M}(\omega)| > \lambda) &\le \frac{\| S_{4,N}(\omega) - S_{4,M}(\omega)\|_{L^p(\Omega)}^p}{\lambda^p} \\
&\le \left(\frac{\tau \mathcal{M}^4 p^2}{N^{1/2}\lambda} \right)^p.\\
\end{split}
\end{equation*}
When $\sqrt{N^{1/2}\lambda/e\tau\mathcal{M}^4} \ge 2$, set $p = \sqrt{N^{1/2}\lambda/e\tau\mathcal{M}^4}$. We have 
\begin{equation*}
\mathbb{P}(| S_{4,N}(\omega) - S_{4,M}(\omega)| > \lambda) \le e^{-CN^{1/4}\lambda^{1/2}/\mathcal{M}^2}.
\end{equation*}
When $\sqrt{N^{1/2}\lambda/e\tau\mathcal{M}^4} < 2$ the quantity $N^{1/4}\lambda^{1/2}$ is small, and we have 
\begin{equation*}
\mathbb{P}(| S_{4,N}(\omega) - S_{4,M}(\omega)| > \lambda) \le 1 \lesssim  e^{-CN^{1/4}\lambda^{1/2}/\mathcal{M}^2}.
\end{equation*}
\end{proof}
We turn to proving Proposition \ref{24}.

\begin{proof}[Proof of Proposition \ref{24}]
We first expand the sum $S_{4,N}(\omega)$: 
\begin{equation}
S_{4,N}(\omega) =\frac{3}{4}\sum_{n} \sum_{n_1 + n_2 = n,|n_i| \le N} \sum_{m_1+m_2 = n, |m_i| \le N} c(\overline{n})  \frac{ng_{n_1}(\omega)g_{n_2}(\omega)\overline{g_{m_1}(\omega)}\overline{g_{m_2}(\omega)}}{\langle n_1 \rangle \langle n_2 \rangle \langle m_1 \rangle \langle m_2 \rangle}.
\end{equation}
We split this sum into three terms, $S^1_N$, $S^2_N$ and $S^3_N$, based on how many of the $n_i, m_j$ are equal. In $S^1_N$ all are equal, in $S^2_N$ two pairs are equal, and in $S^3_N$ none are equal. 
\begin{equation*}
\begin{split}
S^1_N &=\sum_{|n_1| \le N}  \frac{2c(\overline{n})n_1|g_{n_1}(\omega)|^4}{\langle n_1 \rangle^4}\\
S^2_N&= \sum_{|n_1| \le N} \sum_{|n_2| \le N, n_2 \ne n_1}  \frac{2c(\overline{n})(n_1+n_2)|g_{n_1}(\omega)|^2 |g_{n_2}(\omega)|^2}{\langle n_1 \rangle^2 \langle n_2 \rangle^2}\\
S_N^3 &=  \sum_{n} \sum_{n_1 + n_2 = n, |n_i| \le N} \sum_{m_1+m_2 = n, m_i \ne n_j, |m_i| \le N} \frac{c(\overline{n})ng_{n_1}(\omega)g_{n_2}(\omega)\overline{g_{m_1}(\omega)}\overline{g_{m_2}(\omega)}}{\langle n_1 \rangle \langle n_2 \rangle \langle m_1 \rangle \langle m_2 \rangle}.\\
\end{split}
\end{equation*}

We have $S_{4,N}(\omega) =  \frac{3}{4}\left(S_N^1 + S_N^2 +S_N^3\right)$ and therefore 
\begin{equation*}
\|  S_{4,N}(\omega) - S_{4,M}(\omega) \|_{L^2(\Omega)} \lesssim \|S_N^1 - S_M^1\|_{L^2(\Omega)} + \|S_N^2 - S_M^2 \|_{L^2(\Omega)} +  \|S_N^3 - S_N^3 \|_{L^2(\Omega)}.
\end{equation*}
So it suffices to show that each of these three $L^2$ norms is bounded by $\frac{\mathcal{M}^4}{N^{1/2}}$. We split the remainder of the proof into three parts.

\subsection{Bounding $S^1_N(\omega) - S^1_M(\omega)$}
We have $S_N^1 - S_M^1 = \sum_{N <  |n| \le M} \frac{2c(\overline{n})n|g_{n}(\omega)|^4}{\langle n \rangle^4}$. We will let $\overline{m}$ denote $(m_1,m_2,m_3,m_4)$. Squaring yields:
\begin{equation}
\begin{split}
|S_N^1 - S_M^1|^2 &=  4\sum_{N < |n| \le M}  \sum_{N < |m| \le M} \frac{c(\overline{n})c(\overline{m})nm|g_{n}(\omega)|^4|g_{m}(\omega)|^4}{\langle n \rangle^4 \langle m \rangle^4}\\
\|S_N^1 - S_M^1\|_{L^2(\Omega)}^2 &\le  4\sum_{N <  |n| \le M}  \sum_{N<  |m| \le M} \frac{\mathcal{M}^8\mathbb{E}(|g|^8)}{\langle n \rangle^3 \langle m \rangle^3}\\ 
\|S_N^1 - S_M^1\|_{L^2(\Omega)}^2 &\lesssim  \frac{\mathcal{M}^8}{N^4}.\\ 
 \end{split}
\end{equation}
This is way smaller than necessary. 

\subsection{Bounding $S^2_N(\omega) - S^2_M(\omega)$}
We have $S_N^2 =  2\sum_{n_1 \ne n_2, |n_i| \le N} \frac{c(\overline{n})(n_1+n_2)|g_{n_1}(\omega)|^2 |g_{n_2}(\omega)|^2}{\langle n_1 \rangle^2 \langle n_2 \rangle^2}$. All terms where $n_1= n_2$ are not included in this sum, since they are contained in $S_N^1$, and the $n_1= -n_2$ terms yield $n=0$. Therefore we can eliminate all terms where $|n_1| = |n_2|$ and take advantage of the $n_1, n_2$ symmetry to rewrite this sum as    
\begin{equation}
S_N^2 =  4\sum_{|n_1| \le N} \sum_{|n_2| \ne |n_1|, |n_2| \le N} \frac{c(\overline{n})n_1|g_{n_1}(\omega)|^2 |g_{n_2}(\omega)|^2}{\langle n_1 \rangle^2 \langle n_2 \rangle^2}.
\end{equation}
So when we subtract this from $S^2_M, M \ge N$, we lose all terms where $|n_1|,|n_2|$ are both $\le N$.  Recall that $c(\overline{n})$ does not change when we flip the sign of an entry. So we exploit the $n_1 \rightarrow -n_1, m_1 \rightarrow -m_1$ symmetry and obtain 
\begin{equation*}
\begin{split}
&|S_N^2 - S_M^2|^2 = 16  \sum_{\max{|n_i|} > N} \sum_{\max{|m_i|} > N}  \frac{c(\overline{n})c(\overline{m})n_1m_1|g_{n_1}(\omega)|^2|g_{n_2}(\omega)|^2|g_{m_1}(\omega)|^2|g_{m_2}(\omega)|^2}{\langle n_1 \rangle^2 \langle n_2 \rangle^2\langle m_1 \rangle^2 \langle m_2 \rangle^2}\\
&|S_N^2 - S_M^2|^2 =  16  \sum_{n_1 \ne n_2 > 0, \max{|n_1|,|n_2|} > N} \sum_{m_1 \ne m_2 >0, \max{|m_1|,|m_2| } > N} c(\overline{n})c(\overline{m})n_1m_1 \times \\
& \frac{(|g_{n_1}(\omega)|^2-|g_{-n_1}(\omega)|^2)(|g_{n_2}(\omega)|^2 + |g_{-n_2}(\omega)|^2)(|g_{m_1}(\omega)|^2-|g_{-m_1}(\omega)|^2)(|g_{m_2}(\omega)|^2 + |g_{-m_2}(\omega)|^2)}{\langle n_1 \rangle^2 \langle n_2 \rangle^2\langle m_1 \rangle^2 \langle m_2 \rangle^2}\\
\end{split}
\end{equation*}

Since $\mathbb{E}(|g_{n_i}(\omega)|^2- |g_{-n_i}(\omega)|^2) = 0 = \mathbb{E}(|g_{n_i}(\omega)|^4- |g_{-n_i}(\omega)|^4) $ for each $i$, and these are independent for distinct indices, each expected value in the above sum is $0$ unless $n_1 = m_1$. This would make $n_1m_1$ equal to  $n_1^2$. So we have 
\begin{equation*}
\begin{split}
\|S_N^2 - S_M^2\|_{L^2(\Omega)}^2 &\lesssim  \sum_{n_1, n_2, m_2 \ge 0, \max{|n_1|,|n_2|, |m_2|} > N} \frac{\mathcal{M}^8}{\langle n_1 \rangle^2 \langle n_2 \rangle^2 \langle m_2 \rangle^2} \times \\
\mathbb{E}[(|&g_{n_1}(\omega)|^2 - |g_{-n_1}(\omega)|^2)^2] \mathbb{E}[(|g_{n_2}(\omega)|^2 + |g_{-n_2}(\omega)|^2)]\mathbb{E}[(|g_{m_2}(\omega)|^2 + |g_{-m_2}(\omega)|^2)].\\
&\lesssim  \sum_{n_1, n_2, m_2 \ge 0, \max{|n_1|,|n_2|, |m_2|} > N} \frac{\mathcal{M}^8}{\langle n_1 \rangle^2 \langle n_2 \rangle^2 \langle m_2 \rangle^2}.\\
\end{split}
\end{equation*}
Since this term is now symmetric, we can assume without loss of generality that $n_1 \ge N$ and obtain 
\begin{equation}
\|S_N^2 - S_M^2\|_{L^2(\Omega)}^2 \lesssim \sum_{n_1 \ge N} \sum_{n_2} \sum_{m_2}  \frac{\mathcal{M}^8}{\langle n_1 \rangle^2 \langle n_2 \rangle^2 \langle m_2 \rangle^2} \lesssim \frac{\mathcal{M}^8}{N}.
\end{equation}
Therefore the $L^2(\Omega)$ norm is bounded by $\frac{\mathcal{M}^4}{N^{1/2}}$.

\subsection{Bounding $S^3_N(\omega) - S^3_M(\omega)$}
This is the most complicated term. We start with the following lemma. 
\begin{lemma} \label{27}
For any integer $n \ne 0$ we have 
\begin{equation} \label{28}
S_n = \sum_{m \in \mathbb{Z}} \frac{1}{\langle m \rangle^2 \langle n-m \rangle ^2}\lesssim \frac{1}{\langle n \rangle^2}.
\end{equation}
\end{lemma}
\begin{proof} Each term of this sum is bounded by $\frac{1}{\langle m \rangle^2}$, so the sum is absolutely convergent. This means we can rearrange the terms however we want. First let $n = 2k$ and replace $m$ with $m+k$. Then the summand of equation \eqref{28} is:
\begin{equation*}
\begin{split}
\frac{1}{\langle m \rangle^2 \langle n-m \rangle ^2} &= \frac{1}{\langle m+k \rangle^2 \langle m-k \rangle^2} \\
&= \frac{1}{(1+(m-k)^2)(1 + (m+k)^2)}\\
&= \frac{1}{4(1+k^2)}\left(\frac{1+ (k-m)/k}{1+(k-m)^2}+\frac{1+(k+m)/k}{1+(k+m)^2}\right)\\
\end{split}
\end{equation*}
Therefore the whole sum $S_n$ satisfies
\begin{equation*}
\begin{split}
S_n &= \sum_{m+k \in \mathbb{Z}} \frac{1}{4(1+k^2)}\left(\frac{1+ (k-m)/k}{1+(k-m)^2}+\frac{1+(k+m)/k}{1+(k+m)^2} \right)\\
&= \frac{1}{4(1+k^2)} \sum_{m+k \in \mathbb{Z}}\frac{1+(k+m)/k}{1+(k+m)^2} +  \frac{1}{4(1+k^2)} \sum_{m-k \in \mathbb{Z}} \frac{1+ (k-m)/k}{1+(k-m)^2}\\
&= \frac{1}{4(1+k^2)} \sum_{m \in \mathbb{Z}}\frac{1+m/k}{1+m^2} +  \frac{1}{4(1+k^2)} \sum_{m \in \mathbb{Z}} \frac{1+ -m/k}{1+m^2}\\
&= \frac{2}{4(1+k^2)} \sum_{m \in \mathbb{Z}}\frac{1}{1+m^2}  \\
&\lesssim \frac{1}{\langle k \rangle^2}. \\
\end{split}
\end{equation*}
Substituting $n/2$ for $k$, we obtain the desired bound.
\end{proof}
Now let $\mathbb{A}_n$ be the set of $(n_1,n_2,m_1,m_2)$, not all with absolute value $\le M$, but not all $\le N$, such that $n_1+n_2 = n = m_1+m_2$, and $n_1 \ne m_1,m_2$, $n_2 \ne m_1,m_2$. We have  
\begin{equation*}
\begin{split}
S^3_N(u) - S^3_M(u) &=  \sum_{n} \sum_{n_1 + n_2 = n, |n_i| \le N} \sum_{m_1+m_2 = n, m_i \ne n_j}  \frac{c(\overline{n})ng_{n_1}(\omega)g_{n_2}(\omega)\overline{g_{m_1}(\omega)}\overline{g_{m_2}(\omega)}}{\langle n_1 \rangle \langle n_2 \rangle \langle m_1 \rangle \langle m_2 \rangle}\\
|S^3_N(u) - S^3_M(u)|^2 &= \sum_{n,p} \sum_{\mathbb{A}_n \times \mathbb{A}_p}  \frac{c(\overline{n})ng_{n_1}g_{n_2}\overline{g_{m_1}}\overline{g_{m_2}}(\omega)}{\langle n_1 \rangle \langle n_2 \rangle \langle m_1 \rangle \langle m_2 \rangle}\frac{c(\overline{p})pg_{p_1}g_{p_2}\overline{g_{q_1}}\overline{g_{q_2}}(\omega)}{\langle p_1 \rangle \langle p_2 \rangle \langle q_1 \rangle \langle q_2 \rangle}\\
\|S^3_N(u) - S^3_M(u)\|_{L^2(\Omega)}^2 &=  \sum_{n,p}\sum_{\mathbb{A}_n \times \mathbb{A}_p} \frac{c(\overline{n})c(\overline{p})np \mathbb{E}[g_{n_1}g_{n_2}\overline{g_{m_1}}\overline{g_{m_2}}g_{p_1}g_{p_2}\overline{g_{q_1}}\overline{g_{q_2}}(\omega)]}{\langle n_1 \rangle \langle n_2 \rangle \langle m_1 \rangle \langle m_2 \rangle\langle p_1 \rangle \langle p_2 \rangle \langle q_1 \rangle \langle q_2 \rangle}.\\
\end{split}
\end{equation*}
This expectation is zero unless we can pair off each index. So our indices must satisfy \newline $\{n_1,n_2\} = \{q_1,q_2\}$, and  $\{m_1,m_2\} = \{p_1,p_2\}$. Noting that all other terms are $0$, we obtain 
\begin{equation*}
\begin{split}
\|S^3_N(u) - S^3_M(u)\|_{L^2}^2 &=  \sum_{n} \sum_{\mathbb{A}_n} \frac{4c(\overline{n})c(\overline{p})n^2 \mathbb{E}[|g_{n_1}(\omega)|^2 |g_{n_2}(\omega)|^2|g_{m_1}(\omega)|^2|g_{m_2}(\omega)|^2]}{\langle n_1 \rangle^2 \langle n_2 \rangle^2 \langle m_1 \rangle^2 \langle m_2 \rangle^2}\\
&\lesssim  \sum_{n} \sum_{\mathbb{A}_n} \frac{\mathcal{M}^8n^2}{\langle n_1 \rangle^2 \langle n_2 \rangle^2 \langle m_1 \rangle^2 \langle m_2 \rangle^2}.
\end{split}
\end{equation*}
Observe that this sum is symmetric and at least one index has absolute value $ > N$. Assume, without loss of generality, that $|n_1| > N$. Applying Lemma \ref{27}, we have 
\begin{equation}
\begin{split}
\|S^3_N(u) - S^3_M(u)\|_{L^2}^2  &\lesssim  \sum_{n} \sum_{\mathbb{A}_n} \frac{\mathcal{M}^8n^2}{\langle n_1 \rangle^2 \langle n_2 \rangle^2 \langle m_1 \rangle^2 \langle m_2 \rangle^2}\\
&\lesssim \sum_{n} \sum_{n_1+n_2 = n, |n_1| > N} \frac{\mathcal{M}^8}{\langle n_1 \rangle^2 \langle n_2 \rangle^2} \sum_{m_1+m_2 = n} \frac{n^2}{\langle m_1 \rangle^2 \langle m_2 \rangle^2}\\
&\lesssim  \sum_{n} \sum_{n_1+n_2 = n, |n_1| > N} \frac{\mathcal{M}^8}{\langle n_1 \rangle^2 \langle n_2 \rangle^2} \\
&\lesssim \sum_{|n_1| > N} \sum_{n_2}   \frac{\mathcal{M}^8}{\langle n_1 \rangle^2 \langle n_2 \rangle^2} \\
&\lesssim \frac{\mathcal{M}^8}{N}.\\
\end{split}
\end{equation}

Therefore all three $L^2(\Omega)$ terms are bounded by $\frac{\mathcal{M}^4}{N^{1/2}}$, and equivalently their squares are bounded by $\frac{\mathcal{M}^8}{N}$. 
\end{proof}

\section{Boundedness of scaled $\Psi_N$} \label{mathrefs}

In this section we will complete the proof of Theorem \ref{401}, implying that $\Psi_N \in L^1(\mu)$ even after scaling its Fourier coefficients. Keeping the notation $\mathcal{M}, c(\overline{n})$, and $S_4(u)$ of Theorem \ref{401} we expand the function $S_{4,N}$ as  

\begin{equation} \label{500}
\begin{split}
S_{4,N}(\omega)
&= \frac{3}{4} \sum_{n_1+n_2 = n_3+n_4, |n_i| \le N}(n_1+n_2)\frac{c(n_1)g_{n_1}c(n_2)g_{n_2}c(n_3)\overline{g_{n_3}}c(n_4)\overline{g_{n_4}}}{\langle n_1 \rangle \langle n_2 \rangle \langle n_3 \rangle \langle n_4 \rangle},\\
\end{split}
\end{equation}
We now define the function we need to bound in $L^p(\mu)$. 
\begin{definition}
For any fixed exponent $p$ and mass $m_{\mathcal{M},p}$ define the function
\begin{equation} \label{501}
\begin{split}
\psi_N &= 1_{\|u_N \|_{L^2(\mathbb{T})}^2 < m_{\mathcal{M},p}} e^{S_{4,N}(\omega)}.\\
\end{split}
\end{equation}
Note that for the same values of $m_{\mathcal{M},p}$, $\psi_N$ is $\Psi_N$ without the $e^{-\frac{1}{2}\int_{\mathbb{T}} |u|^6 dx}$ term, and therefore $\Psi_N \le \psi_N$.

\end{definition}

Observe that if we let $\tilde{u} = \sum_{n} \frac{c(n)g_n(\omega)e^{inx}}{\langle n \rangle}$ then  the distribution of $S_{4,N}(\omega)$ with respect to $\mathbb{P}$ is the same as the distribution of $f_N(\tilde{u})$. We now prove Theorem \ref{401} by replacing $S_{4,N}(\omega)$ with $f_N(\tilde{u})$. When $c(n) =1$, we have $u = \tilde{u}$, so Theorem \ref{5} is a subcase of Theorem \ref{401}.

\begin{proof}[Proof of Theorem \ref{401}] 
 
We will bound the expected value of $\psi_N(\tilde{u})$ by obtaining bounds on the $\mu$-measure of the set where $\psi_N$ is larger than some $\lambda$. 

Fix a value of $\lambda >16\mathcal{M}^4\ln^2(2)$, a constant $c_0$ that will be defined later, a parameter $\delta \in (0,\frac{1}{8})$  and set $c_{\delta} = \frac{1-2^{-\delta}}{4c_0}$. Let $N_1$ be the greatest dyadic integer less than $\frac{c_1\lambda^{1/2}}{m_{\mathcal{M},p}}$, and similarly define dyadic integers $N_2 \le  N_3 $ to be the greatest dyadics not exceeding $\frac{c_1\lambda}{m_{\mathcal{M},p}}, \frac{c_1\lambda^{2}}{m_{\mathcal{M},p}}$ respectively, for a  constant $c_1$ we will pick later.  

First we define a weight function $h(n)$:

\begin{definition}
For $N_2\ge  n \ge N_1$ let $h(n) = c_{\delta}(\frac{N_1}{n})^{\delta}$. For $N_3 \ge n > N_2$ let $h(n) = c_{\delta}(\frac{n}{N_3})^{\delta}$.

\end{definition}
Note that the sum $\sum_{n -dyadic} h(n)$ is bounded by the geometric series 
\begin{equation*}
2c_{\delta}(1 + 2^{-\delta} + 2^{-2\delta} + \ldots) \le \frac{2c_{\delta}}{1-2^{-\delta}} \le \frac{1}{2c_0}.
\end{equation*}
We conclude that
\begin{equation}
\sum_{n -dyadic} h(n) \le \frac{1}{2c_0}.
\end{equation}
We now consider the norm $\| \psi_N(\tilde{u}) \|^p_{L^p(\mu)}$ and bound it by viewing it as an expected value:
\begin{equation*}
\begin{split}
\|  \psi_N(\tilde{u})\|^p_{L^p(\mu)} &=
\mathbb{E}_{\mu}(|1_{\|\tilde{u}_N\|_{L^2(\mathbb{T})}^2 < m_{\mathcal{M},p}}e^{f_N(\tilde{u})}|^p)\\
&= \int_{0}^{\infty} p\lambda^{p-1}\mu(1_{\|\tilde{u}_N\|_{L^2(\mathbb{T})}^2 < m_{\mathcal{M},p}}e^{f_N(\tilde{u})} >\lambda) d\lambda\\
&= \int_{0}^{\infty} pe^{p\lambda}\mu(\{ f_N(\tilde{u}) >\lambda\} \cap \{\|\tilde{u}_N\|_{L^2(\mathbb{T})}^2< m_{\mathcal{M},p}\}) d\lambda\\
\end{split}
\end{equation*}

It is sufficient to bound the probability 
\begin{equation} \label{prob}
\mu(\{ f_N(\tilde{u}) >\lambda\} \cap \{\|\tilde{u}_N\|_{L^2(\mathbb{T})}^2< m_{\mathcal{M},p}\}) 
\end{equation}
for large values of $\lambda$. When $\lambda \le 16\mathcal{M}^4\ln^2(2)$ we just bound \eqref{prob} by $1$, and note that this provides a finite contribution to the total integral. 
We will show that for $\lambda > 16\mathcal{M}^4\ln^2(2)$ expression \eqref{prob} is bounded by $e^{-C\lambda/m_{\mathcal{M},p}^{\alpha}\mathcal{M}^2}$ for some positive exponent $\alpha$. This implies that $\|  \psi_N(\tilde{u})\|^p_{L^p(d\mu)}$ is bounded by a constant depending only on $\mathcal{M},p$ for sufficiently small $m_{\mathcal{M},p}$. 

Now we split $f_N(\tilde{u})$ into $f_{N_3}(\tilde{u})$ and $f_N(\tilde{u}) - f_{N_3}(\tilde{u})$. Throughout this section $n$ or $n_i$ will denote a dyadic integer. We have 
\begin{equation*}
\begin{split}
f_{N_3}(\tilde{u})  &= -\frac{3}{2}i \int_{\mathbb{T}} P_{\le N_3}\tilde{u}_{x}P_{\le N_3}\tilde{u}\overline{P_{\le N_3}\tilde{u}}\overline{P_{\le N_3}\tilde{u}} dx \\
|f_{N_3}(\tilde{u})| &\lesssim  \sum_{N_3 \ge n_1 \sim n_2 \ge n_3 \ge n_4}\int_{\mathbb{T}} n_1|P_{n_1}\tilde{u}| \cdot |P_{n_2}\tilde{u}| \cdot |P_{n_3}\tilde{u}| \cdot |P_{n_4}\tilde{u}|   dx \\
&\lesssim  \sum_{N_3 \ge n_1 \sim n_2 \ge n_3 \ge n_4}\int_{\mathbb{T}} n_1^{1/2}n_2^{1/2}|P_{n_1}\tilde{u}| \cdot |P_{n_2}\tilde{u}| \cdot |P_{n_3}\tilde{u}| \cdot |P_{n_4}\tilde{u}|   dx \\
&\lesssim  \sum_{n_i \le N_3} n_1^{1/2}n_2^{1/2}\|P_{n_1}\tilde{u}\|_{L^2(\mathbb{T})}  \|P_{n_2}\tilde{u}\|_{L^2(\mathbb{T})}  \|P_{n_3}\tilde{u}\|_{L^{\infty}(\mathbb{T})}  \|P_{n_4}\tilde{u}\|_{L^{\infty}(\mathbb{T})} \\
&\lesssim  \sum_{n_i \le N_3} n_1^{1/2}n_2^{1/2}n_3^{1/2}n_4^{1/2}\|P_{n_1}\tilde{u}\|_{L^2(\mathbb{T})}  \|P_{n_2}\tilde{u}\|_{L^2(\mathbb{T})}  \|P_{n_3}\tilde{u}\|_{L^2(\mathbb{T})}  \|P_{n_4}\tilde{u}\|_{L^2(\mathbb{T})}, \\
\end{split}
\end{equation*}  
where we applied Sobolev embedding and the fact that the product is $0$ if the largest frequency is significantly larger than others. Factoring yields
\begin{equation} \label{520}
\begin{split}
f_{N_3}(\tilde{u}) &\lesssim  \left( \sum_{n \le N_3}n^{1/2}\|P_n\tilde{u}\|_{L^2(\mathbb{T})} \right)^4. \\
\end{split}
\end{equation}  
So for some constant $c_0$, 
\begin{equation*}
\begin{split}
\mu(\{f_N(&\tilde{u}) > \lambda\} \cap \{\|\tilde{u}_N\|_{L^2(\mathbb{T})}^2 < m_{\mathcal{M},p}\}) \\
&\le \mu \left( |f_N(\tilde{u})-f_{N_3}(\tilde{u})|> \frac{\lambda}{2}\right) +   \mu\left(\{\sum_{n \le N_3}n^{1/2}\|P_n\tilde{u}\|_{L^2(\mathbb{T})} > \frac{2\lambda^{1/4}}{c_0}\} \cap \{\|\tilde{u}_N\|_{L^2(\mathbb{T})}^2 < m_{\mathcal{M},p}\} \right)\\
&\le  \mu\left( |f_N(\tilde{u})-f_{N_3}(\tilde{u})|> \frac{\lambda}{2} \right) + \mu\left(\{\sum_{n \le N_1}n^{1/2}\|P_n\tilde{u}\|_{L^2(\mathbb{T})} > \frac{\lambda^{1/4}}{c_0}\} \cap \{\|\tilde{u}_N\|_{L^2(\mathbb{T})}^2< m_{\mathcal{M},p}\} \right) \\
&+  \mu\left( \sum_{N_1 < n \le N_3}n^{1/2}\|P_n\tilde{u}\|_{L^2(\mathbb{T})} > \frac{\lambda^{1/4}}{c_0} \right) .  \\
\end{split}
\end{equation*}  

This provides the constant $c_0$ in the definition of $h(n)$. We will split the rest of the proof into three sections and show that each of these three terms is bounded by $e^{-C\lambda/\mathcal{M}^2m_{\mathcal{M},p}^{\alpha}}$ for some positive power $\alpha \ge \frac{1}{4}$. Then taking $m_{\mathcal{M},p}$ small enough proves that this integral is bounded by a constant depending on  $\mathcal{M},p$. 

\subsection{The High Frequency Term: $\mu(\{ |f_N(\tilde{u})-f_{N_3}(\tilde{u})|   > \lambda /2 \})$} \label{mathrefs}
By Corollary \ref{25}, this probability is bounded by
\begin{equation}
\begin{split}
\mu(\{ |f_N(\tilde{u})-f_{N_3}(\tilde{u})|   > \lambda /2 \})  &\le  e^{-C\lambda^{1/2}N_3^{1/4}/\mathcal{M}^2}\\
&\le e^{-C\lambda/m^{1/4}_{\mathcal{M},p}\mathcal{M}^2}.\\ 
\end{split}
\end{equation}

\subsection{The Low Frequency Term:}$ \mu(\{\sum_{n \le N_1}n^{1/2}\|P_n\tilde{u}\|_{L^2(\mathbb{T})} > \frac{\lambda^{1/4}}{c_0}\} \cap \{\|\tilde{u}_N\|_{L^2(\mathbb{T})}^2 < m_{\mathcal{M},p}\}) $ \newline

In this section we utilize the fact that the mass $m_{\mathcal{M},p}$ is small to show that the above probability is $0$.

By the Cauchy-Schwarz inequality and our mass bound, we have 
\begin{equation} \label{530}
\begin{split}
\sum_{n \le N_1}n^{1/2}\|P_n\tilde{u}\|_{L^2(\mathbb{T})} &\le \left( \sum_{n \le N_1} n\right)^{1/2} \cdot \left( \sum_{n \le N_1} \|P_n\tilde{u}\|_{L^2(\mathbb{T})}^2\right)^{1/2}\\
&\le 2N_1^{1/2}m_{\mathcal{M},p}^{1/2}\\
&\le 2\left(\frac{c_1\lambda^{1/2}}{m_{\mathcal{M},p}} \right)^{1/2}m_{\mathcal{M},p}^{1/2}\\
&\le 2c_1^{1/2}\lambda^{1/4}.\\
\end{split}
\end{equation}

Therefore there exists a sufficiently small constant $c_1$ such that 
\begin{equation}
\sum_{n \le N_1}n^{1/2}\|P_n\tilde{u}\|_{L^2(\mathbb{T})} \le 2c_1^{1/2}\lambda^{1/4} < \frac{\lambda^{1/4}}{c_0},
\end{equation}
and this sum is $> \frac{\lambda^{1/4}}{c_0}$ with probability $0$.

\subsection{The Middle Frequency terms:}\label{mathrefs} $ \mu\left( \sum_{ N_1 < n \le N_3}n^{1/2}\|P_n\tilde{u}\|_{L^2(\mathbb{T})} > \frac{\lambda^{1/4}}{c_0} \right)$ \newline

These are arguably the hardest terms to bound. Their frequencies are too large to be bounded by the mass and too small to take advantage of decay.

Recall that $h(n) \le 1$ and $\sum_{n -dyadic} h(n) \le \frac{1}{2c_0}$. We have
\begin{equation}
\begin{split}
 \mu\left( \sum_{ N_1 < n \le N_3}n^{1/2}\|P_n\tilde{u}\|_{L^2(\mathbb{T})} > \frac{\lambda^{1/4}}{c_0} \right) &\le \sum_{N_1 < n \le N_3} \mu\left(n^{1/2}\|P_n\tilde{u}\|_{L^2(\mathbb{T})} > 2h(n)\lambda^{1/4}\right)\\
&\le \sum_{N_1 < n \le N_3} \mu\left(\|P_n\tilde{u}\|^2_{L^2(\mathbb{T})} > \frac{4h(n)^2\lambda^{1/2}}{n}\right).\\
\end{split}
\end{equation}

Assuming $n \le N_3$ is small enough to satisfy the conditions of Lemma \ref{13}, we apply Lemma \ref{13} and obtain 
\begin{equation}
\begin{split}
 \mu\left( \sum_{ N_1 < n \le N_3}n^{1/2}\|P_n\tilde{u}\|_{L^2(\mathbb{T})} > \frac{\lambda^{1/4}}{c_0}\right) &\le \sum_{N_1 < n \le N_3} \mu\left(\|P_n\tilde{u}\|^2_{L^2(\mathbb{T})} > \frac{4h(n)^2\lambda^{1/2}}{n}\right)\\
&\lesssim \sum_{N_1 < n \le N_3} e^{- h(n)^2n\lambda^{1/2}/\mathcal{M}^2}  \\
\end{split}
\end{equation}

Recall that $2N_1 \ge \frac{c_1\lambda^{1/2}}{m_{\mathcal{M},p}}$. We split the sum into the sum from $N_1 < n \le N_2$ and $N_2 < n \le N_3$. The sum from $N_1$ to $N_2$ is bounded by 
$$e^{-C\lambda/{m_{\mathcal{M},p}}\mathcal{M}^2}  + e^{-C2^{1-\delta}\lambda/{m_{\mathcal{M},p}}\mathcal{M}^2} + e^{-C2^{2(1-\delta)}\lambda/{m_{\mathcal{M},p}}\mathcal{M}^2} + e^{-C2^{3(1-\delta)}\lambda/{m_{\mathcal{M},p}}\mathcal{M}^2} + \ldots \lesssim  e^{-C\lambda/{m_{\mathcal{M},p}}\mathcal{M}^2}$$
and for $\delta < \frac{1}{8}$ this value is sufficiently small. 

For the second sum from $N_2$ to $N_3$ we have  
\begin{equation*}
\begin{split}
\sum_{N_2 < n \le N_3} e^{- h(n)^2n\lambda^{1/2}/\mathcal{M}^2} &=
e^{-C(2N_2/N_3)^{2\delta}2N_2\lambda^{1/2}/\mathcal{M}^2}+  
e^{-C(4N_2/N_3)^{2\delta}4N_2\lambda^{1/2}/\mathcal{M}^2}
\ldots  + e^{-CN_3\lambda^{1/2}/\mathcal{M}^2} \\
&\lesssim e^{-C(N_2/N_3)^{2\delta}N_2\lambda^{1/2}/\mathcal{M}^2}\\
&\lesssim e^{-C\lambda^{3/2-2\delta}/m_{\mathcal{M},p}^{1+2\delta}\mathcal{M}^2}\\
\end{split}
\end{equation*}
which for $\delta < \frac{1}{8}$ is more than small enough. 

It remains to verify that $N_3$ is small enough to fulfill the conditions of Lemma \ref{13}. We must ensure that $N_3 = c_2\lambda^2$ satisfies  $\frac{4h(N_3)^2\lambda^{1/2}}{N_3} > \frac{16\mathcal{M}^2\ln(2)}{N_3}$ to apply Lemma \ref{13}. Since $h(N_3)=1$ this holds as long as $\lambda > 16\mathcal{M}^4\ln^2(2)$, which is our lower bound on $\lambda$. 

In each of the three sections we have shown that a specific term is bounded by $e^{-C\lambda/m_{\mathcal{M},p}\mathcal{M}^2}$. Therefore, for sufficiently small $m_{\mathcal{M},p}$, the $\| \psi_N(\tilde{u}) \|^p_{L^p(\mu)}$ term is bounded by a constant depending on $\mathcal{M}$ and $p$. This completes the proof of Theorem \ref{401} and in turn Theorem \ref{5}.

\end{proof}

\subsection{A related bound}
In the previous subsection we proved that 
$$\int 1_{\|u\|_{L_2}^2 < m_{\mathcal{M},p}} e^{pf_N(\tilde{u})} d\mu = \int 1_{\|u\|_{L_2}^2 < m_{\mathcal{M},p}} e^{\frac{-3i}{2}\int_{\mathbb{T}} \partial_x (\tilde{u}) \cdot  \tilde{u} \cdot \overline{\tilde{u}}\cdot \overline{\tilde{u}}  dx} d\mu$$ is bounded by $C(\mathcal{M},p)$ for all $\mathcal{M},p$. We now prove a slight generalization for when we only take portions of $\tilde{u}$ in each term. We start with the following definition.

\begin{definition} \label{551}
Let $Q$ be an even subset of the integers, meaning that for each $n \in \mathbb{Z}$, $n \in Q$ if and only if $-n \in Q$. We let $P_Qu = \sum_{n \in Q} \widehat{u}(n)e^{inx}$. One can view $P_Q$ as a Fourier projection operator, and it commutes with other projection operators such as $P_N, P_{\le N}$ for $N$ dyadic.
\end{definition}
We now prove our generalization of Theorem \ref{401}:

\begin{proposition} \label{552}

Keep the notation $\mathcal{M}$ and $\tilde{u}$ of Theorem \ref{401}. For any 4-tuple $Q_1,Q_2,Q_3,Q_4$ of even subsets of the integers and dyadic $N \ge 0$ we have 
\begin{equation*}
\int 1_{\|u\|_{L_2}^2 < m_{\mathcal{M},p}} e^{p\frac{-3i}{2}\int_{\mathbb{T}} \partial_x (P_{Q_1} \tilde{u}_N) \cdot  P_{Q_2} \tilde{u}_N \cdot P_{Q_3} \overline{\tilde{u}_N}\cdot P_{Q_4} \overline{\tilde{u}_N}  dx} d\mu   \le C(\mathcal{M},p).
\end{equation*}
\end{proposition}
We emphasize that the constant does not depend on the choice of sets $Q_i$.
\begin{proof}
Proposition \ref{24} implies that for any four even subsets $Q_1,Q_2,Q_3,Q_4$ and any $M \ge N$

\begin{equation}
\int_{H^{1/2-}(\mathbb{T})} 
| \int_{\mathbb{T}} \partial_x (P_{Q_1}\tilde{u}_M) \cdot  P_{Q_2}\tilde{u}_M \cdot P_{Q_3}\overline{\tilde{u}_M}\cdot P_{Q_4}\overline{\tilde{u}_M}   - \partial_x (P_{Q_1}\tilde{u}_N) \cdot  P_{Q_2}\tilde{u}_N \cdot P_{Q_3}\overline{\tilde{u}_N}\cdot P_{Q_4}\overline{\tilde{u}_N}  dx |^2  d\mu \le \frac{C\mathcal{M}^8}{N} \\
\end{equation}
for a universal constant $C$. Therefore we can apply the Wiener Chaos bound and the same argument as Corollary \ref{25} to the expansion of this term in $g_n$ form to obtain for any $M >N$,
\begin{equation} \label{555}
\begin{split}
\mu( |\int_{\mathbb{T}} \partial_x (P_{Q_1}\tilde{u}_M) \cdot  P_{Q_2}\tilde{u}_M \cdot P_{Q_3}\overline{\tilde{u}_M}\cdot P_{Q_4}\overline{\tilde{u}_M}   - \partial_x (P_{Q_1}\tilde{u}_N) \cdot  P_{Q_2}\tilde{u}_N \cdot P_{Q_3}\overline{\tilde{u}_N}\cdot &P_{Q_4}\overline{\tilde{u}_N}  dx| > \lambda) \\
&\lesssim e^{-CN^{1/4}\lambda^{1/2}/\mathcal{M}^2}. \\
\end{split}
\end{equation}

 We also know  from the proof of Theorem \ref{401} that for the same constants $c_0,c_1$ and $N_3$ equal to the greatest dyadic not exceeding $\frac{c_1\lambda^{2}}{m_{\mathcal{M},p}}$, we have 
$$ \mu \left(\{\sum_{n \le N_3}n^{1/2}\|P_n\tilde{u}\|_{L^2(\mathbb{T})} > \frac{2\lambda^{1/4}}{c_0}\} \cap \{\|\tilde{u}_N\|_{L^2(\mathbb{T})}^2 < m_{\mathcal{M},p}\} \right) \le e^{-C\lambda/\mathcal{M}^2m_{\mathcal{M},p}^{\alpha}}.$$

The operator $P_{Q_i}$ is a projection, which implies  that $\|P_{Q_i}P_n \tilde{u}\|_{L^2(\mathbb{T})} \le \|P_n\tilde{u}\|_{L^2(\mathbb{T})}$. We conclude that for each $Q_i$ we have 
\begin{equation}  \label{556} 
\mu \left(\{\sum_{n-\text{dyadic} \le N_3} n^{1/2}\|P_{Q_i}P_n\tilde{u}\|_{L^2(\mathbb{T})} > \frac{2\lambda^{1/4}}{c_0}\} \cap \{\|\tilde{u}_N\|_{L^2(\mathbb{T})}^2 < m_{\mathcal{M},p}\} \right) \le e^{-C\lambda/\mathcal{M}^2m_{\mathcal{M},p}^{\alpha}}.
\end{equation}
Combining equations \eqref{555} and \eqref{556}, we can repeat the proof of Theorem \ref{401} to conclude that 

\begin{equation*}
\int 1_{\|u\|_{L_2}^2 < m_{\mathcal{M},p}} e^{p\frac{-3i}{2}\int_{\mathbb{T}} \partial_x (P_{Q_1} \tilde{u}_N) \cdot  P_{Q_2} \tilde{u}_N \cdot P_{Q_3} \overline{\tilde{u}_N}\cdot P_{Q_4} \overline{\tilde{u}_N}  dx} d\mu   \le C(\mathcal{M},p),
\end{equation*}
for any dyadic $N$.

\end{proof}

This proposition allows us to bound the exponential of some of the individual components of $f_N(u)$. For example, we can take two disjoint even sets $Q_1$ and $Q_2$ whose union is the integers and bound $e^{f_N(u)}$ by writing it as the product of the exponential of the 16 terms of the expansion of $f_N(u_{Q_1}+u_{Q_2})$ and applying H\"older's inequality to all $16$ components. An argument like this will be necessary in Section 6.

\section{Analysis of the measure $\mu_m$ and its decomposition}
In this section we define the measure $\mu_m$ supported on $H^{1/2-}(\mathbb{T}) \cap \{\|u\|_{L^2(\mathbb{T})}^2 =  m\}$, and write it as a linear combination of measures $\nu_m^k$ derived from scaling.

\subsection{Bounds on the distribution of the mass} \label{mathrefs}

We now begin developing the tools necessary to construct the fixed mass measure. Since $\text{mass}(u) = \sum_{n} \frac{|g_n|^2}{\langle n \rangle^2}$ we will need bounds on the probability distribution functions of various sums of Fourier coefficients squared.  

\begin{definition} \label{601} We define the following probability distribution functions relating to the mass. A priori we must view them as distributions, but we will prove that they are all in fact uniformly continuous functions. 

\begin{enumerate}
\item{} For $N \ge 0$, we let $p_N$ be the distribution function of the random variable $\underset{|n| \ge N}{\sum} \frac{|g_{n}|^2}{\langle n \rangle^2}$ with respect to to the probability measure $\mu$. Note that $p_0$ is the distribution of the mass.

\item{} For any $N \ge 0$  let $P_N$ be the distribution function of  $\underset{|n| \ne N}{\sum} \frac{|g_{n}|^2}{\langle n \rangle^2}$.

\end{enumerate}
\end{definition}
We begin with two lemmas bounding the $L^{\infty}(\mathbb{R})$ norms of these distribution functions. 
\begin{lemma} \label{602}
For each $N$, $\|\widehat{p_N}\|_{L^1(\mathbb{R})}$ is finite, implying that $p_N$ is uniformly continuous and has finite $L^{\infty}(\mathbb{R})$ norm. 
\end{lemma}
\begin{proof}
We know that 
\begin{equation}
\begin{split}
\widehat{p_N}(\xi) &= \mathbb{E_{\mu}}\left (e^{i \xi \sum_{|n| \ge N} \frac{|g_{n}|^2}{\langle n \rangle^2}} \right )  \\
&= \Pi_{|n| \ge N}  \int_{\mathbb{R}} \frac{1}{\sqrt{ \pi}} e^{(i\xi/\langle n \rangle^2 -1)\text{Re}(g_{n})^2} d\text{Re}(g_{n})  \times  \int_{\mathbb{R}}  \frac{1}{\sqrt{\pi}} e^{(i\xi/\langle n \rangle^2 - 1)\text{Im}(g_{n})^2} d\text{Im}(g_{n}) \\  
&= \Pi_{|n| \ge  N} \left ( \int_{\mathbb{R}}  \frac{1}{\sqrt{\pi}} e^{(i\xi/\langle n \rangle^2 - 1)x^2} dx \right)^2 \\
\end{split}
\end{equation}
Noting that $\int_{\mathbb{R}} \frac{1}{\sqrt{\pi}} e^{-cx^2}dx = \frac{1}{\sqrt{c}}$ for $c$ with a positive real part, we have
\begin{equation}
\begin{split}
\widehat{p_N}(\xi) &= \Pi_{|n| \ge N}\left ( \int_{\mathbb{R}}  \frac{1}{\sqrt{ \pi}} e^{(i\xi/\langle n \rangle^2 - 1)x^2} dx \right)^2   \\
&=\Pi_{|n| \ge N} \frac{1}{1 - i\xi/ \langle n \rangle^2 }\\
\end{split}
\end{equation}
Each term of this product has complex norm $\le 1$, therefore the norm of the product is bounded by the norm of the product of any two terms, such as $|\frac{1}{1 - i\xi/ \langle N +1 \rangle^2 }|^2 = \frac{1}{1 + \xi^2/(1+(N+1)^2)^2}$. The $\xi^2$ term ensures that this function has finite $L^1(\mathbb{R})$ norm. 

Therefore we conclude that $\|\widehat{p_N}\|_{L^{1}(\mathbb{R})}$ and in turn $\|p_N\|_{L^{\infty}(\mathbb{R})}$ is bounded. 
\end{proof}

\begin{lemma} \label{605}
The distribution functions $P_N$ have uniformly bounded $L^{\infty}(\mathbb{R})$ norms. 
\end{lemma}
\begin{proof}
Similar to the previous lemma, we have
\begin{equation}
\begin{split}
\widehat{P_N}(\xi) &= \mathbb{E_{\mu}}\left (e^{i \xi \sum_{|n| \ne N} \frac{|g_{n}|^2}{\langle {n} \rangle^2}} \right )  \\
&= \Pi_{|n| \ne N}  \int_{\mathbb{R}} \frac{1}{\sqrt{\pi}} e^{(i\xi/\langle n \rangle^2 - 1)\text{Re}(g_{n})^2} d\text{Re}(g_{n})  \times  \int_{\mathbb{R}}  \frac{1}{\sqrt{\pi}} e^{(i\xi/\langle n \rangle^2 - 1)\text{Im}(g_{n})^2} d\text{Im}(g_{n}) \\  
&= \Pi_{|n| \ne N} \frac{1}{1 - i\xi/ \langle n \rangle^2 }\\
\end{split}
\end{equation}
Each term of this product has complex norm $\le 1$, therefore the norm of the product is bounded by the norm of the product of the first few terms. For $N \ne 1$, this product is bounded by $$|\frac{1}{1 - i\xi/ \langle 1 \rangle^2 }|^2 = \frac{1}{(1- i\xi/2)(1 + i \xi/2)} = \frac{1}{1 + \xi^2/4}.$$ For $N = 1$ the product is bounded by $$|\frac{1}{1 - i\xi/ \langle 2 \rangle^2 }|^2 = \frac{1}{(1- \frac{i\xi}{5})(1 + \frac{i \xi}{5})} = \frac{1}{1 + \xi^2/25}.$$ Therefore for all $N$ the $L^1(\mathbb{R})$ norm of $\widehat{P_N}$ is bounded by the max of the $L^1(\mathbb{R})$ norm of these two functions. 

We conclude that $\|\widehat{P_N}\|_{L^{1}(\mathbb{R})}$ and in turn $\|P_N\|_{L^{\infty}(\mathbb{R})} \le C$ uniformly in $N$. \end{proof}

Since the terms $\frac{|g_N|^2 + |g_{-N}|^2}{\langle N \rangle^2}$ of the mass tend to $0$ as $N \rightarrow \infty$ we expect the distribution of $P_N$ to approach $p_0$, the distribution of the mass, as $N \rightarrow \infty$. We quantify this convergence in the following lemma.

\begin{lemma} \label{607}
For each $N \ge 0$, and $\infty \ge p\ge 2$, we have 
\begin{equation}
\|P_N(x) - p_0(x)\|_{L^p(\mathbb{R})} \le  \frac{C(p)}{\langle N \rangle^2}.
\end{equation} 

\end{lemma}

\begin{proof}
By the duality properties of the Fourier transform we know that 
\begin{equation}
\|P_N(x) - p_0(x)\|_{L^p(\mathbb{R})} \le C(p) \|\widehat{P_N(\xi)} - \widehat{p_0(\xi)}\|_{L^{p'}(\mathbb{R})}, 
\end{equation}  where $p'$ is the H\"older conjugate of $p$: $\frac{1}{p} + \frac{1}{p'} = 1$.
We calculated $\widehat{P_N(\xi)}$ and  $\widehat{p_0(\xi)}$ in Lemma \ref{602} and Lemma \ref{605}. We have 
\begin{equation}
\begin{split}
|\widehat{P_N(\xi)} - \widehat{p_0(\xi)}| &= |\Pi_{|n| \ne N}\frac{1}{1 - i\xi/ \langle n \rangle^2 }  - \Pi_{n \in \mathbb{Z}}\frac{1}{1 - i\xi/ \langle n \rangle^2 }|  \\
&= |\Pi_{|n| \ne N} \left(\frac{1}{1 - i\xi/ \langle n \rangle^2 }\right) \left(1-\frac{1}{(1 - i\xi/ \langle N \rangle^2)^2}\right)|\\
&\le \Pi_{n \ne N, n \ge 0} \left(\frac{1}{1 + \xi^2/\langle n \rangle^4}\right) \left|\frac{\frac{-2i\xi}{\langle N \rangle^2} + \frac{-\xi^2}{\langle N \rangle^4}}{(1 - i\xi/ \langle N \rangle^2)^2}\right|\\
&\le \frac{\langle \xi \rangle^2}{\langle N \rangle^2} \Pi_{n \ne N, n \ge 0} \left(\frac{1}{1 + \xi^2/\langle n \rangle^4}\right) .\\
\end{split}
\end{equation}
By the same arguments as Lemma \ref{602} and Lemma \ref{605} we can bound the rest of the product by $\frac{1}{1+ \langle \xi \rangle ^4}$ and obtain a function whose $L^{p'}(\mathbb{R})$ norm is bounded by $\frac{C(p)}{\langle N \rangle^2}$.

\end{proof}

\subsection{Analysis of the scaling transformation.}  \label{mathrefs}
We now define the linear transformation that will correspond to scaling of the Fourier coefficients of our random initial data $u$. 

\begin{definition}
For any $k \in \mathbb{Z}^+$ and $s \in \mathbb{R}^+$ and function $u \in L^2(\mathbb{T}): u(x) = \sum_{n \in \mathbb{Z}} \widehat{u}(n)e^{inx}$ we define 
\begin{equation} \label{701}
T_s(u) = T^0_s(u) =  s\widehat{u}(0) + \sum_{|n| >0 } \widehat{u}(n)e^{inx}
\end{equation}
 and for $k >0$, 
\begin{equation} \label{702}
T^k_s(u)  = s\widehat{u}(k)e^{ikx} + s\widehat{u}(-k)e^{-ikx} + \sum_{|n| \ne k} \widehat{u}(n)e^{inx}.
\end{equation} 
\end{definition} 

Since $T_s^0(u)$ is defined by scaling just one coefficient, some of our analysis of it will be done separately, in which case we'll refer to it as $T_s$. Though $s$ can be any positive number, we will be taking limits as $s \rightarrow 1$, and always assume $s \in (\frac{1}{2},2)$.

For every $k$ and $s$, $T^k_s$ is a Fourier multiplier with norm equal to $\max \{1,s \}$. Clearly it takes $L^2(\mathbb{T})$ to $L^2(\mathbb{T})$ and $H^{\sigma}(\mathbb{T})$ to $H^{\sigma}(\mathbb{T})$, with norm equal to $\max \{1,s \}$.
We now define the level sets of this transformation.
\begin{definition} \label{703}
For any $s \in \mathbb{R}^+$, define $A_{m,s}$ to be the set of functions $u$ such that $$\|T_{s}(u)\|_{L^2(\mathbb{T})}^2 \le m.$$ Let $A_{m} = A_{m,1}$ be the functions with mass $\le m$. 
\end{definition}

\begin{definition} \label{704}
For any $m >0$ and $s>1 $, define $\Gamma_{m,s}$ to be the set of functions $u$ such that $\|T_s(u)\|_{L^2(\mathbb{T})}^2 > m \ge \|T_{1/s}(u)\|_{L^2(\mathbb{T})}^2$. Also let $\Gamma_{m} = \cap_{s > 1} \Gamma_{m,s}$ be the set of functions with mass $m$. 
\end{definition}

Similarly for $k \ge 1$, define $A^k_{m,s}$ to be the set of functions $u$ such that $\|T^k_{s}(u)\|_{L^2(\mathbb{T})}^2 \le m$. 
For any $s>1 $, define $\Gamma^k_{m,s}$ to be the set of functions $u$ such that $$\|T^k_s(u)\|_{L^2(\mathbb{T})}^2 > m \ge \|T^k_{1/s}(u)\|_{L^2(\mathbb{T})}^2.$$

It follows that $\Gamma^k_{m,s} = A^k_{m,1/s} \backslash A_{m,s}$. The sets $\Gamma^k_{m,s}$ will allow us to define conditional probability measures based upon the scaling $T_s^k(u)$, similar to the measures $\mu^{a,b}_{\epsilon}$ in \cite{Oh}.
\begin{definition}
Consider constants $m >0$, $2 > s >1$. For any measurable set $E \subset L^{2}(\mathbb{T})$ let 
$$\nu^0_{m,s}(E)  = \frac{\mu(E \cap \Gamma_{m,s})}{\mu(\Gamma_{m,s})}.$$ Similarly, for $k \ge 1$ let $$\nu^k_{m,s}(E)  = \frac{\mu(E \cap \Gamma^k_{m,s})}{\mu(\Gamma^k_{m,s})}.$$ 
\end{definition}

In order to reasonably bound any set with respect to this measure we will need a lower bound on $\mu(\Gamma_{m,s})$ and $\mu(\Gamma^k_{m,s})$. We not only bound the size of these sets, but show that their $\mu$-measure is roughly proportional to $\frac{s-1}{\langle k \rangle^2}$ as $ s \rightarrow 1$. 

\begin{lemma} \label{705}
For every $m >0$  $$\lim_{s \rightarrow 1+}\frac{\mu(A_{m,1/s} \backslash A_{m,s})}{s^2-1/s^2} = \lim_{s \rightarrow 1+}\frac{\mu(\Gamma_{m,s})}{s^2-1/s^2} = C(m),$$ for some constant $C(m) \ge \int_{y=0}^{m/2} \frac{m P_0(y)}{2}e^{-m} dy> \frac{me^{-m}\mu(\|u\|_{L^2(\mathbb{T})}^2 < m/2)}{2}$ depending only on $m$. 
\end{lemma}
\begin{proof}

Given $s$, the quantity in the limit is the $\mu$-measure of the set of $u = \underset{n \in \mathbb{Z}}{\sum} \frac{g_n}{\langle n \rangle}e^{inx}$ such that $$\frac{|g_0|^2}{s^2} + \sum_{|n|\ge 1} \frac{|g_n|^2}{\langle n \rangle^2} \le m < s^2|g_0|^2 + \sum_{|n|\ge 1} \frac{|g_n|^2}{\langle n \rangle^2}.$$

Recall that $ P_0$ is the distribution function of $y = \underset{|n| \ge 1}{\sum} \frac{|g_n|^2}{\langle n \rangle^2}$ with respect to the measure $\mu$. Letting $g_0 = a +ib$, with $a,b$ both $\mathcal{N}(0,\frac{1}{2})$ distributed, we have

\begin{equation} \label{706}
\begin{split}
\mu(A_{m,1/s} \backslash A_{m,s}) &=\int_{y=0}^{m}  P_0(y)\mu\left( \frac{m-y}{s^2} \le |g_0|^2 <s^2(m-y) \right) dy \\
&= \int_{y=0}^{m}  P_0(y)\int_{a^2+b^2 = (m-y)/s^2}^{s^2(m-y)} \frac{1}{\pi}e^{-(a^2+b^2)}  db da dy\\
&=  \int_{y=0}^{m} P_0(y) \int_{\theta = 0}^{2 \pi} \int_{r = \sqrt{m-y}/s}^{s\sqrt{m-y}} \frac{1}{\pi} e^{-r^2} rdrd\theta dy\\ 
&=  \int_{y=0}^{m} P_0(y) |_{r = \sqrt{m-y}/s}^{s\sqrt{m-y}}  -e^{-r^2} dy\\
&=   \int_{y=0}^{m}  P_0(y)  \left(e^{-(m-y)/s^2} - e^{-s^2(m-y)}\right) dy,\\ 
\end{split}
\end{equation}
after switching to polar coordinates in the middle. Dividing yields
\begin{equation}\label{707}
\frac{\mu(A_{m,1/s} \backslash A_{m,s})}{s^2-1/s^2} =  \frac{1}{s^2-1/s^2} \int_{y=0}^{m}  P_0(y)  \left(e^{-(m-y)/s^2} - e^{-s^2(m-y)}\right) dy.
\end{equation}
Taking the limit as $s \rightarrow 1$ produces the derivative $dr$ of the function $$\int_{y=0}^{m} - P_0(y) e^{-r(m-y)} dy$$ at the point $r=1$, which equals 
\begin{equation} \label{708}
\int_{y=0}^{m}  P_0(y) (m-y)e^{-(m-y)} dy \ge \int_{y=0}^{m/2}  P_0(y) \frac{m}{2}e^{-m} dy > \frac{me^{-m}\mu(\|u\|_{L^2(\mathbb{T})}^2 < m/2)}{2}.
\end{equation}
\end{proof}
We now prove the same result for $ k \ge 1$.

\begin{lemma} \label{709}
For every $m >0$ and $k \ge 1$  $$\lim_{s \rightarrow 1+}\frac{\mu(A^k_{m,1/s} \backslash A^k_{m,s})}{s^2-1/s^2} = \lim_{s \rightarrow 1+}\frac{\mu(\Gamma^k_{m,s})}{s^2-1/s^2} = C(m,k),$$ for some constant $C(m,k)$ that satisfies
\begin{equation}
\begin{split}
\frac{C}{\langle k \rangle^2} \ge C(m,k) &>  \frac{m^2e^{-m}}{4\langle k \rangle ^2}\int_{m/2}^{m} P_k\left( m-\frac{y}{\langle k \rangle^2}\right)  dy \\
\end{split}
\end{equation} 
\end{lemma}

\begin{proof}

Given s, the quantity in the limit is the measure of the set of $u = \underset{n \in \mathbb{Z}}{\sum} \frac{g_n}{\langle n \rangle}e^{inx}$ such that $$\frac{|g_k|^2+|g_{-k}|^2}{s^2\langle k \rangle ^2} + \sum_{|n|\ne k} \frac{|g_n|^2}{\langle n \rangle^2} \le m < \frac{s^2(|g_k|^2+ |g_{-k}|^2)}{\langle k \rangle^2} + \sum_{|n|\ne k} \frac{|g_n|^2}{\langle n \rangle^2}.$$

Recall that $P_k(y)$ is the distribution function of $y = \underset{|n| \ne k}{\sum} \frac{|g_n|^2}{\langle n \rangle^2}$ with respect to the measure $\mu$. Setting $g_k = a +ib$, $g_{-k} = a'+ib'$, with $a,a',b,b'$ all $\mathcal{N}(0,\frac{1}{2})$ distributed, we have 

\begin{equation}
\begin{split}
\mu(A^k_{m,1/s} \backslash A^k_{m,s}) &=\int_{y=0}^{m} P_k(y)\mu\left( \frac{m-y}{s^2} < \frac{(|g_k|^2+ |g_{-k}|^2)}{\langle k \rangle^2} \le s^2(m-y)\right) dy \\
&=\int_{y=0}^{m} P_k(y)\mu\left( \frac{\langle k \rangle^2(m-y)}{s^2} < |g_k|^2+ |g_{-k}|^2  \le \langle k \rangle^2s^2(m-y)\right) dy \\
&= \int_{y=0}^{m} P_k(y)\int_{a^2+a'^2+b^2+b'^2= \frac{\langle k \rangle^2(m-y)}{s^2}}^{\langle k \rangle^2s^2(m-y)}  \frac{e^{-(a^2+a'^2+b^2+b'^2)}}{\pi^2} da db da'db'  dy\\
&=  \int_{y=0}^{m} P_k(y)\int_{r_1^2+r_2^2= \frac{\langle k \rangle^2(m-y)}{s^2}}^{\langle k \rangle^2s^2(m-y)}  4r_1r_2e^{-(r_1^2+r_2^2)}dr_1 dr_2  dy\\
&=  \int_{y=0}^{m} P_k(y) \int_{\theta = 0}^{\pi/2} \int_{r = \langle k \rangle \sqrt{m-y}/s}^{s\langle k \rangle \sqrt{m-y}} 4r^2\cos(\theta)\sin(\theta) e^{-r^2} rdrd\theta dy\\ 
&=  \int_{y=0}^{m} P_k(y) \int_{r = \langle k \rangle \sqrt{m-y}/s}^{s\langle k \rangle \sqrt{m-y}} 2r^3 e^{-r^2} dr dy\\ 
&=  \int_{y=0}^{m} P_k(y) |_{r =  \langle k \rangle \sqrt{m-y}/s}^{s \langle k \rangle\sqrt{m-y}}  -(r^2+1)e^{-r^2} dr dy\\
&=   \int_{y=0}^{m} P_k(y)  (1+ \langle k \rangle^2\frac{m-y}{s^2})e^{- \langle k \rangle^2(m-y)/s^2} dy \\
&-\int_{y=0}^m P_k(y)(1 +  \langle k \rangle^2s^2(m-y) )e^{- \langle k \rangle^2s^2(m-y)} dy\\ 
\end{split}
\end{equation}
where we switched to polar coordinates twice. Proceeding as in the proof of Lemma \ref{705}, and applying a change of variables, we have 
\begin{equation}
\begin{split}
\lim_{s \rightarrow 1} \frac{\mu(A^k_{m,1/s} \backslash A^k_{m,s})}{s^2-1/s^2}&= \frac{d}{dr}|_{r=1} \int_{y=0}^{m} -P_k(y)(1+(m-y)\langle k \rangle^2 r)e^{-r\langle k \rangle^2(m-y)} dy \\
&= \int_{y=0}^{m} \langle k \rangle^2 P_k(y) [(m-y+\langle k \rangle^2(m-y)^2) - (m-y)]e^{-r \langle k \rangle^2(m-y)} dy \\
&= \int_{y=0}^{m} P_k(y) \langle k \rangle^4(m-y)^2e^{- \langle k \rangle^2(m-y)} dy.\\
&= \int_{y=0}^{m} P_k(m-y) \langle k \rangle^4y^2e^{- \langle k \rangle^2y} dy.\\
&= \frac{1}{\langle k \rangle^2}\int_{y=0}^{\langle k \rangle^2m} P_k\left( m-\frac{y}{\langle k \rangle^2} \right)y^2e^{-y}dy.\\ 
\end{split}
\end{equation}

This limit satisfies 
\begin{equation} \label{730}
\begin{split}
\lim_{s \rightarrow 1} \frac{\mu(A^k_{m,1/s} \backslash A^k_{m,s})}{s^2-1/s^2}  &\le \frac{1}{\langle k \rangle^2} \int_0^{\infty} \|P_k\|_{L^{\infty}(\mathbb{R})} y^2e^{-y}\\
&\le \frac{C}{\langle k \rangle^2} \\
\end{split}
\end{equation}
as well as 
\begin{equation} \label{729}
\begin{split}
\lim_{s \rightarrow 1} \frac{\mu(A^k_{m,1/s} \backslash A^k_{m,s})}{s^2-1/s^2} &\ge 
 \frac{m^2e^{-m}}{4\langle k \rangle ^2}\int_{m/2}^{m} P_k\left( m-\frac{y}{\langle k \rangle^2}\right)  dy.\\
\end{split}
\end{equation}
\end{proof}

We now prove analogous results for the measure of the set of functions with mass in the ball of radius $\epsilon$ centered at $m$. We start with the following definition and measure:
\begin{definition}
Let $B_{m,\epsilon}$ denote the set of functions $u \in L^2(\mathbb{T})$ with mass in $(m-\epsilon,m+\epsilon)$. For any measurable set $E \subset L^{2}(\mathbb{T})$ we define the measure $$\mu^{\epsilon}_{m}(E) = \frac{\mu(E \cap B_{m,\epsilon})}{\mu(B_{m,\epsilon})}.$$
\end{definition}
In order to reasonably bound $\mu_m^{\epsilon}(E)$ for any $E$, we will need a lower bound on the measure of $B_{m,\epsilon}$.

\begin{lemma} \label{710}
For any $m>0$ we have $$\lim_{\epsilon \rightarrow 0+}\frac{\mu(m-\epsilon < \|u\|_{L^2(\mathbb{T})}^2 < m+ \epsilon)}{2\epsilon} = \lim_{\epsilon \rightarrow 0+}\frac{\mu(B_{m,\epsilon})}{2\epsilon} = p_0(m)$$ with $p_0(m) >0$.
\end{lemma}

\begin{proof} 
Recall from Lemma \ref{602} that $p_0$, the distribution function of the mass, is uniformly continuous with $\|p_0\|_{L^{\infty}} < \infty$. Clearly \begin{equation}
\mu(m-\epsilon < \|u\|_{L^2(\mathbb{T})}^2 < m+ \epsilon) = \int_{m-\epsilon}^{m+\epsilon} p_0(m')dm'.
\end{equation} 
Uniformly continuity of $p_0$ implies that 
\begin{equation} 
\lim_{\epsilon \rightarrow 0+}\frac{\mu(m-\epsilon < \|u\|_{L^2(\mathbb{T})}^2 < m+ \epsilon)}{2\epsilon} = p_0(m).
\end{equation}

It remains to show that $p_0(m) >0$. We have 
\begin{equation*}
\begin{split}
\mu(m-\epsilon < \|u\|_{L^2(\mathbb{T})}^2 < m+ \epsilon) &=\int_{y=0}^{m+\epsilon} p_1(y) \mu( m-\epsilon-y < |g_0|^2 <m+\epsilon-y) dy \\
&=  \int_{y=0}^{m+\epsilon} p_1(y) \int_{\theta = 0}^{2 \pi} \int_{r = \sqrt{\max\{m-\epsilon-y, 0\}}}^{\sqrt{m+\epsilon-y}} \frac{1}{\pi} e^{-r^2} rdrd\theta dy\\ 
&=  \int_{y=0}^{m+\epsilon} p_1(y) \int_{r = \sqrt{\max\{m-\epsilon-y, 0\}}}^{\sqrt{m+\epsilon-y}}  2re^{-r^2} dr dy\\
&\ge   \int_{y=0}^{m-\epsilon} p_1(y) e^{-(m-y)}(e^{\epsilon} - e^{-\epsilon}) dy \\
&= \int_{y=0}^{m}  p_1(y) e^{-(m-y)}(e^{\epsilon} - e^{-\epsilon}) dy -  \int_{y=m-\epsilon}^{m}  p_1(y) e^{-(m-y)}(e^{\epsilon} - e^{-\epsilon}) dy. \\
\end{split}
\end{equation*}
The limit of the second integrand as $\epsilon \rightarrow 0$ is clearly $0$, and thus $$\lim_{\epsilon \rightarrow 0} \frac{1}{2\epsilon}\int_{y=m-\epsilon}^{m}  p_1(y) e^{-(m-y)}(e^{\epsilon} - e^{-\epsilon}) dy = 0.$$ Dividing the first integral by $2 \epsilon$ and taking the limit, we have  
\begin{equation}
\begin{split}
\lim_{\epsilon \rightarrow 0+}\frac{\mu(m-\epsilon < \|u\|_{L^2(\mathbb{T})}^2 < m+ \epsilon)}{2\epsilon} &\ge \frac{d}{dr}|_{r=0} \int_{y=0}^{m} p_1(y) e^{-(m-y)}e^{r}  dy \\
&= \int_{y=0}^{m} p_1(y) e^{-(m-y)}  dy\\
\end{split}
\end{equation}
So we have 
\begin{equation}
p_0(m) =  \lim_{\epsilon \rightarrow 0+}\frac{\mu(m-\epsilon < \|u\|_{L^2(\mathbb{T})}^2 < m+ \epsilon)}{2\epsilon} >  \int_{y=0}^{m} p_1(y) e^{-(m-y)}  dy>0.
\end{equation}
\end{proof}

We conclude this subsection with a bound on the functions $P_N$ that will be useful when combined with Lemma \ref{709}. 
\begin{lemma} \label{609}
For each $m>0$ there exists an interval $[m',m'']$ containing $m$ in its interior and a continuous positive function $C_0(x)$ such that for each $x \in [m',m'']$
\begin{equation}
\inf_{N} \int_{x/2}^x  P_N\left( x-\frac{y}{\langle N \rangle^2} \right) dy \ge C_0(x).
\end{equation} 
\end{lemma}
\begin{proof}
Fix a value of $N \ge 0$. Applying Lemma \ref{607} and the triangle inequality, we have 
\begin{equation}
\begin{split}
\int_{m/2}^m  &P_N \left( m-\frac{y}{\langle N \rangle^2} \right) dy\\
&\ge \int_{m/2}^m  p_0\left( m-\frac{y}{\langle N \rangle^2} \right) - \left| P_N\left( m-\frac{y}{\langle N \rangle^2} \right) -  p_0\left( m-\frac{y}{\langle N \rangle^2} \right)\right| dy \\
&\ge \int_{m/2}^m   p_0(m)  - \left| p_0\left( m-\frac{y}{\langle N \rangle^2} \right)- p_0(m)\right|- \left| P_N\left( m-\frac{y}{\langle N \rangle^2} \right) -  p_0\left( m-\frac{y}{\langle N \rangle^2} \right) \right| dy \\ 
&\ge \frac{m}{2}\left(p_0(m) - \frac{C}{\langle N \rangle ^2} - \sup_{y \in [m/2,m]}\left| p_0\left( m-\frac{y}{\langle N \rangle^2} \right)- p_0(m)\right|\right) \\ 
\end{split}
\end{equation} 
Since $p_0$ is uniformly continuous, $\sup_{y \in [m/2,m]}| p_0\left( m-\frac{y}{\langle N \rangle^2} \right)- p_0(m)| \le G(N)$ for some positive decreasing function $G$ that satisfies $\lim_{N \rightarrow \infty} G(N) = 0$. Therefore for each $m,N$ we have
\begin{equation}
\int_{m/2}^m  P_N\left( m-\frac{y}{\langle N \rangle^2} \right) dy \ge \frac{m}{2}(p_0(m) - \frac{C}{\langle N \rangle ^2} - G(N)).
\end{equation} 
Now fix a value of $m>0$. There exists a minimal value of $N$, which we denote $N_m$, that satisfies $p_0(m) > \frac{C}{\langle N_m \rangle^2} + G(N_m)$. This means that for all $N \ge N_m$ we have 
\begin{equation}
p_0(m) > \frac{C}{\langle N \rangle^2} + G(N).
\end{equation}

Since $p_0$ is a continuous function, there exists an interval $[m',m'']$ containing $m$ in its interior such that $p_0(x) > \frac{C}{\langle N_m \rangle^2} + G(N_m)$ for each $x \in [m',m'']$. Now observe that $$\{\frac{x}{2}(p_0(x) - \frac{C}{\langle N_m \rangle ^2} - G(N_m))\} \cup\{ \int_{x/2}^x  P_N\left( x-\frac{y}{\langle N \rangle^2} \right) dy, N \le N_m-1\}$$ is a finite set of $N_m+1$ continuous functions on the interval $[m',m'']$. Setting $C_0(x)$ equal to 
\begin{equation}
\min\{\int_{x/2}^x  P_0\left( x-\frac{y}{\langle 0 \rangle^2} \right) dy, \ldots, \int_{x/2}^x  P_{N_m-1}\left( x-\frac{y}{\langle N_m-1 \rangle^2} \right) dy,\frac{m}{2}(p_0(m) - \frac{C}{\langle N_m \rangle^2} - G(N_m))  \}
\end{equation}
defines a continuous function such that for each $x \in [m',m'']$, we have 
\begin{equation}
\int_{x/2}^x  P_N\left( x-\frac{y}{\langle N \rangle^2} \right) dy \ge C_0(x)
\end{equation}
for all $N \ge 0$.
\end{proof}

\subsection{The limit measures $\mu_m$ and $\nu_m^k$} \label{mathrefs}

In this subsection we demonstrate that the sequences $\mu_m^{\epsilon}$ and $\nu_{m,s}^k$, both weakly converge to measures $\mu_m$, $\nu_{m}^k$ that are supported on the set $\Gamma_m$ and have a simple relation to each other. We will find that these measures satisfy the Radon-Nikodym derivative 
\begin{align} 
 c_{k,m} d\nu^k_m &= (|g_k|^2+|g_{-k}|^2)d\mu_m \label{801} \\ 
 c_{0,m} d\nu_m^0 &= |g_0|^2 d\mu_m \label{802}
\end{align}
for constants $c_{k,m}$ clearly equal to $\mathbb{E}_{\mu_m}(|g_k|^2 + |g_{-k}|^2)$. If we can show that these constants $c_{k,m}$ are uniformly bounded 
then summing these together we formally have the relation 
\begin{equation} \label{804} 
 \sum_{k \ge 0} \frac{c_{k,m}}{\langle k \rangle^2}d\nu_m^k = \left( \sum_{n} \frac{|g_n|^2}{\langle n \rangle ^2} \right) d\mu_m = m \cdot d\mu_m.
\end{equation}

In order to bound $\int_{H^{1/2-}(\mathbb{T})} F(u)d\mu_m$ it suffices to bound  $\int_{H^{1/2-}(\mathbb{T})} F(u)d\nu_m^k$, uniformly in $k$. This is how we will show the energy term $e^{f(u)-\frac{1}{2}\int |u|^6dx}$ is in $L^1(\mu_m)$, which will allow us to construct an invariant measure while only scaling one pair of frequencies at a time. 

We now define the measures $\mu_m$ and $\nu_m^k$. The limit definition of $\mu_m$ is the same as in Oh, Quastel \cite{Oh}. Recall that the Borel subsets of $H^{\sigma}(\mathbb{T})$ that depend only on finitely many frequencies generate the sigma algebra of $H^{\sigma}(\mathbb{T})$ in the weak topology.

\begin{definition} \label{805}
Let $E\subset L^{2}(\mathbb{T})$ be a measurable set depending only on the frequencies $\le N$. We define $\mu_{m}(E) = \lim_{\epsilon \rightarrow 0}  \mu_{m}^{\epsilon}(E)$. This defines the measure $\mu_m$ on the whole sigma algebra provided that this limit exists for each $E$. 
\end{definition}

Consider a positive integer $N$ and a set $E \subset L^{2}(\mathbb{T})$ generated only by the values of $g_n$, for $n \le N$. Once they have been specified we denote the real and imaginary parts of $g_n$ by $x_n, x_n'$ respectively. Let $E'$ be the set of possible values of $g_0, g_1, g_{-1}, \ldots$ to produce an element of $E$, and let $y = \underset{|n| \le N}{\sum}\frac{|g_n|^2}{\langle n\rangle^2}$. Recall that $p_N(x)$ denotes the distribution function of $\underset{|n| \ge N}{\sum}\frac{|g_n|^2}{\langle n\rangle^2}$.

\begin{lemma} \label{806}
For each measurable set $E$ the limit $\lim_{\epsilon \rightarrow 0}  \mu_{m}^{\epsilon}(E)$ exists and equals 
\begin{equation} \label{807}
\lim_{\epsilon \rightarrow 0}  \mu_{m}^{\epsilon}(E) = \frac{1}{p_0(m)} \int_{E'} p_{N+1}(m-y)  \Pi_{|n| \le N}\frac{e^{-(|x_n|^2+|x_n'|^2)}}{\pi}   dx_ndx_n'.
\end{equation}
\end{lemma}
\begin{proof}
 We have, after taking the limit and applying Lemma \ref{710},
\begin{equation*}
\begin{split}
\mu_{m}(E) &= \lim_{\epsilon \rightarrow 0} \mu^{\epsilon}_{m}(E)\\
&= \lim_{\epsilon \rightarrow 0} \int_{E'}  \frac{\mu( m-\epsilon -y <\sum_{|n|>N} \frac{|g_n|^2}{\langle n\rangle^2} < m+\epsilon -y )}{\mu(m-\epsilon <\sum_n \frac{|g_n|^2}{\langle  n\rangle^2} < m+\epsilon )} \Pi_{|n| \le N}\frac{e^{-(|x_n|^2+|x_n'|^2)}}{\pi}   dx_ndx_n'\\
&= \lim_{\epsilon \rightarrow 0} \int_{E'}  \frac{\mu( m-\epsilon -y <\sum_{|n|>N} \frac{|g_n|^2}{\langle n\rangle^2} < m+\epsilon -y )/2\epsilon}{\mu(m-\epsilon <\sum_n \frac{|g_n|^2}{\langle  n\rangle^2} < m+\epsilon )/2\epsilon} \Pi_{|n| \le N}\frac{e^{-(|x_n|^2+|x_n'|^2)}}{\pi}   dx_ndx_n'\\
&=  \frac{1}{p_0(m)} \int_{E'} p_{N+1}(m-y)  \Pi_{|n| \le N}\frac{e^{-(|x_n|^2+|x_n'|^2)}}{\pi}   dx_ndx_n'.\\
\end{split}
\end{equation*}

The last equality follows from pointwise convergence of the integrand, however, this  equality is contingent upon the integrand satisfying the boundedness condition of the dominated convergence theorem, which we will now prove. 

Clearly for all $\epsilon$ we have 
$$\mu\left(  m-\epsilon -y <\sum_{|n|>N} \frac{|g_n|^2}{\langle n\rangle^2} < m+\epsilon -y \right)/2\epsilon \le \|p_{N+1}\|_{L^{\infty}(\mathbb{R})}.$$
In addition, for sufficiently small $\epsilon$ we have $$\mu\left( m-\epsilon <\sum_n \frac{|g_n|^2}{\langle  n\rangle^2} < m+\epsilon \right)/2\epsilon > p_0(m)/2 >0,$$ since the limit of this quantity is $p_0(m)$.

Therefore for sufficiently small $\epsilon$ the integral is pointwise bounded by $$\frac{2\|p_{N+1}\|_{L^{\infty}(\mathbb{R})}}{p_0(m)}\Pi_{|n| \le N}\frac{e^{-(|x_n|^2+|x_n'|^2)}}{\pi}.$$ This function is clearly integrable over the set of possible values $x_n, x'_n$. \end{proof}

We now apply the same limit argument to construct the measure $\nu_m^k$ and prove an  equation analogous to \eqref{807}. 
\begin{definition} \label{808}
For any set $E \subset L^{2}(\mathbb{T})$ depending only on the frequencies $\le N$, we define $\nu^0_{m}(E) = \lim_{s \rightarrow 1^+}  \nu^0_{m,s}(E)$. This defines the measure on the whole sigma algebra. 
\end{definition}

\begin{lemma} \label{809}
Keeping the notation of Lemma \ref{806}, for each measurable set $E$ the limit \newline $\lim_{s \rightarrow 1}  \nu_{m,s}^{0}(E)$ exists and, for some constant $C(m)> 0$ depending only on $m$, equals \begin{equation} \label{810}
\lim_{s \rightarrow 1^+}  \nu_{m,s}^{0}(E) = \frac{1}{C(m)}\int_{E'} p_{N+1}(m-y)(|x_0|^2 + |x_0'|^2)  \Pi_{|n| \le N}\frac{e^{-(|x_n|^2+|x_n'|^2)}}{\pi}   dx_ndx_n'.
\end{equation}
\end{lemma}
\begin{proof}
Letting $y_1 = \underset{1 \le |n| \le N}{\sum}\frac{|g_n|^2}{\langle n\rangle^2}$, we have
\begin{equation*}
\begin{split}
\nu^0_{m}(E) &= \lim_{s \rightarrow 1} \nu^0_{m,s}(E)\\
= \lim_{s \rightarrow 1^+} &\int_{E'}  \frac{\mu\left( \sum_{|n|>N} \frac{|g_n|^2}{\langle n\rangle^2} \in (m - s^2(|x_0|^2 + |x_0'|^2)-y_1, m - (|x_0|^2 + |x_0'|^2)/s^2-y_1)\right)}{\mu(m-s^2|g_0|^2 <\sum_{|n| \ge 1} \frac{|g_n|^2}{\langle  n\rangle^2} < m - |g_0|^2/s^2 )}\\
&\times \Pi_{|n| \le N}\frac{e^{-(|x_n|^2+|x_n'|^2)}}{\pi}   dx_ndx_n'\\
=  \lim_{s \rightarrow 1^+} &\int_{E'}  \frac{\mu\left( \sum_{|n|>N} \frac{|g_n|^2}{\langle n\rangle^2}-m+y_1 \in (- s^2(|x_0|^2 + |x_0'|^2), - (|x_0|^2 + |x_0'|^2)/s^2)\right)/(s^2-1/s^2)}{\mu(m-s^2|g_0|^2 <\sum_{|n| \ge 1} \frac{|g_n|^2}{\langle  n\rangle^2} < m - |g_0|^2/s^2)/(s^2-1/s^2)}\\
&\times \Pi_{|n| \le N}\frac{e^{-(|x_n|^2+|x_n'|^2)}}{\pi}   dx_ndx_n'\\
\nu^0_{m}(E) &=  \frac{1}{C(m)}\int_{E'} p_{N+1}(m-y)(|x_0|^2 + |x_0'|^2)  \Pi_{|n| \le N}\frac{e^{-(|x_n|^2+|x_n'|^2)}}{\pi}   dx_ndx_n'.\\
\end{split}
\end{equation*}
As in the previous lemma, the final equality follows from the integrand satisfying the boundedness condition of the dominated convergence theorem. For all $s \approx 1$ we have $$\frac{\mu\left( \sum_{|n|>N} \frac{|g_n|^2}{\langle n\rangle^2}-m+y_1 \in (- s^2(|x_0|^2 + |x_0'|^2), - (|x_0|^2 + |x_0'|^2)/s^2)\right)}{s^2-1/s^2} \le (|x_0|^2 + |x_0'|^2)\|p_{N+1}\|_{L^{\infty}(\mathbb{R})}.$$ In addition, for sufficiently small $\epsilon$ we have $$\mu\left( m-\epsilon <\sum_n \frac{|g_n|^2}{\langle  n\rangle^2} < m+\epsilon \right)/2\epsilon > C(m)/2 >0,$$ since the limit of this quantity is $C(m)$ for some constant function of $m$.

Therefore the integral is pointwise bounded by $$\frac{2(|x_0|^2 + |x_0'|^2)\|p_N\|_{L^{\infty}(\mathbb{R})}}{C(m)}\Pi_{|n| \le N}\frac{e^{-(|x_n|^2+|x_n'|^2)}}{\pi}.$$ This function is clearly integrable over the set of possible values $x_n, x'_n$. \end{proof}

The above formula leads to the following Radon-Nikodym Derivative:

\begin{corollary} \label{811}
For each $m >0$ we have $$c_{0,m}d\nu^0_{m} = |g_0|^2 d\mu_m $$ for the constant $c_{0,m} = \int_{H^{1/2-}(\mathbb{T})} |g_0|^2 d\mu_m$.
\end{corollary}
\begin{proof} By equations \eqref{807} and \eqref{810}, $c_{0,m}d\nu^0_{m} = |g_0|^2 d\mu_m $ must hold for some constant, $c_{0,m}$. Then integrating yields the formula for $c_{0,m}$. \end{proof}

We now do the same thing for $k \ge 1$.

\begin{definition} \label{812}
For any set $E\subset L^{2}(\mathbb{T})$ depending only on the frequencies $\le N$, we define $\nu^k_{m}(E) = \lim_{s \rightarrow 1^+}  \nu^k_{m,s}(E)$. This defines the measure on the whole sigma algebra. 
\end{definition}

\begin{lemma} \label{813}
Keeping the notation of Lemma \ref{809}, for each measurable set $E$, the limit \newline $\lim_{s \rightarrow 1^+}  \mu_{m,s}^{k}(E)$ exists and equals \begin{equation*} \label{814}
\lim_{s \rightarrow 1^+}  \nu_{m,s}^{k}(E) = \int_{E'} \frac{p_{N+1}(m-y)(|x_k|^2 + |x_k'|^2+ |x_{-k}|^2 + |x_{-k}'|^2)}{C(m)}  \Pi_{|n| \le N}\frac{e^{-(|x_n|^2+|x_n'|^2)}}{\pi}   dx_ndx_n'.
\end{equation*}
\end{lemma}
The proof is nearly identical to the proof of Lemma \ref{809}. We just need to specify that the frequency of $E$ satisfies  $N \ge k$, and note there is an extra factor of $\langle k \rangle ^2$ that gets absorbed into the constant $C(m)$. 

We arrive at the following corollary:
\begin{corollary} \label{815}
For each $m >0$ we have $$c_{k,m}d\nu^k_{m} = (|g_k|^2 +|g_{-k}|^2)d\mu_m $$ for the constant $c_{k,m} = \int_{H^{1/2-}(\mathbb{T})} |g_k|^2 +|g_{-k}|^2 d\mu_m$.
\end{corollary}

We want to apply Corollaries \ref{811} and \ref{815} to write the measure $m \cdot \mu_m$ as the sum $\underset{k \ge 0}{\sum} c_{k,m}\nu^k_{m}$. First we will need to bound $\int_{H^{1/2-}(\mathbb{T})} |g_k|^2 d\mu_m$ for each pair $m,k$. 
\begin{lemma} \label{816} Consider a mass $m >0$ and exponent $p \ge 1$.

1) There exists a constant $C(p)$ such that for any $n \in \mathbb{Z}$ and sufficiently small $\epsilon >0$ we have $\int |g_n|^p d\mu_{m}^{\epsilon} \le \frac{C(p)}{p_0(m)}$. The constant $C(p)$ does not depend on $m$ or $n$.

2) For the same $m,p,C$ we have $\int |g_n|^p d\mu_{m}  = \lim_{\epsilon \rightarrow 0} \int |g_n|^p d\mu_{m}^{\epsilon}  \le  \frac{C(p)}{p_0(m)}$. 
\end{lemma}
Observe that setting $p=2$ and applying Corollary \ref{815} we obtain $c_{k,m} \le C(m)$ uniformly in $k \ge 0$. 

\begin{proof}

We first prove that the first part  implies the second part. It suffices to show that for each $n$, $\int |g_n|^p d\mu_{m}  = \lim_{\epsilon \rightarrow 0} \int |g_n|^p d\mu_{m}^{\epsilon}$.

By the properties of the Lebesgue integral, this is equivalent to showing that 
\begin{equation} \label{817}
\int_0^{\infty} \mu_m(|g_n|^p >\lambda) d\lambda = \lim_{\epsilon \rightarrow 0} \int_0^{\infty} \mu_m^{\epsilon}(|g_n|^p >\lambda) d\lambda.
\end{equation}
We seek to apply the dominated convergence theorem. By the construction of the measure $\mu_m$, the integrand converges pointwise and for $\epsilon < 1$, the integrand is dominated by the integrable function $1_{\lambda \le (m+1)\langle n \rangle ^p}$, by the mass bound. Therefore the limit equals the integral.

To prove the first part, observe that $g_n = a+bi$ for real and imaginary parts $a,b$ with $\mathcal{N}(0,\frac{1}{2})$ distributions. Recall  from Lemma \ref{605}   that $P_n$, the probability distribution function of the random variable $\sum_{n' \ne n} \frac{|g_{n'}|^2}{\langle n' \rangle^2}$, has $L^{\infty}(\mathbb{R})$ norm bounded uniformly in $n$. Lemma \ref{710} implies that 
\begin{equation}\mu(m-\epsilon < \|u\|_{L^2(\mathbb{T})}^2 < m+ \epsilon) \ge p_0(m)\epsilon,
\end{equation} for sufficiently small $\epsilon$. We have the following bound on the integral:

\begin{equation*} \label{818}
\begin{split}
\int &|g_n|^p d\mu_{m}^{\epsilon} = \mu\left( m - \epsilon < \sum_{n'} \frac{|g_{n'}|^2}{\langle n' \rangle^2} < m + \epsilon \right)^{-1}\times \\
&\int\limits_{a^2+b^2 \le m+\epsilon} (a^2+b^2)^{p/2}\frac{e^{-(a^2+b^2)}}{\pi} \mu(m-\epsilon - \frac{a^2+b^2}{\langle n \rangle^2} <\sum_{n' \ne n} \frac{|g_{n'}|^2}{\langle m \rangle^2} < m+\epsilon - \frac{a^2+b^2}{\langle n \rangle^2}) db da\\
&\le \frac{1}{\epsilon p_0(m)}\int_{r=0}^{\sqrt{m+ \epsilon}} 2r^{p+1}e^{-r^2}  \mu(m-\epsilon - \frac{r^2}{\langle n \rangle^2} <\sum_{m \ne n} \frac{|g_m|^2}{\langle m \rangle^2} < m+\epsilon - \frac{r^2}{\langle n \rangle^2}) dr\\
&\le \frac{4}{p_0(m)} \|P_n\|_{L^{\infty}}\int_{r=0}^{\sqrt{m+ \epsilon}} r^{p+1}e^{-r^2} dr\\
&\le \frac{4}{p_0(m)} \|P_n\|_{L^{\infty}}\int_{r=0}^{\infty} r^{p+1}e^{-r^2} dr\\
&\lesssim \frac{\|P_n\|_{L^{\infty}}}{p_0(m)} \\
\end{split}
\end{equation*}
where the constant of the inequality depends on $p$, but not on $n$. Since $\|P_n\|_{L^{\infty}}$ is uniformly bounded, we conclude that the integral is $\le \frac{C(p)}{p_0(m)}$. \end{proof} 

Before our main result we will prove a bound that is similar to Lemma \ref{816} that will help us bound the expected value of $L^p$ norms of $u$ with respect to $\mu_m$. 

\begin{lemma} \label{819}
For any $m >0$, $p\ge 1$, and $\alpha >1$ we have 
\begin{equation} \label{820}
\mathbb{E}_{\mu_m}\left[ \left( \sum_{n} \frac{|g_n|^2}{\langle n \rangle ^{\alpha}} \right)^p \right] \le C(\alpha,p,m)
\end{equation}
and 
\begin{equation} \label{820}
\mathbb{E}_{\mu_m}\left[ \left( \sum_{n \ge N} \frac{|g_n|^2}{\langle n \rangle ^{\alpha}} \right)^p \right] \le \frac{C(\alpha,p,m)}{N^{(\alpha-1)p}}.
\end{equation} The same result holds for sufficiently small $\epsilon >0$ if we replace $\mathbb{E}_{\mu_m}$ with $\mathbb{E}_{\mu^{\epsilon}_m}$.
\end{lemma}

\begin{proof} Assume that $p$ is an integer. We can then use H\"older's inequality to extend to the case when $p$ is non-integer. We expand the sum, apply the power means inequality, and apply Lemma \ref{816} to obtain 
\begin{equation} \label{821}
\begin{split} 
\mathbb{E}_{\mu_m}\left[ \left( \sum_{n} \frac{|g_n|^2}{\langle n \rangle ^{\alpha}} \right)^p \right] &\le \mathbb{E}_{\mu_m} \left( \sum_{n_1}\sum_{n_2} \cdots \sum_{n_p} \Pi_{1 \le i \le p}\frac{|g_{n_i}|^2}{\langle n_i \rangle ^{\alpha}} \right) \\
&\le \sum_{n_1}\sum_{n_2} \ldots \sum_{n_p}  \frac{\sum_{i=1}^p \mathbb{E}_{\mu_m}(|g_{n_i}|^{2p})}{p \Pi_{1 \le i \le p}\langle n_i \rangle ^{\alpha}} \\
&\le \sum_{n_1}\sum_{n_2} \ldots \sum_{n_p} \frac{C(2p)}{ p_0(m) \Pi_{1 \le i \le p}\langle n_i \rangle ^{\alpha}} \\
&\le \frac{C(2p)}{\cdot p_0(m)}\Pi_{i=1}^p \sum_{n_i} \frac{1}{\langle n_i \rangle ^{\alpha}}\\
&\le C(\alpha,p,m). \\
\end{split}
\end{equation}
Observe that our upper bound for $\mathbb{E}_{\mu_m}(|g_{n_i}|^{2p})$ also applies to $\mathbb{E}_{\mu^{\epsilon}_m}(|g_{n_i}|^{2p})$ by the first part of Lemma \ref{816}.
In addition, if the sum is over the set of $|n_i| \ge N$ we have 
\begin{equation}
\begin{split}
\mathbb{E}_{\mu_m}\left[ \left( \sum_{|n| \ge N} \frac{|g_n|^2}{\langle n \rangle ^{\alpha}}\right)^p \right] &\le \frac{C(2p)}{p_0(m)}\Pi_{i=1}^p \sum_{|n_i| \ge N} \frac{1}{\langle n_i \rangle ^{\alpha}} \\
&\le \frac{C(\alpha,p,m)}{N^{(\alpha-1)p}}. \\
\end{split}
\end{equation}
\end{proof}
This leads to the following proposition on $L^p(\mu_m)$ norms. 
\begin{proposition} \label{829}
For any $m >0$ and $p \ge 2$, we have 
\begin{equation}
\mathbb{E}_{\mu_m} \left( \int_{\mathbb{T}} |u|^p dx \right ) \le C(m,p)
\end{equation} and 
\begin{equation}
\mathbb{E}_{\mu_m} \left( \int_{\mathbb{T}} |u-u_N|^p dx \right) \le \frac{C(m,p)}{N}.
\end{equation} 
\end{proposition} 
\begin{proof} This is equivalent to bounding $\mathbb{E}_{\mu_m} \left( \|u\|_{L^p(\mathbb{T})}^p \right) $. Let $\alpha = \frac{1}{2} - \frac{1}{p}$ and note that $0< \alpha < \frac{1}{2}$. By Sobolev embeddings, we have $\|u\|_{L^p(\mathbb{T})} \lesssim \|u\|_{H^{\alpha}(\mathbb{T})}$. So we seek to bound 
\begin{equation}\label{822}
\mathbb{E}_{\mu_m} \left( \|u\|^p_{H^{\alpha}(\mathbb{T})} \right) =  \mathbb{E}_{\mu_m} \left( \sum_{n} \frac{|g_n|^2}{\langle n \rangle ^{2- 2\alpha}} \right)^{p/2}.
\end{equation}
Since $\alpha < \frac{1}{2}$, this sum satisfies the conditions of Lemma \ref{819} and is bounded by $C(m,p)$. In the $u-u_N$ case we similarly have 
\begin{equation}\label{852}
\mathbb{E}_{\mu_m} \left( \|u-u_N\|^p_{H^{\alpha}(\mathbb{T})} \right) =  \mathbb{E}_{\mu_m} \left( \sum_{|n| >N} \frac{|g_n|^2}{\langle n \rangle ^{2- 2\alpha}} \right)^{p/2},
\end{equation}
which is $\le \frac{C(m,p)}{N^{(1-2\alpha)p/2}} = \frac{C(m,p)}{N}$.
\end{proof}

We conclude this subsection with our main result providing the link between the $\nu^k_{m}$ and $\mu_m$.

\begin{theorem} \label{823}
Fix a value of $m >0$. Suppose $F(u)$ is a non-negative $\mu$-measurable function. Then
\begin{equation*}
\int_{H^{1/2-}(\mathbb{T})} F(u) d\mu_m = \sum_{n=0}^{\infty} \int \frac{c_{n,m}}{\langle n \rangle^2} F(u) d\nu_m^n,
\end{equation*} and if for some $\alpha \in (0,1)$, $\frac{\|F(u)\|_{L^1(\nu^n_m)}}{\langle n \rangle^{\alpha}}$ is uniformly bounded in $k$ then $F(u) \in L^1(\mu_m)$ and 
\begin{equation} \label{824}
\int_{H^{1/2-}(\mathbb{T})} F(u) d\mu_m  \lesssim_{\alpha,m} \sup_{n} \frac{\|F(u)\|_{L^1(\nu^n_m)}}{\langle n \rangle^{\alpha}}.
\end{equation}

\end{theorem}
This will allow us to bound any function $F(u)$ on $H^{1/2-}(\mathbb{T})$, such as the energy term $e^{f(u) - \frac{1}{2}\int_{\mathbb{T}} |u|^6 dx}$, in the $L^1(\mu_m)$  norm by bounding it with respect to each $L^1(\nu_m^k)$ norm. 

\begin{proof} It is clear from the definition of $\mu_m$ that $\mu_m(\{ u | \sum_{n} \frac{|g_n|^2}{\langle n \rangle^2} =m \}) = 1$. By the monotone convergence theorem, 
\begin{equation} \label{825}
\int m \cdot F(u) d\mu_m = \int  \sum_{n} \frac{|g_n|^2}{\langle n \rangle^2} F(u) d\mu_m =  \lim_{N \rightarrow \infty} \int  \sum_{|n| \le N} \frac{|g_n|^2}{\langle n \rangle^2} F(u) d\mu_m.
\end{equation}
For each value of $N$ we can interchange integral and sum and apply Corollary \ref{815} to obtain 
\begin{equation} 
\begin{split}
 \lim_{N \rightarrow \infty} \int  \sum_{|n| \le N} \frac{|g_n|^2}{\langle n \rangle^2} F(u) d\mu_m &=  \lim_{N \rightarrow \infty}  \sum_{|n| \le N} \int  \frac{|g_n|^2}{\langle n \rangle^2} F(u) d\mu_m \\
&= \lim_{N \rightarrow \infty}  \sum_{n=0}^{N} \int \frac{c_{n,m}}{\langle n \rangle^2} F(u) d\nu_m^n \\
&= \sum_{n=0}^{\infty} \int \frac{c_{n,m}}{\langle n \rangle^2} F(u) d\nu_m^n \\
&\le   \sum_{n=0}^{\infty} \frac{\sup_n {c_{n,m}} }{\langle n \rangle^{2-\alpha}}\cdot \sup_{n} \frac{\|F(u)\|_{L^1(\nu^n_m)}}{\langle n \rangle^{\alpha}} \\
&\le C(\alpha) \sup_n{c_{n,m}} \cdot \sup_{n} \frac{\|F(u)\|_{L^1(\nu^n_m)}}{\langle n \rangle^{\alpha}}.\\
\end{split}
\end{equation}
By Lemma \ref{816}, we conclude that $\int m \cdot F(u) d\mu_m  \le C(\alpha,m) \frac{\|F(u)\|_{L^1(\nu^n_m)}}{\langle n \rangle^{\alpha}}$ and thus $F(u) \in L^1(\mu_m)$.

\end{proof}

\subsection{The weak convergence $\mu_m^{\epsilon} \rightarrow \mu_m$} \label{mathrefs}   
We conclude this section by proving the weak convergence of $\mu_m^{\epsilon} \rightarrow \mu_m$ and proving some elementary bounds on the integral of a function with respect to these measures. We start with the following lemma. 
\begin{lemma} \label{901}
For any fixed $m$,  $\delta >0$ and positive integer $N$ there exists a constant $C$ such that for sufficiently small $\epsilon >0$ we have the two bounds 
\begin{align}
\mu^{\epsilon}_{m}(\|u-u_N\|_{H^{\sigma}} > \delta) &\le  \frac{C}{p_0(m)\delta^2N^{1-2\sigma}},  \label{9010}\\
\mu_{m}(\|u-u_N\|_{H^{\sigma}} > \delta) &\le  \frac{C}{p_0(m)\delta^2N^{1-2\sigma}}. \label{9011} 
\end{align}
\end{lemma}
\begin{proof} 
By Markov's inequality and Lemma \ref{816}, we have 
\begin{equation} \label{902}
\begin{split}
\mu^{\epsilon}_{m}(\|u-u_N\|_{H^{\sigma}} > \delta)  &\le \frac{1}{\delta^2} \int \|u-u_N\|_{H^{\sigma}}^2 d\mu^{\epsilon}_m\\
&\le  \sum_{|n|>N} \int \frac{|g_n|^2}{\delta^2\langle n \rangle^{2-2\sigma}} d\mu^{\epsilon}_m \\
&\le  \sum_{|n|>N}\frac{C}{p_0(m)\delta^2\langle n \rangle^{2-2\sigma}} \\
&\le  \frac{C}{p_0(m)\delta^2N^{1-2\sigma}}. \\
\end{split}
\end{equation}
The same argument holds for $\mu_{m}(\|u-u_N\|_{H^{\sigma}} > \delta)$ when utilizing the second part of Lemma \ref{816}. \end{proof}
 
With this bound we prove the following theorem.
\begin{theorem} \label{903}
For each $m >0$, and fixed $\sigma < \frac{1}{2}$, the sequence of measures $\mu_m^{\epsilon}$ on $H^{\sigma}(\mathbb{T})$ converges weakly to $\mu_m$ in the $H^{\sigma}(\mathbb{T})$ norm. 
\end{theorem}
As alluded to in the introduction, it is easier to consider $H^{\sigma}(\mathbb{T})$ and not $H^{1/2-}(\mathbb{T})$, which has no norm. 
\begin{proof}
We must show that for any bounded uniformly continuous function $F$ on $H^{\sigma}(\mathbb{T})$, we have 
\begin{equation} \label{904}
\int_{H^{\sigma}(\mathbb{T})} F(u) d\mu_m = \lim_{\epsilon \rightarrow 0} \int_{H^{\sigma}(\mathbb{T})} F(u) d\mu_m^{\epsilon}.
\end{equation}

Consider a function $F$ that is bounded above and uniformly continuous in the $H^{\sigma}(\mathbb{T})$ norm. There exists a continuous,  increasing function $G :\mathbb{R}^+ \rightarrow \mathbb{R}^+$ such that for any $u \ne v$ we have $|F(u)-F(v)| \le G(\|u-v\|_{H^{\sigma}})$ and $\lim_{x \rightarrow 0+} G(x) = 0$.  

For any $N, \epsilon$, we have
\begin{equation*}
\begin{split}
|\int F(u) d\mu_m -  \int F(u) d\mu_m^{\epsilon}| &\le|\int F(u) d\mu_m -  \int F(u_N) d\mu_m| \\
&+  |\int F(u_N) d\mu_m -  \int F(u_N) d\mu_m^{\epsilon}|  + |\int F(u) d\mu^{\epsilon}_m -  \int F(u_N) d\mu_m^{\epsilon}|.\\
\end{split}
\end{equation*}

By the definition of $\mu_m$ the $|\int F(u_N) d\mu_m -  \int F(u_N) d\mu_m^{\epsilon}|$ term goes to $0$ as $\epsilon \rightarrow 0$. 

Applying Lemma \ref{901} and the properties of $G$, the $|\int F(u) - F(u_N) d\mu_m^{\epsilon}|$ term is bounded as follows for sufficiently small $\epsilon$: 
\begin{equation*} \label{905}
\begin{split}
 |\int F(u) - F(u_N) d\mu_m^{\epsilon}| &\le  \int_{|u-u_N|_{H^{\sigma}} > \delta} |F(u) - F(u_N)| d\mu^{\epsilon}_m + \int_{|u-u_N|_{H^{\sigma}} \le \delta} | F(u) - F(u_N)| d\mu^{\epsilon}_m\\
&\le \int_{|u-u_N|_{H^{\sigma}} > \delta}  2\|F\|_{L^{\infty}} d\mu^{\epsilon}_m  + G(\delta) \\ 
&\le 2\|F\|_{L^{\infty}}\mu^{\epsilon}_{m}(|u-u_N|_{H^{\sigma}} > \delta) + G(\delta)  \\
&\le \frac{2\|F\|_{L^{\infty}}}{\delta^2} \frac{C}{p_0(m)N^{1-2\sigma}}  +G(\delta).\\
\end{split}
\end{equation*}

The same argument applies to the $|\int F(u) d\mu_m -  \int F(u_N) d\mu_m|$ term if we make use of the second part of Lemma \ref{901}.
We conclude that for sufficiently small $\delta, \epsilon$ and large $N$ we have  

\begin{equation} \label{907}
|\int F(u) d\mu_m -  \int F(u) d\mu_m^{\epsilon}| \le \frac{C}{\delta^2N^{1-2\sigma}}+ 2G(\delta) +  |\int F(u_N) d\mu_m -  \int F(u_N) d\mu_m^{\epsilon}|.
\end{equation}

Now for any small $\alpha >0$ there exists a $\delta$ such that $2G(\delta) < \frac{\alpha}{2}$. For such $\delta$ select  $N$ large enough that $\frac{C}{\delta^2N^{1-2\sigma}} < \frac{\alpha}{2}$. Taking the $\limsup$ as $\epsilon \rightarrow 0$ of both sides of the previous equation, we have 
\begin{equation} \label{909}
\limsup_{\epsilon \rightarrow 0} |\int F(u) d\mu_m -  \int F(u) d\mu_m^{\epsilon}| \le \alpha.
\end{equation}

This holds for any $\alpha >0$, which implies that $\lim_{\epsilon \rightarrow 0} |\int F(u) d\mu_m -  \int F(u) d\mu_m^{\epsilon}| = 0$. \end{proof} 

Weak convergence will be useful for constructing solutions of mass $m$ as limits of solutions in $B_{m,\epsilon}$ and for proving the invariance of the measures at mass $m$. However, it cannot be used to show that the energy term $f(u)$ is in $L^1(\mu_m)$, because $f(u)$ has regularity at the $H^{1/2}(\mathbb{T})$ level and is not a continuous function in the $H^{\sigma}(\mathbb{T})$ norm for $\sigma < 1/2$. 

We now prove an intuitive relation between $\mu_m$ and $\mu_m^{\epsilon}$.
\begin{lemma} \label{910}
For any $m, \epsilon>0$ and measurable set $E$ we have 
\begin{equation*}\mu_{m}^{\epsilon} (E) = \frac{\int_{m-\epsilon}^{m+\epsilon} p_0(m') \mu_{m'}(E) dm'}{\mu(B_{m,\epsilon})} = \frac{\int_{m-\epsilon}^{m+\epsilon} p_0(m') \mu_{m'}(E) dm'}{\int_{m-\epsilon}^{m+\epsilon} p_0(m')dm'}.
\end{equation*}
\end{lemma}
\begin{proof} Consider a measurable set $E \subset H^{\sigma}(\mathbb{T})$ depending only on frequencies $\le N$, as in the proofs of Lemma \ref{806} and Lemma \ref{809}. Fix a value of $m >0$ and a value of $\epsilon$. By partitioning the interval $[m-\epsilon, m+\epsilon]$ into pieces, with some extra room at the endpoints, we have for each $\delta \in [-\frac{\epsilon}{n},\frac{\epsilon}{n}]$:

\begin{equation}
\begin{split}
 \mu(E \cap B_{m,\epsilon}) &\le \sum_{i=0}^n \mu(E \cap B_{m- \epsilon+\delta + \frac{2i\epsilon}{n},\frac{\epsilon}{n}}) \\  
\end{split}
\end{equation}
This inequality holds for each $\delta \in [-\frac{\epsilon}{n},\frac{\epsilon}{n}]$, therefore it still holds if we take the average over all $\delta$ in the interval:  
\begin{equation}
\begin{split}
\mu(E \cap B_{m,\epsilon}) &\le \frac{n}{2\epsilon} \int_{-\epsilon/n}^{\epsilon/n} \sum_{i=0}^{n} \mu(E \cap B_{m- \epsilon+\delta + \frac{2i\epsilon}{n},\frac{\epsilon}{n}}) d\delta \\
&\le \frac{n}{2\epsilon} \int_{m' = m - \epsilon- \epsilon/n}^{m+\epsilon+ \epsilon/n} \mu(E \cap B_{m',\frac{\epsilon}{n}})  dm' \\
&\le \int_{m - \epsilon- \epsilon'}^{m+\epsilon + \epsilon'} \frac{\mu(E \cap B_{m',\epsilon'})}{2\epsilon'}  dm'  =  \int_{m - \epsilon - \epsilon'}^{m+\epsilon+\epsilon'} \frac{\mu(B_{m',\epsilon'})}{2\epsilon'} \mu_{m'}^{\epsilon'}(E)  dm',  \\
\end{split}
\end{equation}
after making the substitution $\epsilon' =\epsilon/n$. Taking the limit as $n \rightarrow \infty$ which is equivalent to $\epsilon' \rightarrow 0$ the above integral converges pointwise to $\int_{m-\epsilon}^{m+\epsilon} p_0(m') \mu_{m'}(E) dm'$. In addition, the integrand is bounded above by $\int_{m-2\epsilon}^{m+2\epsilon} \|p_0(m')\|_{L^{\infty}} dm'$. Therefore the dominated convergence theorem implies that  
$$ \mu(E \cap B_{m,\epsilon}) \le \int_{m-\epsilon}^{m+\epsilon} p_0(m') \mu_{m'}(E) dm'.$$

Now we provide a lower bound: 
\begin{equation}
\begin{split}
\mu(E \cap B_{m,\epsilon}) &\ge \sum_{i=1}^{n-1} \mu(E \cap B_{m- \epsilon+\delta + \frac{2i\epsilon}{n},\frac{\epsilon}{n}}) \\  
&\ge \frac{n}{2\epsilon} \int_{-\epsilon/n}^{\epsilon/n}  \sum_{i=1}^{n-1} \mu(E \cap B_{m- \epsilon+\delta + \frac{2i\epsilon}{n},\frac{\epsilon}{n}}) d\delta\\  
&\ge \frac{n}{2\epsilon} \int_{m' = m - \epsilon+3\epsilon/n}^{m+\epsilon-3\epsilon/n} \mu(E \cap B_{m',\frac{\epsilon}{n}})  dm' \\
&\ge \int_{m - \epsilon+ 3\epsilon'}^{m+\epsilon -3\epsilon'} \frac{\mu(E \cap B_{m',\epsilon'})}{2\epsilon'}  dm'  =  \int_{m - \epsilon +3\epsilon'}^{m+\epsilon-3\epsilon'} \frac{\mu(B_{m',\epsilon'})}{2\epsilon'} \mu_{m'}^{\epsilon'}(E)  dm'  , \\
\end{split}
\end{equation}
after making the substitution $\epsilon' =\epsilon/n$. Taking the limit as $\epsilon' \rightarrow 0$ the above integral converges pointwise to $\int_{m-\epsilon}^{m+\epsilon} p_0(m') \mu_{m'}(E) dm'$. It is bounded above just as the previous integral. Therefore we have  
\begin{equation*}
\begin{split}
\int_{m-\epsilon}^{m+\epsilon} p_0(m')\mu_{m'}(E) dm' &\le \mu(E \cap B_{m,\epsilon}) \le \int_{m-\epsilon}^{m+\epsilon} p_0(m') \mu_{m'}(E) dm' \\
\mu(E \cap B_{m,\epsilon}) &= \int_{m-\epsilon}^{m+\epsilon} p_0(m') \mu_{m'}(E) dm' .\\
\end{split}
\end{equation*} 

We conclude that for any measurable $E$ depending only on frequencies $\le N$ we have 
\begin{equation} \label{915}
\mu_{m}^{\epsilon} (E) = \frac{\mu(E \cap B_{m,\epsilon})}{\mu(B_{m,\epsilon})} = \frac{\int_{m-\epsilon}^{m+\epsilon} p_0(m') \mu_{m'}(E) dm'}{\int_{m-\epsilon}^{m+\epsilon} p_0(m')dm'}.
\end{equation} Since measurable sets of the form $E$ are dense in the sigma algebra of $H^{\sigma}(\mathbb{T})$, this proves the result for all $E$. \end{proof}

The following immediate corollary allows us to integrate functions with respect to $\mu_{m}^{\epsilon}$.

\begin{corollary} \label{916}
For any function $F \in L^1(\mu_m^{\epsilon})$ we have $$\int F(u) d\mu_m^{\epsilon} = \frac{\int_{m-\epsilon}^{m+\epsilon} p_0(m')[\int F(u) d\mu_{m'}]dm'}{\int_{m-\epsilon}^{m+\epsilon} p_0(m')dm'}.$$
\end{corollary}

\section{Construction of the measure $\rho_m$} \label{mathrefs}
Now that we have the base measures $\mu_m$ and $\nu_m^k$ defined, we can construct the invariant measure $\rho_m$ by bounding the $e^{f_N(u)}$ term with respect to $\nu_m^k$ and $\mu_m$.

\subsection{Computing $\int F(u) d\nu_m^k$.} \label{mathrefs}
In this subsection we will derive a formula for $\int_{H^{1/2-}(\mathbb{T})} F(u) d\nu_m^k$ for each $m,k$ and any integrable function $F$, and eventually apply this bound to $F(u) = e^{pf_N(u)}$. We will rely on Theorem \ref{401} and Proposition \ref{24}, as well as the argument outlined in Section 2 that is motivated by the divergence theorem.
 
Recall that $$f_{N}(u) = \frac{3}{4} \sum_{n_1+n_2 = n_3+n_4, |n_i| \le N}
(n_1+n_2)\frac{g_{n_1}(\omega)g_{n_2}(\omega)\overline{g_{n_3}(\omega)}\overline{g_{n_4}(\omega)}}{\langle n_1 \rangle \langle n_2 \rangle \langle n_3 \rangle \langle n_4 \rangle}.$$ Our sets $\Gamma_{m,s}^k$ are defined based on multiplying the Fourier coefficients at frequencies $k$ and $-k$ by $s$. This corresponds to multiplying each term by $s$ raised to the number of entries of $\overline{n} = (n_1,n_2,n_3,n_4)$ equal to $\pm k$. This leads us to the following definition:   

\begin{definition}
For any $k \ge 0$, and $s \in [\frac{1}{2},2]$ we define $c_k(\overline{n},s) =  s^{\kappa(\overline{n})}$ where $\kappa(\overline{n})$ is the number of coordinates of $\overline{n}$ equal to $k$ or ${-k}$.  
\end{definition}
When we take $f(T^k_{s}(u))$ we get an extra factor of $s$ in each $g_k$ and $g_{-k}$ term. Therefore 
\begin{equation} \label{929}
f_N(T^k_s(u)) = \frac{3}{4} \sum_{n_1+n_2 = n_3+n_4, |n_i| \le N} c_k(\overline{n},s)\frac{(n_1+n_2)g_{n_1}g_{n_2}\overline{g_{n_3}}\overline{g_{n_4}}}{\langle n_1 \rangle \langle n_2 \rangle\langle n_3 \rangle\langle n_4 \rangle }.  \\
\end{equation}
Note that the definition of $c_k(\overline{n},s)$ depends on the number of entries of $\overline{n}$ equal to $k$ or $-k$. The function $c_k$ is even, unaffected by permutations of the entries and bounded above by $16$, which means it satisfies the conditions of Theorem \ref{401} and Proposition \ref{24}. We have the following corollary.
\begin{corollary} \label{928}
For every $s \in [\frac{1}{2},2]$, $k \ge 0$ and $p \ge 1$ there exists a constant $m_{p} >0$ such that 
$$\int_{H^{1/2-}(\mathbb{T})} 1_{\|u \|_{L^2}^2 < m_{p}} e^{pf_N(T_s^k(u))}d\mu \le C(p)$$
for some constant depending only on $p$.
In addition, for $N \ge M$  we have 
$$\int_{H^{1/2-}(\mathbb{T})} |f_{N}(T_s^k(u)) - f_{M}(T_s^k(u))|^2 d\mu \lesssim \frac{1}{N}.$$
\end{corollary}

We seek to bound $\int_{H^{1/2-}(\mathbb{T})} 1_{\|u \|_{L^2}^2 < m_{p}} e^{pf_N(u)}d\mu $ and start with the following key lemma. 

\begin{lemma} \label{930}
Suppose $F$ is an integrable, measurable function on $(H^{1/2-}(\mathbb{T}),d\mu)$ and $2\ge  s > 1$. We have
$$\int_{\Gamma_{m,s}} F(u) d\mu = \int_{A_m} s^2e^{(1-s^2)|g_0|^2} F(T_s(u)) - \frac{1}{s^2}e^{(1-1/s^2)|g_0|^2}F(T_{1/s}(u))  d\mu. $$  
\end{lemma}
\begin{proof}
Recall that 
$\Gamma^0_{m,s}= \Gamma_{m,s}  = A_{m,1/s} \backslash A_{m,s}.$
Therefore \begin{equation} \label{931}
\int_{\Gamma_{m,s}} F(u) d\mu  = \int_{A_{m,1/s}} F(u) d\mu  - \int_{A_{m,s}} F(u) d\mu. \end{equation}  

Applying the change of variables $g_0 = sg_0'$, we have 

\begin{equation} \label{932}
\begin{split}
\int_{A_{m,1/s}} F(u) d\mu &= \int_{|g_0|^2 < s^2m} \int_{\sum_{|n| \ge 1} \frac{|g_n|^2}{\langle n \rangle ^2} \le m - |g_0|^2/s^2}F(u) d\mu(g_i, |i| \ge 1) e^{-|g_0|^2}/\pi dRe(g_0) dIm(g_0)\\  
&= \int_{|g_0'|^2 < m } \int_{\sum_{|n| \ge 1} \frac{|g_n|^2}{\langle n \rangle ^2} \le m - |g_0'|^2} F(T_s(u)) d\mu(g_i, |i| \ge 1) s^2e^{-s^2|g_0'|^2}/\pi dRe(g_0') dIm(g_0')\\  
&= \int_{A_m} s^2e^{(1-s^2)|g_0|^2} F(T_s(u))  d\mu. \\
\end{split}
\end{equation}

We apply the exact same argument to the integral over $A_{m,s}$ to obtain 
\begin{equation} \label{949}
\begin{split}
\int_{A_{m,s}} F(u) d\mu &= \int_{|g_0|^2 < m/s^2} \int_{\sum_{|n| \ge 1} \frac{|g_n|^2}{\langle n \rangle ^2} \le m - s^2|g_0|^2}F(u) d\mu(g_i, |i| \ge 1) e^{-|g_0|^2}/\pi dRe(g_0) dIm(g_0)\\  
&= \int_{|g_0'|^2 < m } \int_{\sum_{|n| \ge 1} \frac{|g_n|^2}{\langle n \rangle ^2} \le m - |g_0'|^2} F(T_{1/s}(u)) d\mu(g_i, |i| \ge 1) \frac{1}{s^2}e^{-|g_0'|^2/s^2}/\pi dRe(g_0') dIm(g_0')\\  
&= \int_{A_m} \frac{1}{s^2}e^{(1-1/s^2)|g_0|^2} F(T_{1/s}(u))  d\mu. \\
\end{split}
\end{equation}

Combining equations \eqref{932} and \eqref{949} yields
\begin{equation} \label{950}
\int_{\Gamma^k_{m,s}} F(u) d\mu = \int_{A_m} s^2e^{(1-s^2)|g_0|^2} F(T_s(u))  -  \frac{1}{s^2}e^{(1-1/s^2)|g_0|^2}F(T_{1/s}(u))  d\mu.
\end{equation}

\end{proof}
We now state the same result for $k \ge 1$.
\begin{lemma} \label{933}
Suppose $F$ is an integrable, measurable function on $H^{1/2-}(\mathbb{T}),d\mu$ and $2 \ge s>1$. Then for any $k \ge 1$.
$$\int_{\Gamma^k_{m,s}} F(u) d\mu = \int_{A_m} s^4e^{(1-s^2)(|g_k|^2 + |g_{-k}|^2)} F(T^k_s(u))  - \frac{1}{s^4}e^{(1-1/s^2)(|g_k|^2+|g_{-k}|^2)}F(T^k_{1/s}(u))  d\mu.$$  
\end{lemma}
\begin{proof}
The proof is identical to that of Lemma \ref{930} except that now there are two terms, so we get a factor of $s^4$ from the change of variables. 
\end{proof}

These lemmas give us a way to evaluate $\int F(u_N) d\nu_m^k$ by taking the limit of the above quantity divided by the measure of $\Gamma^k_{m,s}$.

\begin{theorem} \label{960}
Consider a function $F(u_N)$ that depends on finitely many frequencies of $u$. Suppose there exists a constant $C$ such that $|F(T^k_s(u_N))|, |\frac{\partial}{\partial s} F(T^k_s(u_N))| \le C$ for all $s \in [\frac{1}{2},2]$, $k \ge 0$ and $u \in A_{m+\epsilon}$ for any $\epsilon >0$. Then we have
\begin{equation*}
\begin{split}
\int F(u_N) d\nu_m^0 &= \frac{\int_{A_m}   (1-|g_0|^2)F(u_N) + \frac{1}{2}\frac{\partial}{ \partial s}|_{s=1} F(T_s(u_N))  d\mu }{\int_{A_m} 1-|g_0|^2  d\mu } \\
\int F(u_N) d\nu_m^k &= \frac{\int_{A_m}   (2-|g_k|^2-|g_{-k}|^2)F(u_N) + \frac{1}{2}\frac{\partial}{ \partial s}|_{s=1} F(T_s^k(u_N))  d\mu }{\int_{A_m} 2-|g_k|^2-|g_{-k}|^2  d\mu }. \\
\end{split}
\end{equation*}
In addition, $\int_{A_m} 1-|g_0|^2  d\mu = \lim_{s \rightarrow 1+}\frac{\mu(\Gamma_{m,s})}{s^2-1/s^2}$ and $\int_{A_m} 2-|g_k|^2-|g_{-k}|^2  d\mu =  \lim_{s \rightarrow 1+}\frac{\mu(\Gamma^k_{m,s})}{s^2-1/s^2}$  both of which are positive and can be bounded using Lemma \ref{705}, Lemma \ref{709} and Lemma \ref{609}.

\end{theorem}
Note that the condition $|F(T^k_s(u_N))|, |\frac{\partial}{\partial s} F(T^k_s(u_N))| \le C$ on $A_{m+\epsilon}$  is satisfied by almost any reasonable function $F(u)$ that is  bounded by some Sobolev norm of $u$, since $$\|u_N\|_{H^s(\mathbb{T})} \lesssim N^s\|u_N\|_{L^2(\mathbb{T})} \le N^s(m+\epsilon)^{1/2}.$$ 
\begin{proof}
We do the computation for the $k \ge 1$ case, and note that it is identical in the $k=0$ case. Since $F$ depends only on finitely many frequencies of $u$, we can apply Definition \ref{812} directly. After applying Lemma \ref{933}, cross multiplying and applying the limit definition of the derivative we have
\begin{equation}
\begin{split}
&\int F(u_N) d\nu_m^k\\
 &=  \lim_{s \rightarrow 1+} \frac{\int_{\Gamma_{m,s}^k} F(u_N) d\mu}{\int_{\Gamma_{m,s}^k} 1 d\mu}\\
&= \lim_{s \rightarrow 1+} \frac{\int_{A_m} s^4e^{(1-s^2)(|g_k|^2 + |g_{-k}|^2)} F(T^k_s(u_N))  - \frac{1}{s^4}e^{(1-1/s^2)(|g_k|^2+|g_{-k}|^2)}F(T^k_{1/s}(u_N))  d\mu}{\int_{A_m} s^4e^{(1-s^2)(|g_k|^2 + |g_{-k}|^2)}  - \frac{1}{s^4}e^{(1-1/s^2)(|g_k|^2+|g_{-k}|^2)}  d\mu}\\
&=  \lim_{s \rightarrow 1+} \frac{\int_{A_m} s^4e^{(1-s^2)(|g_k|^2 + |g_{-k}|^2)} F(T^k_s(u_N))  - \frac{1}{s^4}e^{(1-1/s^2)(|g_k|^2+|g_{-k}|^2)}F(T^k_{1/s}(u_N))  d\mu /(s^2 - 1/s^2) }{\int_{A_m} s^4e^{(1-s^2)(|g_k|^2 + |g_{-k}|^2)}  - \frac{1}{s^4}e^{(1-1/s^2)(|g_k|^2+|g_{-k}|^2)}  d\mu/(s^2 - 1/s^2)}\\
&= \frac{\int_{A_m} \frac{\partial}{\partial r}|_{r=1} r^2e^{(1-r)(|g_k|^2 + |g_{-k}|^2)} F(T^k_{\sqrt{r}}(u_N)) d\mu}{\int_{A_m} \frac{\partial}{\partial r}|_{r=1} r^2e^{(1-r)(|g_k|^2 + |g_{-k}|^2)}    d\mu}.\\
\end{split}
\end{equation}
We can place the limit inside the integral by applying the dominated convergence theorem and the boundedness condition on $F$. Now we apply the product rule and obtain
\begin{equation}
\int F(u_N) d\nu_m^k =  \frac{\int_{A_m}   (2-|g_k|^2-|g_{-k}|^2)F(u_N) + \frac{1}{2}\frac{\partial}{ \partial s}|_{s=1} F(T_s^k(u_N))  d\mu }{\int_{A_m} 2-|g_k|^2-|g_{-k}|^2  d\mu }.
\end{equation}

\end{proof}

Observe that the denominator $\int_{A_m} 2-|g_k|^2-|g_{-k}|^2  d\mu$ is just the limit from Lemma \ref{709}. Applying Lemma \ref{705}, Lemma \ref{709} and Lemma \ref{609} we arrive at the following key corollary.

\begin{corollary}
For each $m>0$ there exists a constant $C(m)$ such that 
$$\int_{A_m} 2-|g_k|^2-|g_{-k}|^2  d\mu \ge \frac{C(m)}{\langle k \rangle^2}.$$
\end{corollary}  

This demonstrates the power of our decomposition $\frac{1}{m}\mu_m = \underset{k \ge 0}{\sum} \frac{c_{k,m}}{\langle k \rangle^2}\nu_m^k$.   Theorem \ref{960} provides a direct formula for the integral of a function with respect to $\nu_m^k$, something that would be unachievable for $\mu_m$.  Note the similarity to the idea of the divergence and Stokes' theorem, we are relating the integral of the function $F$ on the boundary set $\Gamma_m$ to its derivative on the interior, $A_m$. The downside to this formula is that in order to bound $\int F(u_N)d\mu_m$ we have to bound the integral  
\begin{equation*}
\int_{A_m}   (2-|g_k|^2-|g_{-k}|^2)F(u_N) + \frac{1}{2}\frac{\partial}{ \partial s}|_{s=1} F(T_s^k(u_N))  d\mu
\end{equation*} 
by $\frac{C}{\langle k \rangle^{1+}}$ for a constant $C$ that does not depend on $k$. The $\frac{\partial}{ \partial s}|_{s=1} F(T_s^k(u_N))$ term should be manageable, but bounding $$\int_{A_m}   (2-|g_k|^2-|g_{-k}|^2)F(u_N)  d\mu$$ will not be an easy task. This integral has positive and negative components that do not neatly cancel. We suspect that the random variables $|g_k|^2 + |g_{-k}|^2$ and $e^{f_N(u)}$ have positive covariance. Noting that the denominator is positive, positive covariance of $|g_k|^2 + |g_{-k}|^2$ and $F(u_N)$ would imply that 
\begin{equation}
\frac{\int_{A_m}   (2-|g_k|^2-|g_{-k}|^2)F(u_N)d\mu }{\int_{A_m} 2-|g_k|^2-|g_{-k}|^2  d\mu } \le  \frac{1}{\mu(A_m)}\int_{A_m} F(u_N)d\mu .
\end{equation}

The function $e^{f_N(u)}$ is too complicated to directly compute its covariance with  $|g_k|^2 + |g_{-k}|^2$, so we will split it into several components.

\subsection{Bounding $\int e^{f_N(u)}d\nu_m^k$} \label{mathrefs}

We want to divide $f_N(u)$ into components, each of which is proportional to a power of  $|g_k|^2 + |g_{-k}|^2$, making it possible to compute the covariance of each term with $|g_k|^2 + |g_{-k}|^2$.

\begin{definition}  \label{961}
 Fix a value of $k \ge 0$. Recall that 
\begin{equation}
 f_N(u) = \frac{3}{4}\sum_{n}  \sum_{n_1+n_2=n, n_3+n_4 =n, |n_i| \le N} \frac{ng_{n_1}g_{n_2}\overline{g_{n_3}}\overline{g_{n_4}}}{\langle n_1 \rangle \langle n_2 \rangle\langle n_3 \rangle\langle n_4 \rangle}.
\end{equation}
For each $0 \le j \le 4$ we let $G_k^j(u_N)$ be the sum over all indices $(n_1,n_2,n_3,n_4)$ such that exactly $j$ of these indices are equal to $k$ or $-k$.  For example $G_k^4(u_N) = \frac{3k(|g_k|^4 - |g_{-k}|^4)}{2\langle k \rangle^4}$ for $k \le N$.
\end{definition} 

Observe that by the $g_n \rightarrow g_{-n}$ symmetry, each $G_k^j(u_N)$ has a symmetric distribution. In addition, if we scale $g_k$ and $g_{-k}$ by a factor of $s$, $G_k^j(u_N)$ is scaled by a factor of $s^j$. We have an immediate result bounding each $G_k^j(u_N)$.
\begin{lemma}\label{959}
For each $N \ge 0$, $k \ge 0$, $0 \le j \le 4$ and $m < m_{6p}$ we have
\begin{equation}
\int_{A_m} e^{pG_k^j(u_N)} d\mu \le C(p),
\end{equation}
for a constant depending only on $p$.
\end{lemma} 
\begin{proof}
We seek to apply Proposition \ref{552}. Let $Q = \{k,-k\}$ and let $Q' = \mathbb{Z} \backslash Q$. We can write $G_k^j(u_N)$ as a sum of $\binom{4}{j}$ terms of the form $\partial_x (P_{Q_1}u_N) \cdot  P_{Q_2}u_N \cdot P_{Q_3}\overline{u_N}\cdot P_{Q_4}\overline{u_N}$, where each $Q_i$ is either $Q$ or $Q'$. Applying H\"older's inequality and Proposition \ref{552} gives us the desired bound. 
\end{proof}

We now define a new pair of sets that will be important in our proof of boundedness of the exponential.
\begin{definition}
Recall that $A_m$ is the set of $u$ with mass $\le m$. Let $\mathcal{A}^k_m \subset L^2(\mathbb{T})$ be the slightly larger set of functions $u$ that satisfy 
$$\sum_{|n| \ne k} |\widehat{u(n)}|^2  = \sum_{|n| \ne k} \frac{|g_n|^2}{\langle n \rangle^2} \le m.$$ 
Also define $\tilde{\mathcal{A}^k_m}$ to be the intersection
$$\tilde{\mathcal{A}^k_m} = \mathcal{A}^k_m \cap \{|g_k|^2 + |g_{-k}|^2 < \langle k \rangle^2m\}.$$
\end{definition}

Note that on these sets any two functions $F_1(g_k,g_{-k})$ and $F_2(\{g_n, |n| \ne k\})$ are independent. Also, the set  $\tilde{\mathcal{A}_m^k}$ contains the set $A_m$. We now prove that the difference between the size of these sets is small. 
\begin{lemma} \label{981}
For each $k \ge 0$ and $m>0$ we have 
$$\mu(\mathcal{A}_m^k \backslash A_m) \le \frac{C}{\langle k \rangle^2},$$
implying that $\mu(\tilde{\mathcal{A}}_m^k \backslash A_m) \le \frac{C}{\langle k \rangle^2}$.
\end{lemma}
\begin{proof}
We proceed as in the proof of Lemma \ref{709}. Recall that $P_k(y)$ is the distribution function of $y = \underset{|n| \ne k}{\sum} \frac{|g_n|^2}{\langle n \rangle^2}$ with respect to the measure $\mu$. Letting $g_k = a +ib$, $g_{-k} = a'+ib'$ with $a,a',b,b'$ all $\mathcal{N}(0,\frac{1}{2})$ distributed, we have 

\begin{equation}
\begin{split}
\mu(\mathcal{A}^k_m  \backslash A_m) &=\int_{y=0}^{m} P_k(y)\mu\left(\langle k \rangle^2(m-y) < |g_k|^2+ |g_{-k}|^2\right) dy \\
&= \int_{y=0}^{m} P_k(y)\int_{a^2+a'^2+b^2+b'^2= \langle k \rangle^2(m-y)}^{\infty}  \frac{e^{-(a^2+a'^2+b^2+b'^2)}}{\pi^2} da db da'db'  dy\\
&=  \int_{y=0}^{m} P_k(y)\int_{r_1^2+r_2^2= \langle k \rangle^2(m-y)}^{\infty} 4r_1r_2e^{-(r_1^2+r_2^2)}dr_1 dr_2  dy\\
&=  \int_{y=0}^{m} P_k(y) \int_{\theta = 0}^{\pi/2} \int_{r = \langle k \rangle \sqrt{m-y}}^{\infty} 4r^2\cos(\theta)\sin(\theta) e^{-r^2} rdrd\theta dy\\ 
&=  \int_{y=0}^{m} P_k(y) \int_{r = \langle k \rangle \sqrt{m-y}}^{\infty} 2r^3 e^{-r^2} dr dy\\ 
&=  \int_{y=0}^{m} P_k(y) |_{r =  \langle k \rangle \sqrt{m-y}}^{\infty}  -(r^2+1)e^{-r^2} dr dy\\
&= \int_{y=0}^{m} P_k(y)(1+ \langle k \rangle^2(m-y)e^{-\langle k \rangle^2(m-y)}dy.\\
\end{split}
\end{equation}
Applying a change of variables and Lemma \ref{605} yields
\begin{equation}
\begin{split}
\mu(\mathcal{A}^k_m  \backslash A_m)  &= \int_{y=0}^{m} P_k(y)(1+ \langle k \rangle^2(m-y)e^{-\langle k \rangle^2(m-y)}dy\\
&= \int_{0}^m P_k(m-y)(1+ y\langle k \rangle^2e^{-y\langle k \rangle^2}dy\\
&= \frac{1}{\langle k \rangle^2} \int_{0}^{m\langle k \rangle^2} P_k(m-\frac{y}{\langle k \rangle^2})(1 + y)e^{-y}dy\\
&\le \frac{\|P_k\|_{L^{\infty}(\mathbb{R})}}{\langle k \rangle^2} \int_0^{\infty} (1+y)e^{-y}dy \\ 
&\le \frac{C}{\langle k \rangle^2}.\\
\end{split}
\end{equation}

\end{proof}

We now turn to proving that $\mathrm{Cov}[G_k^j(u_ N),|g_k|^2 + |g_{-k}|^2] \ge 0$. We start with two standard probability lemmas.
\begin{lemma} \label{962}
If $X$ is a random variable with a symmetric distribution then $\mathbb{E}[Xe^X] \ge 0$. 
\end{lemma}
\begin{proof}
Since $X$ is symmetrically distributed, it has the same distribution as $-X$:
\begin{equation}
\begin{split}
\mathbb{E}[Xe^X] &= \mathbb{E}[-Xe^{-X}]\\
&= \frac{1}{2}\mathbb{E}[Xe^X+(-X)e^{-X}]\\
&= \frac{1}{2}\mathbb{E}[X(e^X-e^{-X})]\\
\end{split}
\end{equation}
This is $\ge 0$, since $X(e^X-e^{-X}) \ge 0$ for all values of $X$.
\end{proof}

\begin{lemma} \label{9620}
Suppose $X$ and $Y$ are random variables with bounded first and second moments, such that $Y$ is continuous with distribution function $p_Y(y)$. If there exists a non-decreasing function $E(y)$ such that for all continuous $f$ we have, 
$$\mathbb{E}(Xf(Y)) = \int_{-\infty}^{\infty} p_{Y}(y)f(y)E(y) dy$$
 then $\mathrm{Cov}(X,Y) \ge 0$.   
\end{lemma} Observe that $E(y)$ plays the role that $\mathbb{E}(X|Y=y)$ would play were $Y$ a discrete random variable, however, since $Y$ is continuous we cannot take  $\mathbb{E}(X|Y=y)$, since $\{Y=y\}$ is a set of measure $0$.
\begin{proof}
By the monotone property of $E(y)$, we know that 
$$(E(y_1) - E(y_2))(y_1-y_2) \ge 0$$
for all $y_1,y_2 \in \mathbb{R}$.

We have 
\begin{equation*}
\begin{split}
0 & \le \int_{y_1}\int_{y_2} p_Y(y_1)p_Y(y_2)(E(y_1) - E(y_2))(y_1-y_2)dy_2dy_1\\
&= \int_{y_1}\int_{y_2} p_Y(y_1)p_Y(y_2)\left( y_1E(y_1)  + y_2 E(y_2)\right) dy_2dy_1\\
&- \int_{y_1}\int_{y_2} p_Y(y_1)p_Y(y_2)\left( y_2E(y_1)  + y_1 E(y_2)\right) dy_2dy_1\\
&= \int_{y_1}y_1p_Y(y_1)E(y_1) dy_1 \int_{y_2}p_Y(y_2)dy_2 + \int_{y_1}p_Y(y_1)dy_1\int_{y_2}y_2p_Y(y_2) E(y_2)dy_2\\
 &- \int_{y_1}p_Y(y_1)E(y_1) dy_1 \int_{y_2}y_2p_Y(y_2)dy_2 - 
\int_{y_1}y_1p_Y(y_1)dy_1 \int_{y_2}p_Y(y_2) E(y_2)dy_2 \\
&= \mathbb{E}(XY) \cdot 1 + 1 \cdot \mathbb{E}(XY) - \mathbb{E}(X)\mathbb{E}(Y) - \mathbb{E}(Y)\mathbb{E}(X)\\
&= 2\mathrm{Cov}(X,Y). \\
\end{split}
\end{equation*}

\end{proof}

We will use this lemma when we prove our covariance result.

\begin{proposition} \label{963}
For any $N \ge 0$, $k\ge 0$, $p \ge 1$, $0 \le j \le 4$, $\mathrm{Cov}[e^{pG_k^j(u_N)}, |g_k|^2 + |g_{-k}|^2] \ge 0$ on the set $\mathcal{A}_m^k$ and over any subset consisting of a fixed set of values of $|g_k|^2+|g_{-k}|^2$.
\end{proposition}

\begin{proof}

We assume $k \ge 1$, since the $k=0$ case is simpler. Let $r^2 = |g_k|^2 + |g_{-k}|^2$ and let $r(\tilde{a}_k + i\tilde{b}_k) = g_k$, $r(\tilde{a}_{-k} + i\tilde{b}_{-k}) = g_{-k}$. This decomposes $g_k$ and $g_{-k}$ into a radius $r$ and an element of $\mathbb{S}^3$. Thus the set where $r^2 = |g_k|^2 + |g_{-k}|^2$ is equivalent to the sphere $\mathbb{S}^3$ as the solution set to the equation $$|\tilde{a}_k|^2 + |\tilde{b}_k|^2 + |\tilde{a}_{-k}|^2 + |\tilde{b}_{-k}|^2 = 1.$$ By change of variables, the component of the measure $\mu$ that corresponds to $g_k$ and $g_{-k}$ is equal to $2r^3e^{-r^2}drd\sigma^3$, where  $\sigma^3$ denotes the normalized surface measure on $\mathbb{S}^3$. Therefore we can write the measure $d\mu$ as the product $2r^3e^{-r^2}dr \times d\sigma^3 \times  d\mu'$ for $\mu'$ that corresponds to the distribution of $P_{\mathbb{Z}/\{k,-k\}} u = \sum_{|n| \ne k} \frac{g_n}{\langle n \rangle}e^{inx}$. 

Therefore for each value of $r>0$, we have 
\begin{equation*}
\begin{split}
\mathbb{E}_{\mu}[e^{pG_k^j(u_N)} &\cdot f(r)| u \in \mathcal{A}^k_m] \\
&= \int_{\mathcal{A}^k_m} e^{pG_k^j(u_N)} \cdot f(r) d\mu\\
&= \int_{0}^{\infty} \int_{\mathbb{S}^3} \int_{\mathcal{A}_m^k}f(r) e^{pG_k^j(u_N(P_{\mathbb{Z}/\{k,-k\}} u,r,\theta))} d\mu'(P_{\mathbb{Z}/\{k,-k\}} u) d\sigma^3(\theta) 2r^3e^{-r^2}dr \\
&= \int_{0}^{\infty}f(r) 2r^3e^{-r^2}\left( \int_{\mathcal{A}_m^k \times \mathbb{S}^3}e^{pG_k^j(u_N(P_{\mathbb{Z}/\{k,-k\}} u,r,\theta))}  d(\mu' \times \sigma^3)(P_{\mathbb{Z}/\{k,-k\}} u, \theta) \right) dr \\
\end{split}
\end{equation*}

We will show the function $$\int_{\mathcal{A}_m^k \times \mathbb{S}^3}e^{pG_k^j(u_N(P_{\mathbb{Z}/\{k,-k\}} u,r,\theta))}  d(\mu' \times \sigma^3)(P_{\mathbb{Z}/\{k,-k\}} u, \theta)$$ is a non-decreasing function of $r$ for $r > 0$, which will allow us to apply Lemma \ref{9620}. 

For each value of $r>0$, we have 
\begin{equation}
\begin{split}
\frac{\partial}{\partial r} \int_{\mathcal{A}_m^k \times \mathbb{S}^3}e^{pG_k^j(u_N)}  d(\mu' \times \sigma^3) 
&= \int_{\mathcal{A}_m^k \times \mathbb{S}^3} \frac{\partial}{\partial r}[pG_k^j(u_N)] e^{pG_k^j(u_N)}  d(\mu' \times \sigma^3) \\
&= \int_{\mathcal{A}_m^k \times \mathbb{S}^3} jr^{j-1}[pG_k^j(u_N)] e^{pG_k^j(u_N)}  d(\mu' \times \sigma^3), \\
\end{split}
\end{equation}
since each term of $G_k^j(u_N)$ is proportional to a $j$th power of $g_k,g_{-k}$ and thus proportional to $r^j$. For $j=0$ this expectation is clearly $0$, and by Lemma \ref{962} this expectation is non-negative for $j \ge 1$. Therefore the $r$ derivative of $$ \int_{\mathcal{A}_m^k \times \mathbb{S}^3}e^{pG_k^j(u_N)}  d(\mu' \times \sigma^3)$$ is non-negative, implying a non-decreasing function. 

In the $k=0$ case we can repeat the same argument, taking the integral  over $\mathbb{S}^1$. Combined with Lemma \ref{9620}, this implies that $e^{pG_k^j(u_N)}$ and $|g_k|^2 + |g_{-k}|^2$ have non-negative covariance over any set of values of $r^2 =  |g_k|^2+|g_{-k}|^2$ within the set $\mathcal{A}_m^k$, such as $\mathcal{A}_m^k$ itself or $\tilde{\mathcal{A}_m^k}$.
\end{proof}

We have the following result:
\begin{proposition} \label{964}
For each $m < \frac{1}{2}m_{96p}$ and $k \ge 0$, we have 
\begin{equation}
\frac{\int_{A_m}   (2-|g_k|^2-|g_{-k}|^2)e^{pG_k^j(u_N)} d\mu }{\int_{A_m} 2-|g_k|^2-|g_{-k}|^2  d\mu} \le \frac{C(p)e^m\langle k \rangle^{1/4}}{\mu(A_m)m^2\int_{m/2}^{m} P_k\left( m-\frac{y}{\langle k \rangle^2}\right)  dy}.
\end{equation}
\end{proposition}
If we wanted we could bound the integral by $\langle k\rangle^{\alpha}$, for any small $\alpha>0$. The power $1/4$ will be sufficient for the later argument.  
\begin{proof}
Assume $k \ge 1$. We start by applying the triangle inequality to bound the numerator  by
\begin{equation} \label{965}
\begin{split}
\int_{A_m}   (2-|g_k|^2-|g_{-k}|^2)e^{pG_k^j(u_N)} d\mu &\le \int_{\tilde{\mathcal{A}_m^k}}   (2-|g_k|^2-|g_{-k}|^2)e^{pG_k^j(u_N)} d\mu \\
&\qquad \qquad + \left| \int_{\tilde{\mathcal{A}_m^k} \backslash A_m}   (2-|g_k|^2-|g_{-k}|^2)e^{pG_k^j(u_N)} d\mu \right|.\\
\end{split}
\end{equation}

By Proposition \ref{963} we know that $|g_k|^2 + |g_{-k}|^2$ and $e^{pG_k^j(u_N)}$ have non-negative covariance on the set $\tilde{\mathcal{A}_m^k}$, so we infer that  
\begin{equation} \label{966}
\begin{split}
\int_{\tilde{\mathcal{A}_m^k}}   (2-|g_k|^2-&|g_{-k}|^2)e^{pG_k^j(u_N)} d\mu \\
 &\le \int_{\tilde{\mathcal{A}_m^k}}   (2-|g_k|^2-|g_{-k}|^2)d\mu  \cdot \frac{1}{\mu(\tilde{\mathcal{A}_m^k})}\int_{\mathcal{A}_m^k} e^{pG_k^j(u_N)} d\mu \\
 &\le \frac{C(p)}{\mu(\tilde{\mathcal{A}^k_m})}\left( \left|\int_{A_m}   (2-|g_k|^2-|g_{-k}|^2) d\mu   \right|+ \left|\int_{\tilde{\mathcal{A}_m^k} \backslash A_m}   (2-|g_k|^2-|g_{-k}|^2)d\mu\right| \right) .\\
\end{split}
\end{equation}
Combining equations \eqref{965} and \eqref{966} and noting that $A_m \subset \tilde{\mathcal{A}_m^k}$, we have 
\begin{equation} \label{967}
\begin{split}
\frac{\int_{A_m}   (2-|g_k|^2-|g_{-k}|^2)e^{pG_k^j(u_N)} d\mu }{\int_{A_m} 2-|g_k|^2-|g_{-k}|^2  d\mu}
&\le \frac{C(p)}{\mu(A_m)}\left(1 + \frac{\left|\int_{\tilde{\mathcal{A}_m^k} \backslash A_m}   (2-|g_k|^2-|g_{-k}|^2)d\mu \right|}{\int_{A_m} 2-|g_k|^2-|g_{-k}|^2  d\mu}\right) \\
&\qquad \quad + \frac{| \int_{\tilde{\mathcal{A}_m^k} \backslash A_m}   (2-|g_k|^2-|g_{-k}|^2)e^{pG_k^j(u_N)} d\mu |}{\int_{A_m} 2-|g_k|^2-|g_{-k}|^2  d\mu}. \\
\end{split}
\end{equation}
By H\"older's inequality, we have
\begin{equation*} \label{968}
\begin{split}
\left| \int_{\tilde{\mathcal{A}_m^k} \backslash A_m}   (2-|g_k|^2-|g_{-k}|^2)d\mu \right|  &\le \mu(\tilde{\mathcal{A}_m^k} \backslash A_m)^{7/8} \left(\int_{A_{2m}} (2-|g_k|^2-|g_{-k}|^2)^{8}d\mu\right)^{1/8}, \\
\end{split}
\end{equation*}
as well as
\begin{equation*} \label{968}
\begin{split}
\left|\int_{\tilde{\mathcal{A}_m^k} \backslash A_m}   (2-|g_k|^2-|g_{-k}|^2)e^{pG_k^j(u_N)} d\mu \right| &\le  \mu(\tilde{\mathcal{A}_m^k} \backslash A_m)^{7/8}\cdot \left( \int_{A_{2m}} e^{16pG_k^j(u_N)} d\mu  \right)^{1/16}\\
&\qquad \quad \cdot \left(\int_{A_{2m}} (2-|g_k|^2-|g_{-k}|^2)^{16}d\mu\right)^{1/16} . \\
\end{split}
\end{equation*}
Since $g_k$, $g_{-k}$ are Gaussian the expectation of each raised to a finite power is bounded. Applying Lemma \ref{709}, Lemma \ref{959}  and Lemma \ref{981}, we have
\begin{equation} \label{969}
\begin{split}
\frac{\int_{A_m}   (2-|g_k|^2-|g_{-k}|^2)e^{pG_k^j(u_N)} d\mu }{\int_{A_m} 2-|g_k|^2-|g_{-k}|^2  d\mu} &\le \frac{C(p)}{\mu(A_m)}\left( 1 + \frac{\mu(\tilde{\mathcal{A}_m^k} \backslash A_m)^{7/8}}{\int_{A_m} 2-|g_k|^2-|g_{-k}|^2  d\mu}\right)  \\
&\le  \frac{C(p)e^m\langle k \rangle^{1/4}}{\mu(A_m)m^2\int_{m/2}^{m} P_k\left( m-\frac{y}{\langle k \rangle^2}\right)  dy}.
\end{split}
\end{equation}

This completes the proof. In the $k=0$ case we merely need to bound the integral by a constant, which can be accomplished by the same argument.  
\end{proof}

All that is left is to bound $\frac{\partial}{ \partial s}|_{s=1}G_k^j(T_s^k(u_N)) e^{pG_k^j(u_N)}$. For each value of $j$ this is bounded by a constant, but for higher values of $j$ we can bound it by a power of $\frac{1}{\langle k \rangle}$. 
\begin{lemma} \label{970}

For $j \in \{1,2\}$ and any $m < m_{12p}$
\begin{equation*}
\int_{A_m} \frac{\partial}{ \partial s}|_{s=1}G_k^j(T_s^k(u_N)) e^{pG_k^j(u_N)} d\mu  \le C(p).
\end{equation*}
\end{lemma}
\begin{proof}
First observe that since each term of $G_k^j(T_s^k(u_N))$ is proportional to $s^j$, taking the derivative at $s=1$ multiplies by a factor of $j$. By H\"older's inequality we have
\begin{equation}
\int_{A_m} \frac{\partial}{ \partial s}|_{s=1}G_k^j(T_s^k(u_N)) e^{pG_k^j(u_N)} d\mu  \le j\| G_k^j(u_N)  \|_{L^2(\mu)}\|  e^{pG_k^j(u_N)} \|_{L^2(A_m,d\mu)}.
\end{equation}
By Proposition \ref{24}, the first term is bounded by a constant, and by Lemma \ref{959} the second term is bounded by $C(p)$.  
\end{proof}

\begin{lemma} \label{971}

For $j \in \{0,3,4\}$ and any $m < m_{12p}$
\begin{equation*}
\int_{A_m} \frac{\partial}{ \partial s}|_{s=1}G_k^j(T_s^k(u_N)) e^{pG_k^j(u_N)} d\mu  \le \frac{C(p)}{\langle k \rangle^3}.
\end{equation*}
\end{lemma}
\begin{proof}
Again, the derivative multiplies by a factor of $j$. So in the $j=0$ case the integral is just $0$.  
For $j=3,4$, we apply H\"older's inequality to obtain
\begin{equation}
\int_{A_m} \frac{\partial}{ \partial s}|_{s=1}G_k^j(T_s^k(u_N)) e^{pG_k^j(u_N)} d\mu  \le j\| G_k^j(u_N)  \|_{L^2(\mu)}\|  e^{pG_k^j(u_N)} \|_{L^2(A_m,d\mu)}.
\end{equation} 
By Lemma \ref{959} the second term is bounded by $C(p)$, however we can achieve a much stronger bound on the first term since for $j=3,4$, $G_k^j(u_N)$ only contains finitely many terms. 

When $j=4$, $G_k^j(u_N) = \frac{3k(|g_k|^4 - |g_{-k}|^4)}{2\langle k \rangle^4}$. Squaring yields $\int_{H^{1/2-}(\mathbb{T})} |G_k^j(u_N)|^2 d\mu \le   \frac{C}{\langle k \rangle^6}$ and therefore $\| G_k^4(u_N)  \|_{L^2(\mu)} \le \frac{C}{\langle k \rangle^3}$.

When $j=3$, three of the four factors must be $g_{k}$ or $g_{-k}$, so every term $\frac{n_1g_{n_1}g_{n_2}\overline{g_{n_3}}\overline{g_{n_4}}}{\langle n_1 \rangle \langle n_2 \rangle\langle n_3 \rangle\langle n_4 \rangle}$ has indices equal to the multiset $\{n_1,n_2,n_3,n_4\} = \{k,k,-k,3k\}$ or the multiset $\{n_1,n_2,n_3,n_4\} = \{-k,-k,k,-3k\}$. So when we expand $\int_{H^{1/2-}(\mathbb{T})} |G_k^j(u_N)|^2 d\mu$ we have a sum of finitely many terms divided by $\frac{1}{\langle k \rangle^6}$. Therefore  $\| G_k^3(u_N)  \|_{L^2(\mu)} \le \frac{C}{\langle k \rangle^3}$.

\end{proof}

We now have all the tools necessary to bound $\int_{H^{1/2-}(\mathbb{T})} e^{pG_k^j(u_N)} d\nu_m^k$.

\begin{proposition} \label{975}
Consider values of $p\ge 1$, $N \ge 0$, $k \ge 0$, and $m < \frac{1}{2}m_{96p}$.

For $j =1,2$ we have 
\begin{equation*}
\int_{H^{1/2-}(\mathbb{T})} e^{pG_k^j(u_N)} d\nu_m^k  \le \frac{C(p)e^m\langle k \rangle^{2}}{\mu(A_m)m^2\int_{m/2}^{m} P_k\left( m-\frac{y}{\langle k \rangle^2}\right)  dy}.
\end{equation*}
For $j =0,3,4$ we have 
\begin{equation*}
\int_{H^{1/2-}(\mathbb{T})} e^{pG_k^j(u_N)} d\nu_m^k   \le \frac{C(p)e^m\langle k \rangle^{1/4}}{\mu(A_m)m^2\int_{m/2}^{m} P_k\left( m-\frac{y}{\langle k \rangle^2}\right)  dy}.
\end{equation*}
\end{proposition}

\begin{proof}
By Theorem \ref{960},
\begin{equation*}
\int_{H^{1/2-}(\mathbb{T})} e^{pG_k^j(u_N)} d\nu_m^k   = \frac{\int_{A_m}   (2-|g_k|^2-|g_{-k}|^2)e^{pG_k^j(u_N)}d\mu}{\int_{A_m} 2-|g_k|^2-|g_{-k}|^2  d\mu }+  \frac{\int_{A_m} \frac{p}{2}\frac{\partial}{ \partial s}|_{s=1}G_k^j(T_s^k(u_N))e^{pG_k^j(u_N)}  d\mu }{\int_{A_m} 2-|g_k|^2-|g_{-k}|^2  d\mu }. 
\end{equation*}

 Proposition \ref{964} bounds the first term by
\begin{equation}
\frac{\int_{A_m}   (2-|g_k|^2-|g_{-k}|^2)e^{pG_k^j(u_N)} d\mu}{\int_{A_m} 2-|g_k|^2-|g_{-k}|^2  d\mu } \le \frac{ C(p)e^m\langle k \rangle^{1/4}}{\mu(A_m)m^2\int_{m/2}^{m} P_k\left( m-\frac{y}{\langle k \rangle^2}\right)  dy}.
\end{equation}

By Lemma \ref{970} and Lemma \ref{971} the second term is bounded by  
\begin{equation}
 \frac{ \int_{A_m} \frac{p}{2}\frac{\partial}{ \partial s}|_{s=1}G_k^j(T_s^k(u_N))e^{pG_k^j(u_N)}  d\mu }{\int_{A_m} 2-|g_k|^2-|g_{-k}|^2  d\mu } \le  \frac{C(p)e^m\langle k \rangle^{2} }{\mu(A_m)m^2\int_{m/2}^{m} P_k\left( m-\frac{y}{\langle k \rangle^2}\right)  dy},
\end{equation}
for $j =1,2$ and
\begin{equation}
 \frac{\int_{A_m} \frac{p}{2}\frac{\partial}{ \partial s}|_{s=1}G_k^j(T_s^k(u_N))e^{pG_k^j(u_N)}  d\mu }{\int_{A_m} 2-|g_k|^2-|g_{-k}|^2  d\mu } \le  \frac{C(p)e^m\langle k \rangle^{-1} }{\mu(A_m)m^2\int_{m/2}^{m} P_k\left( m-\frac{y}{\langle k \rangle^2}\right)  dy},
\end{equation}
for $j=0,3,4$. Summing yields the desired result.  
\end{proof}

We can now prove our main result.

\begin{theorem} \label{1001}
For any $p \ge 1$, $k,N \ge 0$ and $m<\frac{1}{2}m_{480p} < 1$, there exists a constant $C(p)$ depending on $p$ (but not $k,N$) such that $$\int_{H^{1/2-}(\mathbb{T})} e^{pf_N(u)} d\nu^k_{m} \le \frac{ C(p)e^m\langle k \rangle^{19/20}}{\mu(A_m)m^2\int_{m/2}^{m} P_k\left( m-\frac{y}{\langle k \rangle^2}\right)  dy}.$$
\end{theorem}
\begin{proof} 
Recall that for each $k \ge 0$ we can write $f_N(u)$ as the sum 
\begin{equation*}
f_N(u) = \sum_{j=0}^{4}G_k^j(u_N).
\end{equation*}
By H\"older's inequality and Proposition \ref{975} we have 
\begin{equation}
\begin{split}
\int_{H^{1/2-}(\mathbb{T})} e^{pf_N(u)} d\nu^k_{m} &\le  \underset{j=1,2}{\Pi}  \| e^{pG_k^j(u_N)} \|_{L^5(\nu_m^k)}\cdot \Pi_{j=0,3,4} \| e^{pG_k^j(u_N)} \|_{L^5(\nu_m^k)}\\
&\le \frac{\left(C(p)e^m\langle k \rangle^{2}\right)^{2/5} \cdot \left(C(p)e^m\langle k \rangle^{1/4}\right)^{3/5}}{\mu(A_m)m^2\int_{m/2}^{m} P_k\left( m-\frac{y}{\langle k \rangle^2}\right)  dy} \\
&\le \frac{C(p)e^m\langle k \rangle^{19/20}}{\mu(A_m)m^2\int_{m/2}^{m} P_k\left( m-\frac{y}{\langle k \rangle^2}\right)  dy}. \\
\end{split}
\end{equation}

\end{proof} 
 By Theorem \ref{823} and Corollary \ref{916} we have our $\mu_m$ bound.

\begin{theorem} \label{1005}
For fixed $p \ge 1$, $m <\frac{1}{2}m_{480p}$ and each $N \ge 0$ we have $$\int_{H^{1/2-}(\mathbb{T})}  e^{pf_N(u)} d\mu_m \le  C(m,p),$$ and for sufficiently small $\epsilon >0$, $$\int_{H^{1/2-}(\mathbb{T})}  e^{pf_N(u)} d\mu_m^{\epsilon} \le  C(m,p),$$
for some finite constant depending on $m,p$.  
\end{theorem}

\begin{proof}
Applying Theorem \ref{823} with $\alpha = \frac{19}{20}$, we have 
\begin{equation}
\int_{H^{1/2-}(\mathbb{T})}  e^{pf_N(u)} d\mu_m \le C\sup_{k} \frac{\|e^{pf_N(u)}\|_{L^1(\nu^k_m)}}{\langle k \rangle^{19/20}} \le \sup_{k} \frac{C(p)e^m}{\mu(A_m)m^2\int_{m/2}^{m} P_k\left( m-\frac{y}{\langle k \rangle^2}\right)  dy} ,
\end{equation}
for each $m < \frac{1}{2}m_{480p}$. By Lemma \ref{609} there exists an interval $[m',m'']$ containing $m$ in its interior and a continuous positive function $C_0(x)$ on the interval such that for each $x \in [m',m'']$ we have $$\int_{x/2}^{x} P_k\left( x-\frac{y}{\langle k \rangle^2}\right)  dy \ge C_0(x),$$ for all $k$.

Therefore for every $x \in [m',m'']$ we have 
\begin{equation}
\int_{H^{1/2-}(\mathbb{T})}  e^{pf_N(u)} d\mu_x \le  C_0(x) \cdot \frac{C(p)e^{x}}{\mu(A_{x})x^2},
\end{equation}
which proves the first part. In addition, for any $\epsilon>0$ small enough that $(m-\epsilon,m+\epsilon) \subset [m',m'']$ we have
\begin{equation}
\int_{H^{1/2-}(\mathbb{T})}  e^{pf_N(u)} d\mu_m^{\epsilon} \le \min_{x \in [m',m'']} C_0(x)\cdot \frac{C(p)e^{m+\epsilon}}{(m-\epsilon)^2\mu(\|u\|_{L^2}^2 < (m-\epsilon))},
\end{equation} by Corollary \ref{916}.
\end{proof}

\subsection{Bounding $\int_{H^{1/2-}(\mathbb{T})} |f_N(u)-f_M(u)|^2 d\mu_m$}
We now repeat the same type of argument to bound the $L^2(\mu_m)$ norm of $f_N(u)-f_M(u)$  with respect to $\mu_m$. This will help us prove that $f_N \rightarrow f$ in $L^2(\mu_m)$. We start with the following lemma. 
 
\begin{lemma} \label{990}
For each $k \ge 0$ and $M \ge N$ we have
$$\int_{\tilde{\mathcal{A}_m^k}} (2-|g_k|^2 - |g_{-k}|^2) |f_N(u) - f_M(u)|^2 d\mu \le \frac{\int_{\tilde{\mathcal{A}_m^k}} (2-|g_k|^2 - |g_{-k}|^2) d\mu \cdot \int_{\tilde{\mathcal{A}_m^k}} |f_N(u) - f_M(u)|^2 d\mu }{\mu(\tilde{\mathcal{A}_m^k})}.$$
\end{lemma}
\begin{proof}
We will make use of a covariance argument as in Proposition \ref{963} and Proposition \ref{964}. Recall from the proof of Proposition \ref{24} that using the notation $\mathbb{A}_n$ we can write  
\begin{equation*}
\begin{split}
F(u) &= |f_M(u) - f_N(u)|^2 \\
&=\left(\frac{3}{4}\right)^2|S_M^1(u)-S_N^1(u) + S_M^2(u) -S_N^2(u)+ S_M^3(u) - S_N^3(u)|^2\\   
&= \frac{9}{16}|\sum_{N < |n_1| \le M}  \frac{2n_1|g_{n_1}|^4}{\langle n_1 \rangle^4} + 2\sum_{N < |n_1| \le M} \sum_{|n_2| \le M, |n_2| < |n_1|}  \frac{(n_1+n_2)|g_{n_1}|^2 |g_{n_2}|^2}{\langle n_1 \rangle^2 \langle n_2 \rangle^2}\\
&\qquad \qquad \qquad \qquad \qquad \qquad \qquad \qquad \qquad \quad+  \sum_{n} \sum_{\mathbb{A}_n} \frac{ng_{n_1}g_{n_2}\overline{g_{n_3}}\overline{g_{n_4}}}{\langle n_1 \rangle \langle n_2 \rangle \langle n_3 \rangle \langle n_4 \rangle}|^2\\
&= \frac{9}{16}\left|\sum_{|n_1|, |n_2| \le M, \text{not both} \le N}  \frac{(n_1+n_2)|g_{n_1}|^2 |g_{n_2}|^2}{\langle n_1 \rangle^2 \langle n_2 \rangle^2} +\sum_{n} \sum_{\mathbb{A}_n} \frac{ng_{n_1}g_{n_2}\overline{g_{n_3}}\overline{g_{n_4}}}{\langle n_1 \rangle \langle n_2 \rangle \langle n_3 \rangle \langle n_4 \rangle}\right|^2.\\
\end{split}
\end{equation*} 
By symmetry, we have 
\begin{equation}
\begin{split}
0 &=\int_{\tilde{\mathcal{A}_m^k}} (2-|g_k|^2-|g_{-k}|^2)(S_N^1(u)-S_M^1(u))(S_N^3(u)-S_M^3(u)) d\mu\\
&= \int_{\tilde{\mathcal{A}_m^k}} (2-|g_k|^2-|g_{-k}|^2)(S_N^2(u)-S_M^2(u))(S_N^3(u)-S_M^3(u)) d\mu.\\
\end{split}
\end{equation}
So we can ignore the cross terms and attempt to bound $$\int_{\tilde{\mathcal{A}_m^k}} (2-|g_k|^2-|g_{-k}|^2)(S_N^1-S_M^1+S_N^2-S_M^2)^2 d\mu,$$ $$\int_{\tilde{\mathcal{A}_m^k}} (2-|g_k|^2-|g_{-k}|^2)(S_N^3-S_M^3)^2 d\mu.$$ We do so by applying a covariance argument to the integral over $\tilde{\mathcal{A}_m^k}$.

We start by bounding $\int_{\tilde{\mathcal{A}_m^k}} (2-|g_k|^2-|g_{-k}|^2)(S_N^1-S_M^1+S_N^2-S_M^2)^2 d\mu$.
Expanding \newline $(S_N^1-S_M^1+S_N^2-S_M^2)^2$ yields
\begin{equation*}
\begin{split}
(S^1_N(u)-&S^1_M(u)+S^2_N(u)-S^2_M(u))^2 \\
&= \frac{9}{16}\sum_{|n_1|, |n_2| \le M, \text{not both} \le N}\sum_{|n_3|, |n_4| \le M, \text{not both} \le N} \frac{(n_1+n_2)(n_3+n_4)|g_{n_1}|^2 |g_{n_2}|^2|g_{n_3}|^2 |g_{n_4}|^2}{\langle n_1 \rangle^2 \langle n_2 \rangle^2 \langle n_3 \rangle^2 \langle n_4 \rangle^2} \\
&= \frac{9}{4}\sum_{|n_1|, |n_2| \le M, \text{not both} \le N}\sum_{|n_3|, |n_4| \le M, \text{not both} \le N} \frac{n_1n_3|g_{n_1}|^2 |g_{n_2}|^2|g_{n_3}|^2 |g_{n_4}|^2}{\langle n_1 \rangle^2 \langle n_2 \rangle^2 \langle n_3 \rangle^2 \langle n_4 \rangle^2} \\
&=  \frac{9}{4}\sum_{n_i} \frac{n_1n_3(|g_{n_1}|^2-|g_{-n_1}|^2) (|g_{n_2}|^2+|g_{-n_2}|^2)(|g_{n_3}|^2 - |g_{-n_3}|^2)(|g_{n_4}|^2+|g_{-n_4}|^2)}{\langle n_1 \rangle^2 \langle n_2 \rangle^2 \langle n_3 \rangle^2 \langle n_4 \rangle^2}. \\
\end{split}
\end{equation*} 
When we multiply each term by $|g_k|^2 + |g_{-k}|^2$ or by a constant, the integral over \newline $\tilde{\mathcal{A}_m^k}$ is $0$ unless $n_1 =n_3$. In this case we show that the integrand has non-negative covariance with $|g_k|^2+|g_{-k}|^2$.

For each term in which $n_1 = n_3$, the expression $$\frac{n_1^2(|g_{n_1}|^2-|g_{-n_1}|^2)^2 (|g_{n_2}|^2+|g_{-n_2}|^2)(|g_{n_4}|^2+|g_{-n_4}|^2)}{\langle n_1 \rangle^4 \langle n_2 \rangle^2 \langle n_4 \rangle^2}$$ is  positive. This term and $|g_k|^2 + |g_{-k}|^2$ are independent when none of the $n_i$ are equal to $\pm k$, and they clearly have non-negative covariance when $k = \pm n_2$ or $\pm n_4$ on our set.

It remains to show that when $k = \pm n_1$ the terms have non-negative covariance. This is equivalent to showing that $|g_k|^2 + |g_{-k}|^2$ and $(|g_{k}|^2-|g_{-k}|^2)^2$ have non-negative covariance, since $g_{k}$ and $g_{-k}$ are independent from the rest of the $g_n$ on the set $\tilde{\mathcal{A}_m^k}$. 

As in the proof of Proposition \ref{963}, we set  $r(\tilde{a}_k + i\tilde{b}_k) = g_k$, $r(\tilde{a}_{-k} + i\tilde{b}_{-k}) = g_{-k}$ and note that $r$ is independent from the angular components. We have 
\begin{equation*}
\begin{split}
\mathrm{Cov}(&|g_k|^2 + |g_{-k}|^2, (|g_{k}|^2-|g_{-k}|^2)^2)\\
&= \mathrm{Cov}(r^2,r^4(|\tilde{a}_{k}|^2 + |\tilde{b}_{k}|^2 - |\tilde{a}_{-k}|^2 - |\tilde{b}_{-k}|^2)^2  )\\
&=\mathbb{E}(r^6(|\tilde{a}_{k}|^2 + |\tilde{b}_{k}|^2 - |\tilde{a}_{-k}|^2 - |\tilde{b}_{-k}|^2)^2) - \mathbb{E}(r^2)\mathbb{E}(r^4(|\tilde{a}_{k}|^2 + |\tilde{b}_{k}|^2 - |\tilde{a}_{-k}|^2 - |\tilde{b}_{-k}|^2)^2)\\
&= \left(\mathbb{E}(r^6)-\mathbb{E}(r^2)\mathbb{E}(r^4)\right)\mathbb{E}((|\tilde{a}_{k}|^2 + |\tilde{b}_{k}|^2 - |\tilde{a}_{-k}|^2 - |\tilde{b}_{-k}|^2)^2)\\
&=\mathrm{Cov}(r^4,r^2) \cdot \mathbb{E}((|\tilde{a}_{k}|^2 + |\tilde{b}_{k}|^2 - |\tilde{a}_{-k}|^2 - |\tilde{b}_{-k}|^2)^2).\\
\end{split}
\end{equation*}

The second factor is the expectation of a square, which is non-negative, and the first factor is a covariance that is non-negative by applying Lemma \ref{9620} to $Y=r^4$, $X = r^2$ and $E(y) = \sqrt{y}$. Observe that $\mathbb{E}(r^2 | r^4 =y) = \sqrt{y}$, a non-decreasing function. Therefore $$\mathrm{Cov}\left (\frac{n_1^2(|g_{n_1}|^2-|g_{-n_1}|^2)^2 (|g_{n_2}|^2+|g_{-n_2}|^2)(|g_{n_4}|^2+|g_{-n_4}|^2)}{\langle n_1 \rangle^4 \langle n_2 \rangle^2 \langle n_4 \rangle^2},|g_k|^2 + |g_{-k}|^2 \right) \ge 0$$.

Now we bound $\int_{\tilde{\mathcal{A}_m^k}} (2-|g_k|^2-|g_{-k}|^2)(S_N^3(u))^2 d\mu$. Keeping the notation $\mathbb{A}_n$ of Proposition \ref{24}, we expand $(S_N^3(u)-S_M^3(u))^2$ to obtain
\begin{equation*}
\begin{split}
 \int_{\tilde{\mathcal{A}_m^k}} &(2-|g_k|^2 - |g_{-k}|^2)(S_N^3(u)-S_M^3(u))^2 d\mu \\
 &=  \frac{9}{4}\int_{\tilde{\mathcal{A}_m^k}} (2-|g_k|^2 - |g_{-k}|^2)\sum_{n} \sum_{\mathbb{A}_n} \frac{n^2|g_{n_1}|^2 |g_{n_2}|^2|g_{n_3}|^2 |g_{n_4}|^2}{\langle n_1 \rangle^2 \langle n_2 \rangle^2 \langle n_3 \rangle^2 \langle n_4 \rangle^2}, \\
\end{split}
\end{equation*} 
and clearly each term inside the sum is positive and has non-negative covariance with $|g_k|^2 + |g_{-k}|^2$ on our set. 

Combining these two cases, and the cancellation of the cross terms, we have 
\begin{equation}
\textstyle \int_{\tilde{\mathcal{A}_m^k}} (2-|g_k|^2 - |g_{-k}|^2) |f_N(u) - f_M(u)|^2 d\mu \le \frac{\int_{\tilde{\mathcal{A}_m^k}} (2-|g_k|^2 - |g_{-k}|^2) d\mu \cdot \int_{\tilde{\mathcal{A}_m^k}} |f_N(u) - f_M(u)|^2 d\mu }{\mu(\tilde{\mathcal{A}_m^k})}.
\end{equation} 

\end{proof}

\begin{lemma} \label{940}
For each $m>0$, and $M \ge N$ we have 
$$\int_{H^{1/2-}(\mathbb{T})} |f_N(u) - f_M(u)|^2 d\mu_m  \le \frac{C(m)}{N}$$ for all $N \ge 0$. 
\end{lemma}

\begin{proof}
Fix a value of $m$. We start with the decomposition of $\mu_m$ from Theorem \ref{823}:
\begin{equation}
\int_{H^{1/2-}(\mathbb{T})} |f_N(u) - f_M(u)|^2 d\mu_m = \frac{1}{m} \lim_{N \rightarrow \infty}  \sum_{0\le k \le N} \int \frac{c_{k,m}}{\langle k \rangle^2} |f_N(u) - f_M(u)|^2 d\nu_m^k
\end{equation}

By Theorem \ref{960}, we have 
\begin{equation*}
\begin{split}
\int_{H^{1/2-}(\mathbb{T})} |f_N(u) - f_M(u)|^2 d\mu_m &\le  \frac{\sup_{k}c_{k,m}}{m} \sum_{k=0}^{\infty}\frac{1}{\langle k \rangle^2}  \frac{\int_{A_m}   (2-|g_k|^2-|g_{-k}|^2)|f_N(u) - f_M(u)|^2 d\mu}{\int_{A_m} 2-|g_k|^2-|g_{-k}|^2  d\mu }\\
&\quad + \frac{\sup_{k}c_{k,m}}{2m} \sum_{k=0}^{\infty}\frac{1}{\langle k \rangle^2}  \frac{\int_{A_m}\frac{\partial}{ \partial s}|_{s=1} |f_N(T_s^k(u)) - f_M(T_s^k(u))|^2  d\mu }{\int_{A_m} 2-|g_k|^2-|g_{-k}|^2  d\mu }. \\
\end{split}
\end{equation*}
By Lemma \ref{709} and Lemma \ref{609} we can bound the denominator by $\frac{1}{\langle k \rangle^2}$ and obtain 
\begin{equation} \label{989}
\begin{split}
&\int_{H^{1/2-}(\mathbb{T})} |f_N(u) - f_M(u)|^2 d\mu_m \\
&\qquad \le  \frac{C}{m} \sum_{k=0}^{\infty} \int_{A_m}   (2-|g_k|^2-|g_{-k}|^2)|f_N(u) - f_M(u)|^2 +\frac{\partial}{ \partial s}|_{s=1} |f_N(T_s^k(u)) - f_M(T_s^k(u))|^2d\mu. \\
\end{split}
\end{equation}
Observe that each term of $|f_N(T_s^k(u)) - f_M(T_s^k(u))|^2$ contains $8$ factors of the form $g_{n}$ or $\overline{g_n}$, each of which contributes one term to the derivative $\partial s$ at $k=n$. Therefore 
\begin{equation*}
\begin{split}
\sum_{k=0}^{\infty} \int_{A_m}  \frac{\partial}{ \partial s}|_{s=1} |f_N(T_s^k(u)) - f_M(T_s^k(u))|^2d\mu  = \int_{A_m}8 |f_N(u) - f_M(u)|^2d\mu,\\
\end{split}
\end{equation*}
which is bounded by $\frac{C}{N}$ by Proposition \ref{24}.

Now we turn to the first term of equation \eqref{989}. By the triangle inequality, Lemma \ref{990} and  H\"older's inequality with exponents equal to $(\frac{3}{2},6,6)$, we have
\begin{equation*}
\begin{split}
&\int_{A_m}   (2-|g_k|^2-|g_{-k}|^2)|f_N(u) - f_M(u)|^2 d\mu \\
&\le |\int_{\tilde{A_m^k}\backslash A_m}  (2-|g_k|^2-|g_{-k}|^2)|f_N(u) - f_M(u)|^2 d\mu|+ \int_{\tilde{A_m^k}}  (2-|g_k|^2-|g_{-k}|^2)|f_N(u) - f_M(u)|^2 d\mu \\
&\le \mu(\tilde{A_m^k}\backslash A_m)^{2/3} \|  (2-|g_k|^2-|g_{-k}|^2)\|_{L^6(\mu)}  \|f_N(u) - f_M(u)  \|^2_{L^{12}(\mu)}\\
&\qquad \qquad + \frac{\int_{\tilde{\mathcal{A}_m^k}} (2-|g_k|^2 - |g_{-k}|^2) d\mu \cdot \int_{\tilde{\mathcal{A}_m^k}} |f_N(u) - f_M(u)|^2 d\mu }{\mu(\tilde{\mathcal{A}_m^k})}.\\
\end{split}
\end{equation*}

Lemma \ref{981} implies that $\mu(\tilde{A_m^k}\backslash A_m)^{2/3} \le \frac{C}{\langle k \rangle^{4/3}}$. By Proposition \ref{24} and the Wiener Chaos bound in  Proposition \ref{20}, the $L^p(\mu)$ norms of $|f_N(u) - f_M(u)|$ squared are bounded by $\frac{1}{N}$. The expected value of a Gaussian to a finite power is finite, and Lemma \ref{709} implies that $$\int_{\tilde{\mathcal{A}_m^k}} (2-|g_k|^2 - |g_{-k}|^2) d\mu \le \frac{C}{\langle k \rangle^2}.$$ So we have 
\begin{equation*}
\int_{A_m}   (2-|g_k|^2-|g_{-k}|^2)|f_N(u) - f_M(u)|^2 d\mu \le \frac{C}{\mu(\tilde{\mathcal{A}_m^k})N\langle k \rangle^{4/3}}.
\end{equation*}
We conclude that 
\begin{equation} \label{9890}
\begin{split}
\int_{H^{1/2-}(\mathbb{T})} |f_N(u) - f_M(u)|^2 d\mu_m &\le  \frac{C}{m}(\frac{C}{N} +\sum_{k=0}^{\infty}\frac{C}{\mu(\tilde{\mathcal{A}_m^k})N\langle k \rangle^{4/3}} ) \\
&\le \frac{C(m)}{N}.\\
\end{split}
\end{equation}
\end{proof}

\subsection{Bounding $\|e^{f(u)}\|_{L^p(H^{1/2-}(\mathbb{T}), d\mu_m)}$} \label{mathrefs}
At this point we have the tools necessary to bound the integral of $e^{f(u)}$ with respect to the measure $\mu_m$. We already have a bound on $e^{pf_N(u)}$, so we merely need to pass this bound to the limit as $N \rightarrow \infty$. Proving the convergence in $\mu_m$-measure of $f_N(u) \rightarrow f(u)$ will enable us to pass to the limit.

\begin{lemma} \label{1007}
For each $m >0$ the sequence $f_N(u)$ of functions converges in $L^2(\mu_m)$ as well as in $\mu_m$-measure to $f(u)$.
\end{lemma}
\begin{proof} 
By Lemma \ref{940} we have 
\begin{equation} \label{1008}
\int_{H^{1/2-}(\mathbb{T})}|f_N(u)-f_M(u)|^2 d\mu_m \le \frac{C(m)}{N},
\end{equation}
for each $M \ge N$. 

This implies that the sequence of functions $f_N(u)$ is a Cauchy sequence, and thus converges, in $L^2(\mu_m)$. Therefore $f_N \rightarrow f$ in $L^2(\mu_m)$ which implies that $f_N \rightarrow f$ in $\mu_m$-measure. 
\end{proof}
We now prove the main result. 
\begin{theorem} \label{1010}
For any $p \ge 1$ and $m<\frac{1}{2}m_{480p}$, $e^{pf(u)} \in L^1(\mu_m)$ with $$\int_{H^{1/2-}(\mathbb{T})} e^{pf(u)} d\mu_m \le  C(m,p).$$  
\end{theorem}
\begin{proof} 
Convergence in measure implies existence of a subsequence that converges almost surely. Therefore there is a sequence $f_{N_j}(u) $ that converges to $f(u)$ $\mu_m$-almost surely, which in turn means $e^{pf_{N_j}(u)} \rightarrow e^{pf(u)}$ almost surely.

Recall from Theorem \ref{1005} that for each $N$ we have $\int_{H^{1/2-}(\mathbb{T})}  e^{pf_N(u)} d\mu_m \le  C(m,p)$.  By Fatou's lemma, 
\begin{equation} \label{1011}
\int_{H^{1/2-}(\mathbb{T})} e^{pf(u)} d\mu_m \le \liminf_{N_j} \int_{H^{1/2-}(\mathbb{T})}  e^{pf_{N_j}(u)} d\mu_m \le  C(m,p).
\end{equation} 
\end{proof}
The same argument, when utilizing the second part of Theorem \ref{1005} and the fact that $f_N(u) \rightarrow f(u)$ in $L^2(\mu)$, yields the following corollary for $\mu_m^{\epsilon}$.

\begin{corollary} \label{1012}
For any $m< \frac{1}{2}m_{480p}$, there exists a sufficiently small $\delta$ such that for $\epsilon \in (0,\delta)$ we have 
$$\int_{H^{1/2-}(\mathbb{T})} e^{pf(u)} d\mu^{\epsilon}_m \le C(m,p).$$  
\end{corollary}

We now demonstrate that not only is $e^{f(u)} \in L^p(\mu_m)$ but we also have the convergence $e^{f_N(u)} \rightarrow e^{f(u)}$ in $L^p(\mu_m)$, which  allows us to approximate $e^{f(u)}$ by $e^{f_N(u)}$.

\begin{lemma} \label{1015}
For all $m <  \frac{1}{2}m_{960p}$ the sequence $e^{f_N(u)}$ converges to $e^{f(u)}$ in $L^p(\mu_m)$.
\end{lemma}
\begin{proof}

For any $\alpha > 0$, $p \ge 1$ and large $N$ we can split up the integral as follows,

\begin{equation}
\begin{split}
\|e^{f_N(u)} - e^{f(u)}\|^p_{L^p(d\mu_m)} &= \int_{e^{f_N(u)} - e^{f(u)} < \alpha} |e^{f_N(u)} - e^{f(u)}|^p d\mu_m + \int_{e^{f_N(u)} - e^{f(u)} \ge \alpha} |e^{f_N(u)} - e^{f(u)}|^p d\mu_m. \\
\end{split}
\end{equation}
We now bound each individual part of the integral:
\begin{equation*}
\begin{split}
\|e^{f_N(u)} - e^{f(u)}&\|^p_{L^p(d\mu_m)} \\
&\le \alpha^p + \mu_m(|e^{f_N(u)} - e^{f(u)}| \ge \alpha)^{1/2} \cdot (\int |e^{f_N(u)} - e^{f(u)}|^{2p} d\mu_m)^{1/2} \\
&\le \alpha^p + C(m,p)\mu_m(|e^{f_N(u)} - e^{f(u)}| \ge \alpha)^{1/2}  \\
&\le \alpha^p + C(m,p)\mu_m(e^{f(u)}|e^{f_N(u) - f(u)} -1| \ge \alpha)^{1/2}  \\
&\le \alpha^p + C(m,p)\left( \mu_m(e^{f(u)} > 1/\alpha) +\mu_m(|e^{f_N(u) - f(u)} -1| \ge \alpha^2 \right)^{1/2}\\
&\le \alpha^p + C(m,p)\left( \mu_m(e^{f(u)} > 1/\alpha)^{1/2} + (\mu_m(|f_N(u) - f(u)| \ge \ln(1+\alpha^2) )^{1/2} \right).\\
\end{split}
\end{equation*}
Taking the limsup as $N \rightarrow \infty$, and then as $\alpha \rightarrow 0$ we have 
\begin{equation}
\begin{split}
\limsup_{N \rightarrow \infty}\|e^{f_N(u)} - e^{f(u)}\|^p_{L^p(d\mu_m)}  &\le  \alpha^p + C(m,p)\mu_m(e^{f(u)} > 1/\alpha)^{1/2}\\
&\qquad \qquad +\limsup_{N \rightarrow \infty}  C(m,p)\mu_m(|f_N(u) - f(u)| \ge \ln(1+\alpha^2) )^{1/2}\\
\limsup_{N \rightarrow \infty} \|e^{f_N(u)} - e^{f(u)}\|^p_{L^p(d\mu_m)} &\le \alpha^p + C(m,p)\mu_m(e^{f(u)} > 1/\alpha)^{1/2}\\
\end{split}
\end{equation}
 by convergence in $\mu_m$-measure. Taking the $\liminf$ over all $\alpha >0$ we have  $$\limsup_{N \rightarrow \infty} \|e^{f_N(u)} - e^{f(u)}\|_{L^p(d\mu_m)}  = 0,$$ which proves convergence in $L^p(\mu_m)$ for any $p \ge 1$. 
\end{proof}
In addition we show convergence in $\mu_m^{\epsilon}$ for any $\epsilon > 0$.

\begin{lemma} \label{1016}
For sufficiently small $m$ we have the convergence $e^{f_N(u)} \rightarrow e^{f(u)}$ in $L^p(\mu_m^{\epsilon})$.
\end{lemma}
\begin{proof} 
From Section 3 of \cite{TT}, we know that $f_N(u) \rightarrow f(u)$ in $\mu$-measure and therefore in $\mu_m^{\epsilon}$-measure. This allows us to duplicate the proof of Lemma \ref{1015}.
\end{proof}

\subsection{The measures $\rho_{m}^{\epsilon}$ and $\rho_m$.} \label{mathrefs}
Now that we have proven $e^{f(u)} \in L^p(\mu)$ and that \newline $e^{f_N(u)} \rightarrow e^{f(u)}$ in $L^p(\mu_m)$ and $L^p(\mu_m^{\epsilon})$, we can construct the actual invariant measures built upon the energy term integrated against the measures $\mu_m^{\epsilon}$, $\mu_m$. We know from the previous section that $\int_{H^{1/2-}(\mathbb{T})} e^{f(u)-\frac{1}{2}\int_{\mathbb{T}} |u|^6 dx} d\mu_m$ is finite, so we next must show that this integral is non-zero, which implies that $\beta_m = \left ( \int e^{f(u)-\frac{1}{2}\int_{\mathbb{T}} |u|^6} d\mu_m  \right )^{-1}$ is finite and non-zero.
\begin{lemma} \label{1200}
For fixed $m < m_{1}$ we have $\int_{H^{1/2-}(\mathbb{T})} e^{f(u)-   \frac{1}{2}\int_{\mathbb{T}} |u|^6dx} d\mu_m > 0$. 
\end{lemma} 
\begin{proof} The integral is equal to $\mathbb{E}_{\mu_m} ( e^{f(u)-  \frac{1}{2}\int_{\mathbb{T}} |u|^6dx})$ which by Jensen's inequality is \newline $\ge e^{\mathbb{E}_{\mu_m} [f(u)-   \frac{1}{2}\int_{\mathbb{T}} |u|^6 dx]}$. So it suffices to prove that $ \mathbb{E}_{\mu_m} (f(u)-   \frac{1}{2}\int_{\mathbb{T}} |u|^6 dx) > -\infty$.

By H\"older's inequality, we have
\begin{equation} \label{1201}
\begin{split} 
\mathbb{E}_{\mu_m} (f(u) ) &\ge -  \mathbb{E}_{\mu_m} (|f(u)| ) \\
&\ge -\mathbb{E}_{\mu_m} (|f(u)|^2 )^{1/2}.\\
\end{split}
\end{equation} Lemma \ref{1007} tells us that $f_N \rightarrow f$ in $L^2$ and therefore $f(u)$ is bounded in $L^2$. So this term is bounded below. 

By Proposition \ref{829}, $ \mathbb{E}_{\mu_m} (\int_{\mathbb{T}} |u|^6 dx) \le C(m)  < \infty$. So this expectation is finite.  
\end{proof}
We conclude that we can define two classes of probability measures as follows:

\begin{definition} \label{1203}
For any sufficiently small $m, \epsilon > 0$ we define the measure $\rho_m^{\epsilon}$ as follows: for each measurable $E \subset H^{1/2-}(\mathbb{T})$ let \begin{equation*}
\rho_m^{\epsilon}(E) = \beta_{m,\epsilon} \int_{E} e^{f(u)-\frac{1}{2}\int_{\mathbb{T}} |u|^6dx}  d\mu_m,
\end{equation*} for an appropriate normalizing constant $\beta_{m,\epsilon}$. 
\end{definition} 

\begin{definition} \label{1204}
For any sufficiently small $m$  define the measure $\rho_m$ by letting for any measurable set $E$, 
\begin{equation*}
\rho_m(E) = \beta_{m} \int_{E} e^{f(u)- \frac{1}{2} \int_{\mathbb{T}}|u|^6dx}  d\mu_m,
\end{equation*}
 for an appropriate normalizing constant $\beta_{m}$. 
\end{definition} 

These are the most natural definitions of the measures $\rho_m^{\epsilon}$ and $\rho_m$, however, by not defining $\rho_m$ to be the limit of $\rho_m^{\epsilon}$ we still have to prove the weak convergence $\rho_m^{\epsilon} \rightarrow \rho_m$,  and prove that $\underset{\epsilon \rightarrow 0}{\lim} \, \beta_{m,\epsilon} = \beta_m$. We will need the following key lemma. 

\begin{lemma} \label{1205}
For any sufficiently small $m$ and any bounded, uniformly continuous function $F(u)$ on $H^{\sigma}(\mathbb{T})$  we have $$\lim_{\epsilon \rightarrow 0 }\int_{H^{\sigma}(\mathbb{T})} F(u)e^{f(u)-\frac{1}{2}\int_{\mathbb{T}}|u|^6dx}  d\mu_{m}^{\epsilon} =  \int_{H^{\sigma}(\mathbb{T})} F(u)e^{f(u)-\frac{1}{2}\int_{\mathbb{T}}|u|^6dx}  d\mu_{m}.$$ 
\end{lemma}

\begin{proof}

Fix a large value of $N$. Repeated application of the triangle inequality implies that for each $\epsilon >0$ we have
\begin{equation*}
\begin{split}
&\left| \int_{ H^{\sigma}(\mathbb{T})}F(u)e^{f(u)-\frac{1}{2}\int|u|^6dx}  d\mu_{m}^{\epsilon} - \int_{H^{\sigma}(\mathbb{T})} F(u)e^{f(u)-\frac{1}{2}\int|u|^6dx}  d\mu_{m} \right|\\ 
&\le \int_{H^{\sigma}(\mathbb{T})} |F(u)-F(u_N)|e^{f(u)-\frac{1}{2}\int|u|^6dx}  d\mu_{m}^{\epsilon}
+ \int_{H^{\sigma}(\mathbb{T})} F(u_N)e^{-\frac{1}{2}\int |u|^6dx}|e^{f(u)} - e^{f_N(u)}| d\mu_{m}^{\epsilon}\\ 
&+ \int_{H^{\sigma}(\mathbb{T})} F(u_N)e^{f_N(u)}|e^{-\frac{1}{2}\int |u|^6 dx} - e^{-\frac{1}{2}\int |u_N|^6 dx}| d\mu_{m}^{\epsilon} \\
&+ |\int_{H^{\sigma}(\mathbb{T})}  F(u_N)e^{f_N(u)-\frac{1}{2}\int|u_N|^6dx}  d\mu_{m}^{\epsilon} -\int_{H^{\sigma}(\mathbb{T})}  F(u_N)e^{f_N(u)-\frac{1}{2}\int|u_N|^6dx} d\mu_m| \\
& + \int_{H^{\sigma}(\mathbb{T})} F(u_N)e^{f_N(u)}|e^{\int -\frac{1}{2}|u_N|^6 dx}-e^{\int -\frac{1}{2}|u|^6 dx}| d\mu_{m} +  \int_{H^{\sigma}(\mathbb{T})} F(u_N)e^{-\frac{1}{2}\int |u|^6dx}|e^{f(u)} - e^{f_N(u)}| d\mu_{m}\\
&+\int_{H^{\sigma}(\mathbb{T})} |F(u)-F(u_N)|e^{f(u)-\frac{1}{2}\int|u|^6dx}  d\mu_{m}. \\
\end{split}
\end{equation*}

Uniform continuity of $F$ implies that there exists a positive, increasing function $G:\mathbb{R}^+ \rightarrow \mathbb{R}^+$ satisfying $\lim_{x \rightarrow 0} G(x) = 0$ and $|F(u)-F(v)| \le G(\|u-v\|_{H^{\sigma}})$ for all $u \ne v$.

There are seven terms, we bound them as follows:
\begin{enumerate}[leftmargin=*]
\item{} By H\"older's inequality, Lemma \ref{901} and Theorem \ref{1005}
\begin{equation*}
\begin{split}
\int_{H^{\sigma}(\mathbb{T})} |F(u)-F(u_N)|e^{f(u)-\frac{1}{2}\int|u|^6dx} d\mu_{m}^{\epsilon}
&\le \|F(u) - F(u_N)\|_{L^2(\mu_{m}^{\epsilon})} \|e^{f(u)} \|_{L^2(\mu_{m}^{\epsilon})}\\
&\le\left(\int_{\|u-u_N\|_{H^{\sigma}} > \delta} 4\|F\|_{L^{\infty}}^2d\mu_m^{\epsilon} + G(\delta) ^2 \right)^{1/2}C(m)^{1/2}\\
&\le \left((4\|F\|_{L^{\infty}}^2\mu_m^{\epsilon}(\|u-u_N\|_{H^{\sigma}} > \delta) + G(\delta)^2) C(m)\right)^{1/2}\\
&\le  \left((\frac{C\|F\|_{L^{\infty}}^2}{p_0(m)\delta^2N^{1-2\sigma}} + G(\delta)^2)C(m)\right)^{1/2}\\
&\le  \frac{C(m)\|F\|_{L^{\infty}}}{\delta N^{1/2-\sigma}} + C(m)G(\delta),\\
\end{split}
\end{equation*}
for some constant $C(m)$.

\item{} We have $\int_{H^{\sigma}(\mathbb{T})} F(u_N)e^{-\frac{1}{2}\int |u|^6dx}|e^{f(u)} - e^{f_N(u)}| d\mu_{m}^{\epsilon} \le \|F\|_{L^{\infty}}\|e^{f(u)} - e^{f_N(u)}\| _{L^1(\mu_{m}^{\epsilon})}$. By Lemma \ref{1016}, the limit as $N \rightarrow \infty$ of this quantity is $0$.

\item{} By H\"older's inequality and Theorem \ref{1010}, we have
\begin{equation*}
\begin{split}
\int_{H^{\sigma}(\mathbb{T})} F(u_N)e^{f_N(u)}|e^{-\frac{1}{2}\int |u|^6 dx} &- e^{-\frac{1}{2}\int |u_N|^6 dx}| d\mu_{m}^{\epsilon}\\ 
&\le \|F\|_{L^{\infty}} \| e^{f_N(u)}\|_{L^2(\mu_m^{\epsilon})}\| e^{- \frac{1}{2}\int |u|^6 dx} - e^{-\frac{1}{2}\int |u_N|^6 dx} \|_{L^2(\mu_m^{\epsilon})} \\
&\le \|F\|_{L^{\infty}} C(m) ^{1/2} \left(\int_{H^{\sigma}(\mathbb{T})} ( e^{- \frac{1}{2}\int |u|^6 dx} - e^{- \frac{1}{2}\int |u_N|^6 dx})^2 d\mu_m^{\epsilon} \right)^{1/2}. \\
\end{split}
\end{equation*}
Now we bound the integral $\int ( e^{-\frac{1}{2}\int |u|^6 dx} - e^{-\frac{1}{2}\int |u_N|^6 dx})^2 d\mu_m^{\epsilon}$ by splitting into the region where $|\|u\|^6_{L^6} - \|u_N\|^6_{L^6}| \le \delta$ and $|\|u\|^6_{L^6} - \|u_N\|^6_{L^6}| > \delta$, 
\begin{equation*}
\begin{split}
\int_{H^{\sigma}(\mathbb{T})} &( e^{-\frac{1}{2}\int |u|^6 dx} - e^{-\frac{1}{2}\int |u_N|^6 dx})^2 d\mu_m^{\epsilon}\\
&\le \mu_m^{\epsilon}(|\frac{1}{2}\|u\|^6_{L^6} - \frac{1}{2}\|u_N\|^6_{L^6}| > \delta) + \int_{|\frac{1}{2}\|u\|^6_{L^6} - \frac{1}{2}\|u_N\|^6_{L^6}| \le \delta} ( e^{-\frac{1}{2}\int |u|^6 dx} - e^{-\frac{1}{2}\int |u_N|^6 dx})^2 d\mu_m^{\epsilon}.\\ 
\end{split}
\end{equation*}
In the second term we can factor out whichever of $e^{-\frac{1}{2}\int |u|^6 dx}$ and $e^{-\frac{1}{2}\int |u_N|^6 dx}$ is smaller and we're left with something $\le (1-e^{-\delta})^2$. Noting that $e^x \ge 1+x$ for all $x$, the integrand is $\le \delta^2$.

By Markov's inequality, and applying H\"older's inequality twice, the first term is bounded by 
\begin{equation*}
\begin{split}
\mu_m^{\epsilon}(|\|u\|^6_{L^6} - \|u_N\|^6_{L^6}| > 2\delta) &\le \frac{1}{2\delta}\int_{H^{\sigma}(\mathbb{T})} \left| \int_{\mathbb{T}} |u|^6 - |u_N|^6 dx \right| d\mu_m^{\epsilon}\\
&\le \frac{1}{2\delta}\int_{H^{\sigma}(\mathbb{T})} \int_{\mathbb{T}} 3|u-u_N|\cdot (|u|^5 + |u_N|^5) dx  d\mu_m^{\epsilon}     \\
&\le \frac{3}{2\delta}\int_{H^{\sigma}(\mathbb{T})} \|u-u_N\|_{L^6}(\|u\|_{L^6}^5 + \|u + (u_N-u)\|_{L^6}^5) d\mu_m^{\epsilon} . \\
&\le \frac{C}{\delta} (\mathbb{E}_{\mu_m^{\epsilon}} (\|u-u_N\|_{L^6}^6))^{1/6}(\mathbb{E}_{\mu_m^{\epsilon}} (\|u\|_{L^6}^6)^{5/6}+\mathbb{E}_{\mu_m^{\epsilon}} (\|u-u_N\|_{L^6}^6)^{5/6}).\\
\end{split}
\end{equation*}
Applying Proposition \ref{829} and Corollary \ref{916}, this is $\le \frac{C(m)}{\delta N^{1/6}}$.

So we conclude that $\int (e^{- \frac{1}{2}\int |u|^6 dx} - e^{- \frac{1}{2}\int |u_N|^6 dx})^2 d\mu_m^{\epsilon} \le \frac{C(m)}{\delta N^{1/6}} + \delta^2$.

Therefore 
\begin{equation*} \label{1209}
\int_{H^{\sigma}(\mathbb{T})} F(u_N)e^{f_N(u)}|e^{-\frac{1}{2}\int |u|^6 dx} - e^{-\frac{1}{2}\int |u_N|^6 dx}| d\mu_{m}^{\epsilon} \le C(m)\|F\|_{L^{\infty}} \left( \frac{1}{\delta^{1/2}N^{1/12}} + \delta \right).
\end{equation*}

\item{} By weak convergence of $\mu_m^{\epsilon} \rightarrow \mu_m$  we have $$\lim_{\epsilon \rightarrow 0} |\int_{H^{\sigma}(\mathbb{T})}  F(u_N)e^{f_N(u)-\frac{1}{2}\int|u_N|^6dx}  d\mu_{m}^{\epsilon} - \int_{H^{\sigma}(\mathbb{T})}  F(u_N)e^{f_N(u)-\frac{1}{2}\int|u_N|^6dx} d\mu_m| =0.$$ 

\item{} By the same argument as case $(3)$ we have 
\begin{equation*} \label{1210}
\int_{H^{\sigma}(\mathbb{T})} F(u_N)e^{f_N(u)}|e^{-\frac{1}{2}\int |u|^6 dx} - e^{-\frac{1}{2}\int |u_N|^6 dx}| d\mu_{m} \le C(m)\|F\|_{L^{\infty}} \left( \frac{1}{\delta^{1/2}N^{1/12}} + \delta \right).
\end{equation*}

\item{} We have $\int_{H^{\sigma}(\mathbb{T})} F(u_N)e^{-\frac{1}{2}\int |u|^6dx}|e^{f(u)} - e^{f_N(u)}| d\mu_{m} \le \|F\|_{L^{\infty}}\|e^{f(u)} - e^{f_N(u)}\| _{L^1(\mu_{m})}$. By Lemma \ref{1015}, the limit as $N \rightarrow \infty$ of this quantity is $0$.

\item{} By the same argument as $(1)$ we have
\begin{equation*}
\begin{split}
\int_{H^{\sigma}(\mathbb{T})} |F(u)-F(u_N)|e^{f(u)-\frac{1}{2}\int |u|^6dx} d\mu_{m} &\le \|F(u) - F(u_N)\|_{L^2(\mu_{m}^{\epsilon})} \|e^{f(u)} \|_{L^2(\mu_{m}^{\epsilon})}\\
&\le \left((\frac{C\|F\|_{L^{\infty}}^2}{p_0(m)\delta^2N^{1-2\sigma}} + G(\delta)^2) C(m)\right)^{1/2}\\
&\le \frac{C(m)\|F\|_{L^{\infty}}}{\delta N^{1/2-\sigma}} + C(m)G(\delta).\\
\end{split}
\end{equation*}
\end{enumerate}
Combining these 7 cases and taking the $\limsup_{N \rightarrow \infty}$ we have that for any fixed $\delta >0$, 
$$\limsup_{\epsilon \rightarrow 0} |\int_{H^{\sigma}(\mathbb{T})} F(u)e^{f(u)-\frac{1}{2}\int|u|^6dx}  d\mu_{m}^{\epsilon} - \int_{H^{\sigma}(\mathbb{T})} F(u)e^{f(u)-\frac{1}{2}\int|u|^6dx}  d\mu_{m}| \le C(m)( \delta\|F\|_{L^{\infty}} + G(\delta)).$$
This inequality holds for all $\delta>0$. The right hand side can be made arbitrarily small by selecting $\delta$ close enough to $0$. 

Therefore $$\lim_{\epsilon \rightarrow 0} |\int_{H^{\sigma}(\mathbb{T})} F(u)e^{f(u)-\frac{1}{2}\int|u|^6dx}  d\mu_{m}^{\epsilon} - \int_{H^{\sigma}(\mathbb{T})} F(u)e^{f(u)-\frac{1}{2}\int|u|^6dx}  d\mu_{m}| = 0.$$ 

\end{proof}

Setting $F(u)=1$ we have the following corollary:
\begin{corollary} \label{1212}
For each sufficiently small $m$ we have $\underset{\epsilon \rightarrow 0}{\lim} \, \beta_{m,\epsilon} = \beta_m$.
\end{corollary}

Now we prove weak convergence of $\rho_m^{\epsilon} \rightarrow \rho_{m}$.  
\begin{theorem} \label{1215}
There exists a constant $m'$ such that for each $ m <m'$ the sequence of measures  $\rho_m^{\epsilon}$  converges weakly to $\rho_{m}$ in $H^{\sigma}(\mathbb{T})$. 
\end{theorem}
This theorem defines the constant $m'$ that will appear throughout the final section and in the statement of Theorem \ref{17}.

\begin{proof}
Consider any sufficiently small  $m >0$. It suffices to show that for any uniformly continuous function $F(u)$ on $H^{\sigma}(\mathbb{T})$  we have 
\begin{equation} \label{1216}
\lim_{\epsilon \rightarrow 0 } |\beta_{m,\epsilon} \int_{H^{\sigma}(\mathbb{T})} F(u)e^{f(u)-\frac{1}{2}\int|u|^6dx}  d\mu_{m}^{\epsilon} -  \beta_m\int_{H^{\sigma}(\mathbb{T})} F(u)e^{f(u)-\frac{1}{2}\int|u|^6dx}  d\mu_{m}| = 0.
\end{equation} The triangle inequality yields
\begin{equation}
\begin{split}
|\beta_{m,\epsilon} \int_{H^{\sigma}(\mathbb{T})} F(u)&e^{f(u)-\frac{1}{2}\int|u|^6dx}  d\mu_{m}^{\epsilon} -  \beta_m\int_{H^{\sigma}(\mathbb{T})} F(u)e^{f(u)-\frac{1}{2}\int|u|^6dx}  d\mu_{m}| \\
&\le \beta_{m,\epsilon}| \int_{H^{\sigma}(\mathbb{T})} F(u)e^{f(u)-\frac{1}{2}\int|u|^6dx}  d\mu_{m}^{\epsilon} -\int_{H^{\sigma}(\mathbb{T})} F(u)e^{f(u)-\frac{1}{2}\int|u|^6dx}  d\mu_{m}| \\
&\qquad \qquad \qquad \qquad  \qquad \quad + |\beta_{m,\epsilon}-\beta_m| \int_{H^{\sigma}(\mathbb{T})} F(u)e^{f(u)-\frac{1}{2}\int|u|^6dx}  d\mu_{m}.\\
\end{split}
\end{equation}
By Lemma \ref{1205}, the first term goes to $0$ as $\epsilon \rightarrow 0$. By Corollary \ref{1212}, the second term goes to $0$ as $\epsilon   \rightarrow 0$. 
\end{proof}

\section{Proof of invariance}
In this final section we prove the invariance of the measure $\rho_m$ and use this invariance to prove Theorem \ref{17}. We know from Theorem 1.3 of \cite{Burq} that the base measure $\rho$ is invariant with respect to equation \eqref{1} and that the set of  initial data $u_0 \in H^{1/2-}(\mathbb{T})$ that produce a global solution has $\rho$-measure $1$. Therefore this set also has $\rho_m^{\epsilon}$-measure $1$ for each $m,\epsilon$ and we will easily prove that $\rho_m^{\epsilon}$ is an invariant measure using a bump function argument. 

It remains to prove that $\rho_m$ is an invariant measure, and that the set of initial data that produce global solutions has $\rho_m$-measure $1$ for each sufficiently small $m$. Invariance will follow from weak convergence $\rho_m^{\epsilon} \rightarrow \rho_m$ and the existence of solutions will follow from the Prokhorov and Skohorod theorems, which are stated below. Burq, Thomann and Tzvetkov use these theorems repeatedly to construct invariant measures in \cite{Burq}.

In preparation for the Prokhorov and Skohorod theorems, we will need the following definition.
\begin{definition} \label{1301}
Let $\{\lambda_N, N \ge 0\}$ be a sequence of probability measures on  a metric space $\mathcal{S}$. The sequence $\{\lambda_N, N \ge 0\}$ is tight if for every $\delta  >0$ there exists a compact set $K_{\delta}$ such that $\lambda_N(K_{\delta}) \ge 1- \delta$ for all $N \ge 0$.
\end{definition}

Thus if a sequence of measures is 'tight' it means most of the support of the measures stays in a compact set, and does not diverge to infinity too quickly.
Note that $\mathcal{S}$ must be a metric space, so we will prove that for each $m < m'$ the sequence $\{ \rho^{\epsilon}_m: \epsilon > 0\}$ of measures is tight in $H^{\sigma}(\mathbb{T})$, not $H^{1/2-}(\mathbb{T})$ which is not a metric space. We now present the Prokhorov and Skohorod theorems. 
\begin{theorem*}[Prokhorov] \label{1304}
Assume that $\{\lambda_N, N \ge 0\}$ is a tight sequence of probability measures on a metric space $\mathcal{S}$. Then $\{\lambda_N, N \ge 0\}$ is weakly compact, meaning there is a subsequence $N_k$ and a measure $\lambda_{\infty}$ such that  $\lambda_{N_k} \rightarrow \lambda_{\infty}$ weakly. 
\end{theorem*}

For a proof see page 114 of \cite{PS2} or page 309 of \cite{PS1}.
\begin{theorem*}[Skohorod] \label{1305} 
Let $\{\lambda_N, N \ge 0\}$ be a sequence of probability measures on a separable metric space $\mathcal{S}$ that converges weakly to a measure $\lambda_{\infty}$. Then there exists a probability space and random variables $X_N, N \ge 1$ and $X_{\infty}$ on $\mathcal{S}$ such that the law of $X_N$ is $\lambda_N = \mathcal{L}(X_N)$, the law of $X_{\infty}$ is $\lambda_{\infty}$ and $X_N \rightarrow X_{\infty}$ almost surely.

\end{theorem*} 

For a proof see page 79 of \cite{PS1}.

The Prokhorov and Skohorod theorems tell us that if we have a tight sequence of measures on the set of solutions to the DNLS equation then there exists a weakly convergent subsequence that converges to a measure that also produces solutions to the DNLS equation. So our proof of Theorem \ref{17} will hinge on proving tightness of a sequence of measures that produce solutions to the DNLS equation. 

This highlights a deficiency of our current sequence of measures $\rho_m^{\epsilon}$. These are measures on $H^{\sigma}(\mathbb{T})$ that produce initial data at time $0$, not solutions on a time interval. We need to associate to $\rho_m^{\epsilon}$ a new sequence of measures on $C([-T,T] ;  H^{\sigma}(\mathbb{T}))$. 

\begin{definition} \label{1307}
Here we define the solution map, and an associated set. 

\begin{itemize} 

\item{} For any $t \in \mathbb{R}$ let $\Phi_t: L^2(\mathbb{T}) \rightarrow L^2(\mathbb{T})$ be the solution map of equation \eqref{1} for time $t$. So for any solution $u$ to equation \eqref{1} we have $\Phi_t[u(0,x)] = u(t,x)$. 

\item{} Let $\Phi$ be the map that takes input $u_0$ and outputs the solution to equation \eqref{1} on the largest possible interval. 

\item{} For any $T \in (0,\infty]$ let $\mathcal{S}_T \subset C([-T,T] ;  H^{\sigma}(\mathbb{T}))$ be the set of solutions to equation \eqref{1} on the interval $[-T,T]$. Observe that $\mathcal{S}_T$ is a metric space and inherits the subspace metric and topology of $C([-T,T] ;  H^{\sigma}(\mathbb{T}))$.

\end{itemize}

\end{definition}
We refer to $\Phi$ as a solution map and not a flow map since typically a flow map is only defined when there is an existence and uniqueness result. Our result for the DNLS does not include uniqueness, so for any $u_0$ that may produce multiple solutions to \eqref{1} we take $\Phi(u_0)$ to be the solution constructed in Theorem 1.3 of \cite{Burq} as a limit of solutions to the finite dimensional problem. We also define $\mathcal{S}_T$ to contain only these solutions, so that $\Phi^{-1}$ is a one-to-one function from $\mathcal{S}_T$ to $L^2(\mathbb{T})$.  

It is necessary to consider $\mathcal{S}_T$, and not the whole space $C([-T,T] ;  H^{\sigma}(\mathbb{T}))$ because $\Phi^{-1}$ is not well-defined on the whole space. However, $\Phi^{-1}$ is not only a  well-defined function from $\mathcal{S}_T$ to $H^{\sigma}(\mathbb{T})$ but a bounded, continuous function. This allows us to apply weak convergence to integrands containing $\Phi^{-1}$.

Also $\Phi^{-1}(\mathcal{S}_{\infty}) = \cap_{T>0} \Phi^{-1}(\mathcal{S}_{T}) $ is the set of  initial data in $H^{\sigma}(\mathbb{T})$ that produce  global solutions. We will refer to $\mathcal{S}_{T}$ and $\Phi^{-1}(\mathcal{S}_{\infty})$ throughout this final section.

\begin{definition} \label{1308}
For each $m< m'$ we define a measure $\upsilon_m$, supported on $\mathcal{S}_T$, as the pushforward of $\rho_m$ under $\Phi$. For any measurable set $E \subset \mathcal{S}_T \subset C([-T,T] ;  H^{\sigma}(\mathbb{T}))$ we let $$\upsilon_m(E) = \rho_m(\Phi^{-1}(E)).$$

For each $\epsilon >0$ we similary define $$\upsilon^{\epsilon}_{m}(E) = \rho^{\epsilon}_m(\Phi^{-1}(E)).$$
\end{definition}
By Theorem 1.3 of \cite{Burq}, $\rho_m^{\epsilon}(\Phi^{-1}(\mathcal{S}_{\infty}))=1$, implying that for each $T > 0$,  $\upsilon_m^{\epsilon}(\mathcal{S}_T)=1$. Our aim is to prove $\rho_m(\Phi^{-1}(\mathcal{S}_{\infty})) = 1$ which is equivalent to proving that $\upsilon_m(\mathcal{S}_T)=1$ for each $T >0$. For now we do not know $\upsilon_m$ is even a probability measure, and can only infer that $\upsilon_m$ has total measure $\le 1$. 

Having defined our solution map and sequence of measures on the solution space, we now prove invariance of the sequence of measures $\rho_m^{\epsilon}$.

\begin{lemma} \label{1310}
For each $m < m'$ and small $\epsilon >0$, $\rho_{m}^{\epsilon}$ is invariant with respect to $\Phi_t$ for each $t \in \mathbb{R}$.
\end{lemma}

\begin{proof} Let $\rho$ be the measure that satisfies $d\rho = \beta 1_{\|u\|_{L^2(\mathbb{T})}^2<m'}e^{f(u) - \int \frac{1}{2}|u|^6dx}  d\mu$ for a normalizing constant $\beta$. For each $t$ let $\rho_t(E) = \rho(\Phi_{-t}(E))$.  Theorem 1.3 of \cite{Burq} states that $\rho$ is an invariant measure, meaning for each $t$ and measurable set $E \subset H^{\sigma}(\mathbb{T})$, $\rho(E) = \rho(\Phi_{t}(E))$.

We also know the measure $\rho_m^{\epsilon}$ satisfies $d\rho_m^{\epsilon} = \theta_{m,\epsilon}  1_{m-\epsilon, m+\epsilon}(\|u\|_{L^2(\mathbb{T})}^2)d\rho$ for some constant $\theta_{m,\epsilon}$. Therefore for each measurable $E$ we have
\begin{equation} \label{1314}
\begin{split}
\rho_m^{\epsilon}(E) 
&= \theta_{m,\epsilon}\int_{E} 1_{m-\epsilon, m+\epsilon}(\|u\|_{L^2(\mathbb{T})}^2)  d\rho(u) \\
&= \theta_{m,\epsilon} \int_{E \cap 1_{m-\epsilon, m+\epsilon}(\|u\|_{L^2(\mathbb{T})}^2)}  d\rho(u)\\
&= \theta_{m,\epsilon}\rho(1_{m-\epsilon, m+\epsilon}(\|u\|_{L^2(\mathbb{T})}^2) \cap E) \\
&= \theta_{m,\epsilon}\rho(\Phi_t(1_{m-\epsilon, m+\epsilon}(\|u\|_{L^2(\mathbb{T})}^2) \cap E)). \\
\end{split}
\end{equation}
By conservation of mass, $\Phi_t(1_{m-\epsilon, m+\epsilon}(\|u\|_{L^2(\mathbb{T})}^2) \cap E) = 1_{m-\epsilon, m+\epsilon}(\|u\|_{L^2(\mathbb{T})}^2) \cap \Phi_t(E)$. We have
\begin{equation*}
\begin{split}
\rho_m^{\epsilon}(E) &= \theta_{m,\epsilon}\rho(1_{m-\epsilon, m+\epsilon}(\|u\|_{L^2(\mathbb{T})}^2) \cap \Phi_t(E)) \\
&= \rho_m^{\epsilon}(\Phi_t(E)).\\
\end{split}
\end{equation*}

\end{proof}

We now turn to proving well-posedness on a set of measure $1$ with respect to each measure $\rho_m$. Once we guarantee a solution exists with probability $1$, we can prove that $\rho_m$ is also invariant on any time interval $[-T,T]$. In order to apply the Prokhorov and Skohorod theorems it is necessary to show that the sequence $\upsilon_m^{\epsilon}$ is tight. We present the following result from Burq, Thomann and Tzvetkov. They do not state this lemma explicitly, but they go through the steps of proving it while constructing the invariant measures in \cite{Burq}. 

\begin{lemma} \label{1316}
Fix two values $ 0 < \sigma  < \sigma' < \frac{1}{2}$ and a time $T>0$. If $\{\lambda_N, N \ge 0\}$ is a sequence of measures on $C([-T,T] ;  H^{\sigma}(\mathbb{T}))$ such that for all $p \ge 2$ we have 
\begin{equation}
\begin{split}
\|u\|_{L_{\lambda_N}^pL^p([-T,T] ;  H^{\sigma'}(\mathbb{T}))} &\le C\\
\|u_t\|_{L_{\lambda_N}^pL^p([-T,T]; H^{\sigma'-2}(\mathbb{T}))} &\le C\\
\end{split}
\end{equation}
uniformly in $N$, then $\{\lambda_N, N \ge 0\}$ is a tight sequence of measures on $C([-T,T] ;  H^{\sigma}(\mathbb{T}))$. 
\end{lemma}

The proof can be found in Sections 3 and 4 of \cite{Burq}, specifically Proposition 4.11.

We now prove that the sequence $\upsilon_m^{\epsilon}$ satisfies the tightness condition.
\begin{proposition} \label{1317}
For each  $m< m'$, $p \ge 2$, $\sigma' < \frac{1}{2}$ and $T>0$ there exists a constant $C(m,p,T)$ such that 
\begin{equation}
\begin{split}
\|u\|_{L_{\upsilon_m^{\epsilon}}^pL^p([-T,T] ;  H^{\sigma'}(\mathbb{T}))} &\le C(m,p,T)\\
\|u_t\|_{L_{\upsilon_m^{\epsilon}}^pL^p([-T,T]; H^{\sigma'-2}(\mathbb{T}))} &\le C(m,p,T),\\
\end{split}
\end{equation}
implying that $\upsilon_m^{\epsilon}$ is a tight sequence of measures on $\mathcal{S}_T$ for $\sigma < \frac{1}{2}$. 
\end{proposition}

\begin{proof}
Expanding yields
\begin{equation}
\begin{split}
\|u\|_{L_{\upsilon_m^{\epsilon}}^pL^p([-T,T] ;  H^{\sigma}(\mathbb{T}))} =  \left( \int_{C([-T,T] ;  H^{\sigma}(\mathbb{T}))} \int_{-T}^{T} \|u(t,x)\|_{H^{\sigma}(\mathbb{T})}^p dt d\upsilon_m^{\epsilon}(u) \right)^{1/p} .
\end{split}
\end{equation}
By  invariance of the measure $\rho_m^{\epsilon}$ and Fubini's theorem , we have 
\begin{equation}
\begin{split}
\int_{C([-T,T] ;  H^{\sigma}(\mathbb{T}))} \int_{-T}^{T} &\|u(t,x)\|_{H^{\sigma}(\mathbb{T})}^p dt d\upsilon_m^{\epsilon}(u)   
\\
&= \int_{H^{\sigma}(\mathbb{T})}\int_{-T}^{T} \|\Phi_t u_0\|^p_{H^{\sigma}(\mathbb{T})} dtd\rho_m^{\epsilon} \\
&= \int_{-T}^{T} \int_{H^{\sigma}(\mathbb{T})} \|\Phi_t u_0\|^p_{H^{\sigma}(\mathbb{T})} d\rho_m^{\epsilon} dt\\
&= \int_{-T}^{T} \int_{H^{\sigma}(\mathbb{T})} \|u_0\|^p_{H^{\sigma}(\mathbb{T})} d\rho_m^{\epsilon} dt\\
&= 2T  \int_{H^{\sigma}(\mathbb{T})}\left( \sum_{n \in \mathbb{Z}}\frac{|g_n|^2}{\langle n \rangle^{2-2\sigma}} \right)^{p/2} d\rho_m^{\epsilon} \\
&= 2T \int_{H^{\sigma}(\mathbb{T})}\left( \sum_{n \in \mathbb{Z}}\frac{|g_n|^2}{\langle n \rangle^{2-2\sigma}} \right)^{p/2} \beta_{m,\epsilon}e^{f(u)- \frac{1}{2}\int_{\mathbb{T}} |u|^6 dx}d\mu_m^{\epsilon} \\
&\le 2T  \left( \int_{H^{\sigma}(\mathbb{T})} \left( \sum_{n \in \mathbb{Z}}\frac{|g_n|^2}{\langle n \rangle^{2-2\sigma}} \right)^{p} d\mu_m^{\epsilon}\right)^{1/2} \left( \int_{H^{\sigma}(\mathbb{T})} e^{2f(u)- \int_{\mathbb{T}} |u|^6 dx}d\mu_m^{\epsilon} \right)^{1/2}.\\
\end{split}
\end{equation}
By Lemma \ref{819} and Corollary \ref{1012}, this is $\le 2TC(m,p) \le C(m,p,T)$. This completes the first part of the proposition.

Now we prove the second inequality. We know that in the support of $\upsilon_m^{\epsilon}$, 
\begin{equation*}
u_t = iu_{xx} + \partial_x(|u|^2u)
\end{equation*}
 as an element of $C([-T,T] ;  H^{\sigma-2}(\mathbb{T}))$. The $iu_{xx}$ term is bounded in the $H^{\sigma-2}(\mathbb{T})$ norm by the above argument. It suffices to bound $\|\partial_x(|u|^2u)\|_{H^{\sigma-2}(\mathbb{T})}$.

By the definition of Sobolev spaces and Sobolev embedding we have $$\|\partial_x(|u|^2u)\|_{H^{\sigma-2}(\mathbb{T})} \lesssim \||u|^2u\|_{H^{\sigma-1}(\mathbb{T})} \lesssim \|u^3\|_{L^2(\mathbb{T})} =  \|u\|_{L^6(\mathbb{T})}^3 \lesssim \|u\|^3_{H^{1/3}(\mathbb{T})}.$$ Therefore we have reduced the problem to bounding $\|u\|_{L_{\upsilon_m^{\epsilon}}^{3p}L^{3p}([-T,T]; H^{1/3}(\mathbb{T}))}^3$ by $C(m,p,T)$, which we already proved in the first part. 

Therefore $$\|u_t\|_{L_{\upsilon_m^{\epsilon}}^pL^p([-T,T];H^{\sigma-2}(\mathbb{T}))} \le C(m,p,T).$$ \end{proof}

Having proven tightness of the sequence $\upsilon_m^{\epsilon}$, we can now apply the Prokhorov and Skohorod theorems to construct the weak limit $\upsilon^m$ and a set of solutions to equation \eqref{1} with distribution $\upsilon^m$. We also conclude that $\upsilon^m$ is indeed supported on $\mathcal{S}_T$.

\begin{theorem} \label{1320}
For each $m < m'$ and $T>0$ there exists a sequence $w_k$ of random variables and a random variable $w$,  all taking values in $C([-T,T] ;  H^{\sigma}(\mathbb{T}))$, and a probability measure $\upsilon^m$ supported on $\mathcal{S}_T$ such that $\mathcal{L}(w_k) = \upsilon_m^{\epsilon_k}, \mathcal{L}(w) = \upsilon^m$ and $w_k \rightarrow w$ in $C([-T,T] ;  H^{\sigma}(\mathbb{T}))$. In addition, for every bounded continuous function $F : H^{\sigma}(\mathbb{T}) \rightarrow \mathbb{R}$ we have
\begin{equation} \label{1321}
\int_{\mathcal{S}_T}  F(\Phi^{-1}(u)) d\upsilon^m(u)  = \int_{\mathcal{S}_T}  F(\Phi^{-1}(u)) d\upsilon_m(u).
\end{equation}

\end{theorem}
\begin{proof}
 By the Prokhorov and Skorokhod theorems there exists a measure $\upsilon^m$ and a subsequence $\upsilon_{m}^{k}$ such that $\upsilon_{m}^{k} \rightarrow \upsilon^m$ weakly. In addition, there exists a sequence $w_{k}$ and a limit $w$ of random variables taking values in $C([-T,T] ;  H^{\sigma}(\mathbb{T}))$ such that $\mathcal{L}(w_k) = \upsilon_m^{\epsilon_k}, \mathcal{L}(w) = \upsilon^m$ and $w_k \rightarrow w$ almost surely in $C([-T,T] ;  H^{\sigma}(\mathbb{T}))$. 

Now we show $\upsilon^m$ is supported on $\mathcal{S}_T$. 
By Theorem 1.3 of \cite{Burq}, the DNLS equation is almost surely well-posed with respect to each measure $\rho_m^{\epsilon}$. Therefore for every $k$ the variable $w_k$ almost surely satisfies equation \eqref{1} on a set of full measure. We will exploit the almost sure convergence $w_k \rightarrow w$ in $ C([-T,T] ;  H^{\sigma}(\mathbb{T}))$ to verify that $w$ is a solution to equation \eqref{1} as well. 

For each $k$, $w_k$ satisfies 
\begin{equation} \label{1322}
 i\partial_t w_k + \Delta w_k = i\partial_x(|w_k|^2w_k).
\end{equation}

Clearly the linear terms converge in $C([-T,T] ;  H^{\sigma-2}(\mathbb{T}))$, we have $i\partial_t w_k \rightarrow i\partial_t w$ and $\Delta w_k \rightarrow \Delta w$. In addition, 
\begin{equation} \label{1323}
\begin{split}
\lim_{k \rightarrow \infty}\|i\partial_x(|w_k|^2w_k) -  i\partial_x(|w|^2w)   &\|_{C([-T,T] ;  H^{\sigma-2}(\mathbb{T}))} \\
&=  \lim_{k \rightarrow \infty} \sup_{t \in [-T,T]} \|i\partial_x(|w_k|^2w_k) - i\partial_x(|w|^2w)   \|_{H^{\sigma-2}(\mathbb{T})}\\
&\le \lim_{k \rightarrow \infty}  \sup_{t \in [-T,T]} \||w_k|^2w_k - |w|^2w\|_{L^2(\mathbb{T})}\\
&\lesssim \lim_{k \rightarrow \infty}  \sup_{t \in [-T,T]} \|w_k -w\|_{L^6(\mathbb{T})} (\|w\|^2_{L^6(\mathbb{T})} + \|w_k\|^2_{L^6(\mathbb{T})})\\
&\lesssim  \lim_{k \rightarrow \infty}  \sup_{t \in [-T,T]} \|w_k -w\|_{H^{1/3}(\mathbb{T})} (\|w\|^2_{H^{1/3}(\mathbb{T})} + \|w_k\|^2_{H^{1/3}(\mathbb{T})})\\
& =0,\\
\end{split}
\end{equation}
by convergence of $w_k \rightarrow w$ in $C([-T,T] ;  H^{\sigma}(\mathbb{T}))$. Therefore $w$ is a solution to the DNLS equation on $[-T,T]$ on a set of $\upsilon^m$ measure $1$. So $\upsilon^m$ is supported on $\mathcal{S}_T$. From now on $\upsilon^m$ will refer to the measure restricted to this set.

It remains to prove equation \eqref{1321}, an important step in proving that $\upsilon^m = \upsilon_m$.
Consider a bounded continuous function $F(u_0)$ on $H^{\sigma}(\mathbb{T})$ and let $F^* = F \circ \Phi^{-1}$. By the substitution rule, and the fact that $\Phi$ is defined $\rho_m^{\epsilon}$ a.e., we have

\begin{equation} \label{1324}
\begin{split}
\int_{\mathcal{S}_T} F^*(u) d\upsilon_m^{\epsilon_k}(u) &= \int_{u_0 \in H^{\sigma}(\mathbb{T})} F^*(\Phi(u_0)) d\rho_m^{\epsilon}(u_0) \\
\int_{\mathcal{S}_T} F(\Phi^{-1}(u)) d\upsilon_m^{\epsilon_k}(u) &= \int_{u_0 \in H^{\sigma}(\mathbb{T})} F(\Phi^{-1} \circ \Phi(u_0)) d\rho_m^{\epsilon}(u_0) \\
&= \int_{u_0 \in H^{\sigma}(\mathbb{T})} F(u_0)d\rho_m^{\epsilon}(u_0). \\
\end{split}
\end{equation}

Noting that $\Phi^{-1}$ is a continuous function from $\mathcal{S}_T \subset C([-T,T] ;  H^{\sigma}(\mathbb{T}))$ to $H^{\sigma}(\mathbb{T})$ with $\|\Phi^{-1}(u)\|_{H^{\sigma}(\mathbb{T})} \le \|u\|_{C([-T,T] ; H^{\sigma}(\mathbb{T}))}$, we have  
\begin{equation}
\begin{split}
\lim_{\epsilon_k \rightarrow 0} \int_{u_0 \in H^{\sigma}(\mathbb{T})}  F(u_0) d\rho^{\epsilon_k}_{m}(u_0) &= \lim_{\epsilon_k \rightarrow 0}\int_{\mathcal{S}_T} F(\Phi^{-1}(u))d\upsilon^{\epsilon_k}_{m}(u) \\
&= \int_{\mathcal{S}_T}  F(\Phi^{-1}(u)) d\upsilon^m(u).\\
\end{split}
\end{equation}
By weak convergence of $\rho^{\epsilon}_{m} \rightarrow \rho_m$ this limit equals 
\begin{equation} \label{1333}
\int_{u_0 \in H^{\sigma}(\mathbb{T})}  F(u_0) d\rho_{m}(u_0) = \int_{\mathcal{S}_T}  F(\Phi^{-1}(u)) d\upsilon_m(u).
\end{equation}
 Therefore for each continuous, bounded function $F$, $$\int_{\mathcal{S}_T}  F(\Phi^{-1}(u)) d\upsilon^m(u)  = \int_{\mathcal{S}_T}  F(\Phi^{-1}(u)) d\upsilon_m(u).$$
\end{proof}

We can now prove our final result on $\rho_m$, Theorem \ref{17}, which we restate for convenience.
\textbf{Theorem \ref{17}.}
\textit{There exists a small constant $m'$ such that for each $m < m'$ there exists a measure $\rho_m$ supported on the set $H^{1/2-}({\mathbb{T}}) \cap \{ \|u\|_{L^2}^2 = m\}$. There also exists a subset $\Sigma \subset H^{1/2-}(\mathbb{T}) \cap \{ \|u\|_{L^2}^2 = m\}$ of full $\rho_m$ measure such that each $u_0 \in \Sigma$ produces a global solution $u$ to \eqref{1}. The measure $\rho_m$ is invariant, meaning the random variable $u(t)$ has distribution $\rho_m$ for all $t \in \mathbb{R}$.
}
\begin{proof}[Proof of Theorem \ref{17}] 
Fix a sufficiently value of $m <m'$. We know  that $\upsilon^m(\mathcal{S}_T)=1$ for each $T \in (0,\infty]$. We must prove three results: 
\begin{enumerate}
\item{} For each $T>0$, $\rho_m(\Phi^{-1}(\mathcal{S}_{T}))  =\upsilon_m(\mathcal{S}_{T}) =1$, thus $\rho_m$ a.e. initial data produces a global solution. 

\item{} For each $T>0$ the measures $\upsilon^m$ and $\upsilon_m$ are equal. 

\item{} The measure $\rho_m$ is invariant with respect to $\Phi_t$ for all $t \in \mathbb{R}$. 
\end{enumerate}

We prove part $(1)$ for each $T>0$ and take a countable union of such sets as $T \rightarrow \infty$. This proves that $\upsilon_m(\mathcal{S}_{\infty})  =\upsilon^m(\mathcal{S}_{\infty})  =1$, and $\rho_m$ a.e. initial data produces a global solution. 

Theorem \ref{1320} states that for every bounded, continuous $F$ we have 
\begin{equation} \label{13210}
\int_{\mathcal{S}_T}  F(\Phi^{-1}(u)) d\upsilon^m(u)  = \int_{\mathcal{S}_T}  F(\Phi^{-1}(u)) d\upsilon_m(u).
\end{equation}
Setting $F=1$ tells us that $\upsilon_m(\mathcal{S}_T) = \upsilon^m(\mathcal{S}_{T})  =1$. Therefore $\rho_m(\Phi^{-1}(\mathcal{S}_{T}))=1$ for each $T>0$. This proves part $(1)$.

Now we move on to part $(2)$. Since $\rho_m(\Phi^{-1}(\mathcal{S}_{\infty}))=1$ we can conclude that for every $E \subset H^{\sigma}(\mathbb{T})$ we have 
$\rho_m(E) = \rho_m(\Phi^{-1}\circ \Phi(E)) = \upsilon_m(\Phi(E))$. Without our well-posedness result it was not clear that the sets $E$ and $\Phi^{-1} \circ \Phi(E)$  had the same $\rho_m$-measure.  We similarly define a new measure $\rho^m$ on $\Phi^{-1}(\mathcal{S}_T)$ as the pullback of $\upsilon^m$ under $\Phi$, $\rho^m(E) = \upsilon^m(\Phi(E))$. This implies that for each continuous function $F: H^{\sigma}(\mathbb{T}) \rightarrow \mathbb{R}$ we have 
\begin{equation} \label{1325}
\int_{H^{\sigma}(\mathbb{T})}  F(u_0)d\rho^{m}(u_0) = \int_{H^{\sigma}(\mathbb{T})}  F(\Phi^{-1}(u))d\upsilon^m(u).
\end{equation}
Combining equation \eqref{1325} with equation \eqref{13210} and \eqref{1333} we have
\begin{equation} \label{1340}
 \int_{H^{\sigma}(\mathbb{T})}  F(u_0)d\rho^{m}(u_0) = \int_{H^{\sigma}(\mathbb{T})}  F(u_0)d\rho_{m}(u_0), \\
\end{equation}
for any bounded continuous function $F$. Therefore $\rho_m = \rho^m$. 

By definition, $\upsilon_m$ is the pushforward of $\rho_m$, so if we can prove that $\upsilon^m$ is also the pushforward of $\rho^m$  then we conclude that $\upsilon_m = \upsilon^m$. Indeed, $\rho^m(E) = \upsilon^m(\Phi(E))$ for all $E$, and since $\upsilon^m$ is supported on $\mathcal{S}_T$, for any set $E' \subset \mathcal{S}_T$, we have $E' = \Phi(E)$ for some $E \subset H^{\sigma}(\mathbb{T})$. Therefore $$\upsilon^m(E') = \upsilon^m(\Phi(E) ) = \rho^m(E)  = \rho_m(E) =\upsilon_m(\Phi(E)) = \upsilon_m(E').$$

Since $\upsilon_m$ and $\upsilon^m$ are both supported on $\mathcal{S}_T$, this proves the two measures are equal.

We finish by proving part $(3)$, that $\rho_m$ is invariant with respect to $\Phi_t$ for each $t \in \mathbb{R}$. Consider a value of $t$ and an interval $[-T,T]$ containing $t$. By the invariance of $\rho_m^{\epsilon}$, for any bounded continuous $F$ we have 

\begin{equation} \label{1326}
\begin{split}
\int_{H^{\sigma}(\mathbb{T})}  F(u_0)d\rho_{m}^{\epsilon_k}(u_0) &= \int_{H^{\sigma}(\mathbb{T})}  F(u_0)d\rho_m^{\epsilon_k}(\Phi_{t}(u_0))\\
\int_{H^{\sigma}(\mathbb{T})}  F(u_0)d\rho_{m}^{\epsilon_k}(u_0)  &= \int_{H^{\sigma}(\mathbb{T})}  F(\Phi_{-t}(u_0))d\rho_m^{\epsilon_k}(u_0)\\
\int_{C([-T,T] ;  H^{\sigma}(\mathbb{T}))}  F(\Phi^{-1}(u))d\upsilon_{m}^{\epsilon_k}(u) &= \int_{C([-T,T] ;  H^{\sigma}(\mathbb{T}))}  F(\Phi_{-t}\circ \Phi^{-1}(u))d\upsilon_{m}^{\epsilon_k}(u).\\
\end{split}
\end{equation}

Note that even though $\Phi_t$ is not continuous, $\Phi^{-1}$ is clearly a bounded continuous function from $C([-T,T] ;  H^{\sigma}(\mathbb{T}))$ to $H^{\sigma}(\mathbb{T})$, as is $\Phi_{-t} \circ \Phi^{-1}$. Therefore we are integrating a bounded continuous function in each integral and can apply weak convergence as $k \rightarrow \infty$ to obtain 
\begin{equation} \label{1327}
\begin{split}
\lim_{k \rightarrow \infty} \int_{C([-T,T] ;  H^{\sigma}(\mathbb{T}))}  F(\Phi^{-1}(u))d\upsilon_{m}^{\epsilon_k}(u) &= \lim_{k \rightarrow \infty} \int_{C([-T,T] ;  H^{\sigma}(\mathbb{T}))}  F(\Phi_{-t} \circ \Phi^{-1}(u))d\upsilon_{m}^{\epsilon_k}(u) \\
\int_{C([-T,T] ;  H^{\sigma}(\mathbb{T}))}  F(\Phi^{-1}(u))d\upsilon_{m}(u) &= \int_{C([-T,T] ;  H^{\sigma}(\mathbb{T}))}  F(\Phi_{-t} \circ \Phi^{-1}(u))d\upsilon_{m}(u)\\
\int_{H^{\sigma}(\mathbb{T})}  F(u_0)d\rho_{m}(u_0) &= \int_{H^{\sigma}(\mathbb{T})}  F(\Phi_{-t}(u_0))d\rho_{m}(u_0) \\
\int_{H^{\sigma}(\mathbb{T})}  F(u_0)d\rho_{m}(u_0) &=  \int_{H^{\sigma}(\mathbb{T})}  F(u_0)d\rho_{m}(\Phi_t(u_0)). \\
\end{split}
\end{equation}

Therefore the distribution of $\rho_m$ is invariant with respect to $\Phi_t$.
\end{proof}
\bibliographystyle{plain}

\end{document}